\documentclass[a4paper, 10pt]{amsproc}
\usepackage{defs}
\usepackage{orcidlink}
\usepackage{calc}
\usepackage[dvipsnames]{xcolor}
\usepackage{tabularray}
\definecolor{metablue}{HTML}{0064E0}
\definecolor{metafg}{HTML}{1C2B33}
\definecolor{metabg}{HTML}{F1F4F7}
\definecolor{metabgdeep}{HTML}{D9EFFF}
\definecolor{metagreen}{HTML}{EAFFE8}
\definecolor{metagreen}{HTML}{FCFFEE}
\definecolor{metared}{HTML}{FFEAE8}

\DeclareSymbolFont{extraup}{U}{zavm}{m}{n}
\DeclareMathSymbol{\varheart}{\mathalpha}{extraup}{86}
\DeclareMathSymbol{\vardiamond}{\mathalpha}{extraup}{87}
\DeclareMathSymbol{\varclub}{\mathalpha}{extraup}{84}
\DeclareMathSymbol{\vardspade}{\mathalpha}{extraup}{85}

\newcommand{\blue}[1]{{\color{black} #1}}

\usepackage[framemethod=tikz,xcolor=true]{mdframed}

\newmdenv[backgroundcolor=metabgdeep, roundcorner=10pt, skipabove=4pt, linewidth=0pt, innertopmargin=4pt]{myframe}

\newmdenv[backgroundcolor=metabg, roundcorner=10pt, skipabove=4pt, linewidth=0pt, innertopmargin=4pt]{myOCP}

\newmdenv[backgroundcolor=metared, roundcorner=10pt, skipabove=7pt, linewidth=0pt, innertopmargin=7pt]{myalgo}
\newmdenv[%
    leftmargin=0.5cm,
    backgroundcolor=yellow!10,%
    roundcorner=5pt,%
    tikzsetting={draw=red, line width=2.0pt}%
    ]{SpecialText}%

\title[OptStop]{Nonlinear Stochastic Optimal Control and Optimal Stopping using the Fokker--Planck Transformation}

\author[A. Selim]{Akan Selim\,\orcidlink{0000-0002-6505-5796}}
\author[S. Ganguly]{Siddhartha Ganguly\,\orcidlink{0000-0003-2046-2061}}
\author[A. Pakniyat]{Ali Pakniyat\,\orcidlink{0000-0003-0665-0330}}
\author[P. Tsiotras]{Panagiotis Tsiotras\,\orcidlink{0000-0001-7563-4129}}

\thanks{%
	A. Selim, S. Ganguly and P. Tsiotras are with \faGroup\ Daniel Guggenheim School of Aerospace, \faUniversity\ Georgia Institute of Technology, \faMapMarker\  Atlanta, USA. A. Pakniyat is with \faGroup\ Department of Mechanical Engineering, \faUniversity\ University of Alabama, \faMapMarker\ Tuscaloosa, USA}

\thanks{%
	Contact Information: (AS) \faEnvelope\ \texttt{aselim8@gatech.edu}, (SG) \faHome\ \url{https://sites.google.com/view/siddhartha-ganguly}, \faEnvelope\ \texttt{sganguly41@gatech.edu}, (AP): \faHome\ \url{https://pakniyat.github.io/}, \faEnvelope\ \texttt{apakniyat@ua.edu}, (PT) \faHome\ \url{https://dcsl.gatech.edu/tsiotras.html}, \faEnvelope\ \texttt{tsiotras@gatech.edu}.
}

\begin{document}

\maketitle

\begin{abstract}
In this paper, we develop a theoretical framework for nonlinear stochastic optimal control problems with optimal stopping by establishing a density-based deterministic representation of the underlying diffusion. For state-independent diffusion, we rewrite the controlled Fokker–Planck equation as a continuity equation driven by a score-corrected velocity field, yielding a deterministic characteristic dynamics that reproduces the marginal law of the stochastic system. Leveraging Stein-type identities, we show that the associated distributional dynamic programming equation admits the same second-order differential operator as the distributional stochastic Hamilton–Jacobi–Bellman formulation. Building on this representation, we formulate an optimal control problem with state-dependent terminal-time assignment and terminal distributional constraints and derive the first-order necessary conditions using variational analysis. We present the conditions both for a common terminal time and for the general case of state-dependent stopping.
\end{abstract}

\section{Introduction}\label{sec:intro}

In the vast majority of stochastic optimal control problems, the objective is not merely to guide a single state trajectory, but to shape the \emph{entire distribution} of possible system states to a terminal manifold.
This perspective is natural in applications such as robotic swarms \cite{intro1}, where one seeks to steer a population of agents toward a desired spatial distribution, and in aerospace guidance \cite{intro2,intro3}, where uncertainty in sensing, actuation, or the environment makes distribution-level guarantees more meaningful than nominal trajectory tracking alone. 
Similar questions also arise in finance \cite{intro4,intro5}, where one may wish to control the law of a stochastic process while deciding the appropriate time to stop, liquidate, or intervene. 
Optimal stopping has also been studied in related game-theoretic settings through non-anticipative stopping strategies; see, for instance, \cite{intro6}.
In such problems, one must determine \emph{both how to control the dynamics and when to stop the process}. This coupling is fundamental: continuing for too long may increase risk, cost, or exposure to constraints, whereas stopping too early may prevent the system from reaching the desired terminal distribution. 
As a result, the problem is substantially more challenging than standard fixed-horizon distribution steering, since it combines stochastic dynamics, evolution of the full state law, and an endogenous stopping mechanism. These features make direct trajectory-based analysis difficult and call for a framework that can treat control, stopping, and distribution evolution in a unified manner. Motivated by this, in this paper we develop a \emph{density-based approach}, based on a \emph{Fokker–Planck transformation}, for nonlinear stochastic optimal control problems with optimal stopping and terminal distributional constraints.

\textbf{On existing literature:} Significant progress has been made for the fixed final time problem, particularly through solutions to the classical Schr\"{o}dinger Bridge Problem (SBP) \cite{sch1,sch2,ref:teter2025probabilistic} via the theory of optimal transportation \cite{book_ge3,ref:YC:TG:MP:SB-and-OT}. 
Beyond prescribed endpoint distributions, the Schr\"{o}dinger bridge framework has also been extended to incorporate constraints on ensemble averages of path integrals, connecting distribution steering with maximum-caliber inference \cite{ref:OMM:AE:TTG:Caliber}.
However, less is known about the distribution steering problem with free final time \cite{ref:AS:SG:FL:PT:FFTAut}, which is closely related to the problem of \emph{optimal stopping}. The literature on optimal stopping is extensive and rich: the \emph{Skorokhod Embedding Problem} (SEP) \cite{SEP1,SEP-eng} addresses whether a stopping time can transform a Brownian motion into a random variable with a given target distribution. 
Skorokhod linked the problem of finding the optimal stopping times to the Shiryaev Inverse Problem (SIP) \cite{SEP2}, which, conversely, determines stopping boundaries for stochastic processes. 
In fact, the first hitting time of the stopping boundary is exactly the solution to the SEP; see \cite{SEP29} for further details.
The seminal work \cite{SEP3} proposed using barriers in the combined state-time space to define stopping times. The ideas established in \cite{SEP3} connected SEP to variational inequalities and free-boundary problems \cite{SEP4,SEP5}. \cite{SEP6} extended SEP to Markov processes using potential theory and proved Kiefer's conjecture \cite{Kiefer}, namely, that Root's stopping time 
embedding~\cite{SEP3} minimizes the variance of the stopping time \cite{SEP7}. Reference~\cite{SEP9} provided an alternative solution to SEP for any recurrent real-valued diffusion process; the authors employed Martingale theory to derive stopping times that maximize the law of the supremum of the stopped process. Recent studies extended their research to the drifting Brownian motion \cite{SEP13,SEP14}. Multivariate generalizations of these results have also been studied; see, for instance, \cite{SEP21}. 

A technique for constructing stopping times that minimize the distribution of the maximum of a stochastic process, while simultaneously maximizing the distribution of the minimum was established in \cite{SEP11}. This approach introduced a duality between the Azéma-Yor \cite{SEP9} and Perkins embeddings, where the former focuses on maximizing the law of the maximum of the process and the latter on minimizing it. 
Perkins’ duality \cite{SEP11} opened new avenues for constructing embeddings with different optimality criteria. His methods were later extended to general SEP frameworks in \cite{SEP12}. Furthermore, \cite{SEP17} proposed an optimal SEP formulation focused on cost minimization for steering a probability measure from one distribution to another in an optimal manner. More recent advances include multi-dimensional extensions, PDE-based free-boundary methods \cite{SEP18}, and neural network techniques for approximating stopping times \cite{numeric2,numeric3}. 
PDE approaches, in particular, have enabled the characterization of free boundaries in stochastic differential equations (SDEs), yielding solutions to complex embedding problems \cite{SEP5}. On the other hand, research on duality and discretization methods \cite{SEP19,SEP25} has investigated how duality principles with the Stochastic Maximum Principle (SMP) can be applied to optimal stopping problems, particularly those involving barrier-type stopping times like the Root and Rost embeddings \cite{SEP3, SEP7}.

\blue{Related extensions of the Schr\"{o}dinger bridge framework address \emph{stopping-time distributions} and changes in total mass. 
The authors in~\cite{ref:AE:TTG:SB-timeMarg} study random walks on directed graphs and construct \emph{Markov policies} that realize prescribed first-arrival-time distributions at fixed absorbing states. For continuous-state diffusions, \cite{killing_tg1} jointly modify the drift and killing rate to match prescribed spatio-temporal distributions of killed particles, deriving a coupled Schr\"{o}dinger system and a Fortet–Sinkhorn-type algorithm. Unbalanced endpoint marginals are treated in \cite{killing_tg2}, while maximum-entropy formulations for nonlinear reaction–diffusion systems with density-dependent creation and annihilation rates are developed in \cite{ref:OMM:AE:TTG:ReacDiff}.
Our focus is on continuous-state controlled diffusions whose stopping boundary is determined jointly with the feedback control, subject to terminal distributional constraints. We develop a deterministic representation of the density evolution and derive first-order necessary conditions for this joint optimization under the stated regularity assumptions.}


We note that \cite{ref:AP:PP:ACC:MinPrnciple:Steering} considered filtration-adapted stochastic state steering via stochastic minimum-principle methods, enforcing terminal constraints on the conditional expectation of the state as information evolves in time; extensions to partially observed linear systems \cite{ref:AP:PS:CDC:PartObs:Steering} have also been obtained. 
While these results capture an important notion of information-adapted terminal steering, they remain fundamentally different from the present setting, since they impose terminal constraints at the level of state expectation rather than the full terminal law, and do not incorporate optimal stopping. 
Although convex duality provides a pathway for assigning full probability distributions at fixed terminal times \cite{pakniyat2022convex, pakniyat2025graphon}, this formulation comes at the cost of requiring an optimization over a class of Hamilton-Jacobi type partial differential equations. Related dynamic-programming formulations also appear in stochastic reach-avoid and motion-planning problems for diffusions, where a form of PDE characterizations have been developed for exit-time and sequential stopping objectives with discontinuous payoffs \cite{ref:PM:DC:JL:TAC:Motion:Planning,ref:PM:DC:JL:Aut:ReachAvoid}. 
These works, however, focus on state-space reachability and path-planning specifications rather than terminal distributional constraints and distribution-level value functions.

Despite the previous developments, general $n$-dimensional stochastic optimal control and stopping problems, including optimal common stopping, with distributional constraints remain largely unexplored. To address this gap, with the long-term goal of developing a scalable sampling-based method, we establish a novel approach for finding optimal stopping strategies for nonlinear stochastic optimal control problems subject to distributional constraints. 
The primary focus of this article is on \highlight{theoretical developments}, with an eye toward the imminent development of tractable \emph{algorithmic architectures}.

\subsection*{Contributions}
Against this backdrop, our contributions are highlighted below: 
\begin{enumerate}[label={\textup{(C-\alph*)}}, leftmargin=*, widest=b, align=left]
\item \label{intro:contrib:1} \textbf{A distributional framework for stochastic optimal stopping:} We formulate a nonlinear stochastic optimal control problem (OCP) with \emph{optimal stopping} and \emph{distributional objectives} in a measure-theoretic framework that explicitly captures the effect of stopping on the evolving state distribution. This allows us to represent the control problem directly at the distribution level, define both pointwise and distributional value functions, and derive the associated \emph{distributional stochastic Hamilton–Jacobi–Bellman equation} together with a decomposition result linking the distributional value to the expectation of the pointwise value. 

\item \label{intro:contrib:2} \textbf{A deterministic reformulation via the Fokker–Planck transformation:} For the case of state-independent diffusion, we introduce a \emph{Fokker–Planck transformation} that rewrites the stochastic density evolution as a continuity equation with a score-corrected velocity field. This yields a deterministic characteristic description whose evolving marginal law matches that of the original stochastic system, thereby providing a density-based deterministic representation of the diffusion process. This is a crucial idea that enables further analysis. In addition, we characterize also the corresponding \emph{reverse Fokker-Planck transformation}, showing how It\^o's formula \cite[Theorem 4.2.1]{book1} is recovered from the deterministic formulation.

\item \label{intro:contrib:3} \textbf{Equivalence of the reformulated and original problems:} Using Stein-type identities, we show that the transformed deterministic formulation retains the same distributional dynamic programming structure as the original stochastic problem. In particular, the associated distributional value function satisfies an equation involving the same second-order operator as the distributional stochastic Hamilton–Jacobi–Bellman equation, which establishes the equivalence of the two formulations at the level of value-function characterization.

\item \label{intro:contrib:4} \textbf{First-order necessary conditions:} Building on this reformulation, we formulate an optimal control problem with state-dependent terminal-time assignment and terminal distributional constraints, and analyze it by variational methods in the space of probability measures. This yields a (Pontryagin-type) first-order optimality system, applicable to both the common terminal-time case and the more general free/state-dependent stopping case, comprising control stationarity, adjoint dynamics with diffusion-induced correction terms, and appropriate transversality conditions. 
\end{enumerate}

\blue{\textbf{Summary of assumptions:} Assumptions \ref{regularity_dyn} and \ref{regularity_cost} impose standard regularity and global bounds on the dynamics and objectives. The absolute-continuity hypotheses in Sections~\ref{subsec:prob:formulation}--\ref{subs:subprob} permit the density representation; the finite-second-moment requirement in Section~\ref{subsec:prob:formulation} and the vanishing-tail conditions in Remark~\ref{remark_abs} are imposed to eliminate boundary terms at infinity. Assumption \ref{regularity_val} assumes a viscosity characterization without requiring classical differentiability of the value functions; this is standard in stochastic optimal control; see Remark \ref{rem:van:tail} for a discussion.
Assumptions \ref{ass_reg_tau} and \ref{Assumption_5} concern the stopping geometry and the regularity of the alive density, \eqref{mu_nec_cond}--\eqref{mu_suff_cond} record the characteristic-flow requirements, and Assumption \ref{assump_reg_lm} is about the regularity of the multiplier. 
These properties are required for the variational analysis, and they are problem-dependent. 
}

\blue{The rest of the article is organized as follows: \S\ref{sec:prob:form} collects the required preliminaries and the conditional expectation framework used throughout. In \S\ref{sec:problem}, we formulate the central problem, introduce the pointwise and distributional value functions, discuss the associated stochastic Hamilton–Jacobi–Bellman equation, and develop the alive/sub-probability measure framework for stopped processes. \S\ref{sec:main_result} contains the main results of the paper, mainly, the associated distributional sHJB (stochastic Hamilton-Jacobi-Bellman) equation, the Fokker--Planck transformation, and the first-order necessary optimality conditions.}

\section{Some preliminaries and background}\label{sec:prob:form}
We begin by listing the notation used throughout the article; our notation is standard unless stated otherwise.

\subsection*{Notation} \label{notation_sec}
We let $\N \Let \aset{1,2,\ldots}$ denote the set of positive integers, and let $\Z$ be the set of integers. 
Let $n \in \N$; we will denote the standard Euclidean vector space by $\Ree^n$, which is assumed to be equipped with standard norm \(\Ree^n \ni x\mapsto \norm{x}\). 
By \(\R[n]_{\ge 0}\) we denote the set of all \(n\)-dimensional real vectors whose components are all nonnegative. $\Borelsigalg(\Ree^n)$ denotes the standard Borel $\sigma$-algebra on $\Ree^n$, i.e., the smallest $\sigma$-algebra generated by the open sets in $\Ree^n$, and $\polish(\Ree^n)$ denotes the set of all probability measures on the measurable space 
$\bigl(\Ree^n, \Borelsigalg(\Ree^n)\bigr)$. $\polish_{\rm sub}(\Ree^n)$ denotes the set of all nonnegative Borel sub-probability measures on the same measurable space, i.e.,
\begin{align*}
\polish_{\rm sub}(\Ree^n) \Let \aset[\big]{\mu \suchthat \mu \text{ is a nonnegative Borel measure on } \Ree^n \text{ and } \mu(\Ree^n)\le 1 }. \nn
\end{align*}

For two measures $\mu, \nu$ on the same measurable space $(\Ree^n,\Borelsigalg(\Ree^n))$, we write $\mu \ll \nu$ if $\mu$ is absolutely continuous with respect to $\nu$, i.e., if $\nu(A)=0$ implies $\mu(A)=0$ for all $A \in \Borelsigalg(\Ree^n)$ (see \cite[p.~164]{book1}).
We denote by $\mathrm{Leb}^n$ the $n$-dimensional Lebesgue measure on $(\Ree^n,\Borelsigalg(\Ree^n))$. 
For any measurable set
$D\subseteq \Ree^n$ we write $\mathrm{Leb}^n\!\upharpoonright_{D}$, for the restriction of $\mathrm{Leb}^n$ to $D$, i.e.,
$(\mathrm{Leb}^n\!\upharpoonright_{D})(A)\Let \mathrm{Leb}^n(A\cap D)$ for all $A\in\Borelsigalg(\Ree^n)$.

We equip $\polish(\Ree^n)$ with the topology of the weak convergence of measures, namely, a sequence $\{\nu_k \}_{k \ge 1}$ of probability measures in $\polish(\Ree^n)$ converges weakly \cite[Def. 6.8]{villani_ot} to $\nu\in\polish(\Ree^n)$ if and only if, for every real-valued bounded continuous function $\phi$ on $\Ree^n$
\begin{align}
 \int_{\Ree^n} \phi(x) \d \nu_k(x) \lra \int_{\Ree^n} \phi(x) \d \nu(x),\nn    
\end{align}
when $k\ra +\infty$. For any $p\ge 1$, we denote by $\polish_p(\Ree^n)$ the subspace of $\polish(\Ree^n)$ consisting of probability measures with finite moment of order $p$. If, in addition, the sequence of probability measures $\{\nu_k\}_{k \ge 1}$ is $p$-uniformly integrable, then the weak convergence $\nu_k \lra \nu$ also implies the convergence in the $p$-Wasserstein distance \cite[Theorem 5.5]{prob_mean_field}, that is, $\wassdist_p(\nu_k,\nu)\ra 0$ as $k\ra +\infty$, where $\wassdist_p$ denotes the $p$-Wasserstein distance. Likewise, we denote by $\polish_{p,\rm sub}(\Ree^n)$ the subspace of $\polish_{\rm sub}(\Ree^n)$ consisting of sub-probability measures with finite moment of order $p$.

$C^k(\Ree^n)$ denotes the space of $k$-times continuously differentiable functions on $\Ree^n$, while $C^{k,\ell}(\Ree^n\times\Ree^m)$ denotes the space of functions that are $k$-times continuously differentiable with respect to the first argument, and $\ell$-times continuously differentiable with respect to the second argument on $\Ree^n \times \Ree^m$. Furthermore, $C_{\rm loc}^k(\Ree^n)$ denotes the space of functions whose restriction to every compact subset of $\Ree^n$ is $k$-times continuously differentiable. 
Moreover, $\mathscr{L}_{1}$ denotes the space of all absolutely integrable random variables, and for subsets $\mathscr{D}\subseteq \Ree^{d}$ and $\mathscr{R}\subseteq \Ree^{r}$ with \(d,r \in \N\), we write $\mathscr{L}_{\infty}(\mathscr{D};\mathscr{R})$ for the space of all measurable, essentially bounded functions from $\mathscr{D}$ to $\mathscr{R}$.

For any differentiable function $f:\Ree^{n}\ra\Ree$ with arguments $(x_1,\dots,x_n)$ and any $y=(y_1,\dots,y_n)\in\Ree^{n}$, we write
\begin{align}
 \frac{\partial}{\partial { x}_i}{f}(y_1,\dots,y_n)
\Let \left.\frac{\partial {f}(x_1,\dots,x_n)}{\partial x_i}\right|_{(x_1,\dots,x_n)=(y_1,\dots,y_n)}.\nn   
\end{align}
For a differentiable \emph{functional} $\mathscr{J}: \Ree_{\ge 0} \times \polish(\Ree^n) \lra \Ree $, we denote by $\frac{\delta \mathscr{J}}{\delta \mu}(t,\nu)(x)$ the \emph{functional} derivative of $\mathscr{J}$ at $(t,\nu,x)\in(\Ree_{\ge 0},\polish(\Ree^n),\Ree^n)$, 
that is, for any $\nu, \tilde \nu\in \polish(\Ree^n)$, the directional derivative \cite[Chapter 1]{prob_mean_field} satisfies 
\begin{align*}
    &\lim_{\eps\lra 0} \frac{\mathscr{J}(t,(1-\eps)\nu + \eps \tilde \nu)-\mathscr{J}(t,\nu)}{\eps} = \int_{\Ree^n} \frac{\delta \mathscr{J}}{\delta \mu}(t,\nu,x) \d(\tilde \nu-\nu)(x).
\end{align*}
Here, $\frac{\delta \mathscr{J}}{\delta \mu}$ is defined up to an additive constant, that is, we adopt the normalization such that $\int_{\Ree^n} \frac{\delta \mathscr{J}}{\delta \mu}(t,\nu,x) \d \nu(x)=0$.
Moreover, the functional derivative is related to the Lions (L)--derivative 
$\partial_\mu \mathscr{J}$
by \cite[Proposition 5.48]{prob_mean_field} $\partial_\mu \mathscr{J}(t,\nu)(x) \Let \nabla_x \Bigl(\frac{\delta \mathscr{J}}{\delta \mu}(t,\nu)(x)\Bigr)$, assuming the latter
is jointly continuous in $(\nu,x)$, has at most a linear growth in $x$, uniformly in $\nu\in \mathscr{P}$ for any bounded subset $\mathscr{P}\subset\polish_2(\Ree^n)$.

We consider a complete probability space $(\Omega, \filtration, \mathbb{P})$, where $\Omega$ is the sample space, $\omega \in \Omega$ denotes a sample point, $\filtration$ is a $\sigma$-algebra of subsets of $\Omega$, $\mathbb{P} : \filtration \lra [0,1]$ is a probability measure, satisfying $\mathbb{P}(\Omega) = 1$, with all $\mathbb{P} $-null subsets of $\Omega$ included in $\filtration$. 
A stochastic process $\{x_t\}_{t\ge0}$ taking values in $\Ree^n$ is said to be $\filtration_t$-adapted if, for each $t\ge0$, the random variable $x_t$ is $\filtration_t$-measurable. For any $k\in \N$, we denote by $\mathscr{L}_{1}(\Omega,\filtration,\mathbb P;\mathbb R^k)$, or simply $\mathscr{L}_{1}$ when the context is clear,
the space of all $\filtration$-measurable random vectors $Z:\Omega\lra\mathbb R^k$ such that
\begin{align*}
    \mathbb E^{\mathbb P}[\|Z\|] = \int_{\Omega} \|Z(\omega)\|\,\mathbb P(\mathrm d\omega) < +\infty,
\end{align*}
where $\|\cdot\|$ denotes the Euclidean norm, and $\mathbb E^{\mathbb P}[Z] \Let \int_{\Omega} Z(\omega)\,\mathbb P(\mathrm d\omega)$.

Let $\mu_t$ denote the probability measure of $x_t$ under $\mathbb P$ at time $t\ge 0$. We fix a regular conditional probability of $\mathbb P$ given $x_t$, i.e., a probability kernel
$\{\mathbb P^{t,x}\}_{x\in\mathbb R^n}$ on $(\Omega,\filtration)$, denoting $\mathbb P(\cdot|x_t=x)$ (defined $\mu_t$-a.e.\ in $x$), such that, for each $A\in\filtration$,
the map $x\mapsto \mathbb P^{t,x}(A)$ is Borel measurable, and for all $A\in\filtration$ and
$B\in\mathcal B(\mathbb R^n)$,
\begin{align*}
    \mathbb P\big(A\cap \{x_t\in B\}\big) = \int_{B}\mathbb P^{t,x}(A)\, \d \mu_t(x).
\end{align*}
Throughout this article, 
we denote a probability measure with a time index, e.g., $\mu_t$, whenever the probability law evolves in time, whereas $\mu$ denotes a generic probability measure, i.e., when no explicit temporal dependence is required. We use ${\rm S}_\mu \Let \supp(\mu)$ to denote the support of $\mu$, and when there is a time dependency, we use the notation ${\rm S}_{\mu_t}\Let \supp(\mu_t)$, instead. 

We are now ready to proceed to the setting that will be the main focus of this work.


\subsection{Background Theory}\label{sect:back_theory}
We consider a controlled stochastic system governed by the following dynamics
\begin{align}\label{eqn1}
\d x_t = f(t, x_t, u_t) \, \d t + \sigma_t \, \d w_t, \quad \text{for all }t \ge 0,
\end{align}
with the following data:
\begin{enumerate}[label={\textup{(\eqref{eqn1}-\alph*)}}, leftmargin=*, widest=b, align=left]
\item \label{sde:data:1} $x_t\in\Ree^{n}$ denotes the state process, $u_t\in\Ree^{m}$ the control process, $w_t$ an $n$-dimensional Brownian motion, and $\sigma_t\in\Ree^{n\times n}$ a deterministic time-varying diffusion matrix, all evaluated at time $t\in\Ree_{\ge 0}$;
\item \label{sde:data:2} the mapping $f:\Ree_{\ge 0} \times \mathbb{R}^n \times \mathbb{R}^m\lra \mathbb{R}^n$ is a deterministic drift function. We denote by $\{\filtration_{t} \}_{t\ge 0}$ the filtration generated by $w_t$ and we assume that, for any $t$, the map $t \mapsto \sigma_t \sigma_t \t$ is uniformly elliptic, i.e., $\sigma_t \sigma_t\t \succ 0$ for all \(t\ge 0\).
\end{enumerate}

\begin{assumption}[Regularity of the dynamics] \label{regularity_dyn}
The drift function $f$ is continuously differentiable in all its arguments with bounded first derivatives, and the diffusion matrix $\sigma$ is bounded and continuous. Moreover, there exists a constant $\ell>0$ and a modulus of continuity $\Gamma:[0,+\infty)\to[0,+\infty)$ such that, for all $t\in\mathbb{R}_{\ge 0}$, $x,\hat x\in\mathbb{R}^n$, and $u,\hat u\in\mathbb{R}^m$, the following hold:
\begin{align}
\|f(t,x,u)-f(t,\hat x,\hat u)\|
    &\le \ell \|x-\hat x\|+\Gamma(\|u-\hat u\|), \tag{(A1)} \label{A1}\\
\|f(t,0,u)\| &\le \ell, \tag{(A2)} \label{A2} \\
\|\sigma_t\| &\le \ell. \tag{(A3)}\label{A3}
\end{align}
\end{assumption}

Under the assumptions \eqref{A1}-\eqref{A3}, the SDE \eqref{eqn1} admits a unique strong solution \cite[Ch.~1, Sec.~6.4]{book_sc} with continuous sample paths. 

\begin{remark}\label{rem:on:dynamics:diffusion}
Observe that the diffusion coefficient may be taken in the more general state and control dependent form $(t,x,u)\mapsto \sigma_t(x,u)$ and the uniform ellipticity condition \blue{may be} relaxed. Nonetheless, to streamline the exposition and keep the subsequent derivations transparent, we restrict attention to the simpler time-dependent only case $\sigma_t$.
\end{remark}

\textbf{Some important function spaces and sets:} 
We define the set of admissible controls and (space-time) boundary mappings
\begin{align} \label{backgrd:mapping:sets}
\Pi \Let& \aset[\big]{\pi\in \mathscr{L}_\infty(\Ree^n; \Ree^m) \suchthat \pi \text{ measurable}},\nn \\
\admcont \Let& \left\{u_t \Let \pi\in  \mathscr{L}_\infty(\Ree^n; \Ree^m) \;\middle\vert\;  
    \begin{array}{@{}l@{}}
         \pi\ 
\text{measurable}, \{u_t\}_{t\ge 0} \ \text{is progressively} \\ \text{measurable w.r.t.} \ \{\filtration_t \}_{ t \ge 0}
        \end{array}
        \right\}, \nn \\
\admbdr \Let& \left\{ g:\Ree_{\ge0}\times\Ree^n\lra\Ree \;\middle\vert\;  
\begin{array}{@{}l@{}}
\text{for each \(x\), the map }t \mapsto g(t,x) \text{ is continuous},\\
\text{and (locally) Lipschitz continuous in } x, \text{ uniformly in } t 
        \end{array}
        \right\}.
\end{align}

Given a stopping boundary $g\in\admbdr$, we define the associated $\filtration_t$-adapted stopping time (e.g., \cite[Chapter 1.3.]{book_sc}) by
\begin{align}
\tau_g \Let \inf \aset[\big]{t\ge0 \suchthat g(t,x_t) \ge 0},
\end{align}
with the convention that $\tau_g=+\infty$ on the event $\aset[]{g(t,x_t)<0 \text{ for all }t\ge 0}$. That is, $\tau_g$ is the first hitting time of the stochastic process $\{x_t\}_{t\ge 0}$ to the stopping boundary in space--time determined by $g\in \admbdr$.

We define the \emph{continuation (alive)} region by
\begin{align}\label{backgrd:alive:region}
\mathcal{D}^g \Let \aset[]{(t,x)\in \Ree_{\ge 0}\times\Ree^n \suchthat g(t,x)<0 } .
\end{align}
In other words, the stochastic process $\{x_t\}_{t\ge 0}$ continues to evolve as long as $g(t,x_t)<0$, equivalently, as long as $(t,x_t)\in \mathcal{D}^g$, and it is stopped whenever the process exits the continuation region $\mathcal{D}^g$.
We also consider the spatial continuation set at time $t\in\Ree_{\ge 0}$, denoted by $\mathcal D_t^g$, as well as the absorption boundary of the continuation region, denoted by $\partial\mathcal D^{g}$, and its time-slice boundary $\partial\mathcal D_t^{g}$. These are defined by
\begin{align}\label{backgrd:boundary:sets}
    &\mathcal D_t^g \Let \aset[]{x\in\Ree^n \suchthat g(t,x)<0}, \nn \\
    \partial&\mathcal D_t^g \Let \aset[]{x\in\Ree^n \suchthat g(t,x)=0}, \nn \\
    \partial&\mathcal D^g \Let \aset[]{(t,x)\in \Ree_{\ge0}\times \Ree^n \suchthat  g(t,x)=0}.
\end{align}
Moreover, since $t \mapsto x_t$ is continuous a.s.\ and $(t,x) \mapsto g(t,x)$ is continuous, the map $t\mapsto g\bigl(t,x_t\bigr)$ is continuous a.s.\ Hence, whenever $\tau_g<+\infty$, we have
$g\bigl(\tau_g,x_{\tau_g}\bigr)=0$, i.e.\ $(\tau_g,x_{\tau_g})\in \partial \mathcal{D}^g$
(or equivalently $x_{\tau_g}\in \partial \mathcal{D}^g_{\tau_g}$).

We employ two notations for conditional expectation to distinguish whether the initial condition is deterministic or random with a given initial probability measure. The notation $\mathbb{E}^{\filtration_t}_{x}[\cdot]$ is used in the pointwise setting, corresponding to a deterministic initial condition, in Subsection \ref{subs_3_2}, whereas $\mathbb{E}^{\filtration_t}_{x \sim \mu}[\cdot]$ is used in the distributional setting developed later in the same subsection.
\begin{enumerate}[label={\textup{(\ensuremath{\mathsf{E}}-\alph*)}}, leftmargin=*, widest=b, align=left]
    \item \textit{Pointwise expectation.} For a fixed initial condition $x \in \Ree^n$ at time $t$, we define the conditional expectation over trajectories starting from $x_t \Let x$ under the filtration $\filtration_t$ by
\begin{align}\label{backgrd:pw:expect}
        \mathbb{E}^{\filtration_t}_{x}[\cdot] \Let \mathbb{E}^{\mathbb{P}^{t,x}}\bigl[\cdot \mid \filtration_t\bigr].
    \end{align}
    That is, given a stochastic process $\{x_s\}_{s \ge t}$, the filtration $\filtration_t$ generated by the process  $\{x_s\}_{s \ge t}$,
    and a random variable $Z \in \mathscr{L}_1$, we interpret $\mathbb{E}^{\mathbb{P}^{t,x}}[Z\mid \filtration_t]$ as $\filtration_t$-measurable random variable $Y$ (unique $\mathbb{P}^{t,x}$-a.s.) such that \cite[Proposition 1.10]{book_sc}
    \begin{align*}
    \int_{A} Y(\omega)\,\mathbb{P}^{t,x}(\mathrm{d}\omega)
    =
    \int_{A} Z(\omega)\,\mathbb{P}^{t,x}(\mathrm{d}\omega)
    \quad \text{for all }A\in\filtration_t.
    \end{align*}
    The notation $\mathbb{E}_x^{\filtration_t}[\,\cdot\,]$ emphasizes that state trajectories are initialized from the deterministic condition $x_t = x$.

    \item \textit{Distributional expectation.} When the initial state is a random variable with the given initial probability measure $\mu \in \polish(\Ree^n)$, we define the conditional expectation over trajectories where $x_t \sim \mu$, conditioned on the filtration $\filtration_t$ by
\begin{align}\label{backgrd:dist:expect}
        \mathbb{E}^{\filtration_t}_{x \sim \mu}[\,\cdot\,] \Let  \int_{{\rm S}_{\mu}} \mathbb{E}^{\filtration_t}_{x} \left[\,\cdot\,\right] \d \mu (x).
    \end{align}
    By an abuse of notation, for $\mu \in \polish_{\rm sub}(\Ree^n)$, we still write $x\sim\mu$. This should be understood as notation for integration with respect to the sub-probability measure $\mu$.
\end{enumerate}

\begin{remark} \label{backgrd:det_PDE:expect}
When $x_0 \sim \mu_0$, we denote the expectation over the given initial probability measure $\mu_0$ under the filtration $\filtration_0$ by $\mathbb{E}^{\filtration_0}_{x \sim \mu_0}[\,\cdot\,]$. 
When we work with deterministic dynamics or PDEs, we use $\mathbb{E}_{t,x \sim \mu}[\,\cdot\,]$ to denote the integration over a given probability measure $\mu\in \polish(\Ree^n)$ at time $t$, i.e., the probability measure of $x$ at time $t$ is $\mu$. Whenever we invoke $\wassdist_2$-based arguments or Lions derivatives, we implicitly restrict attention to measures $\mu \in \polish_2(\Ree^n)$.
\end{remark}

\section{The central problem}\label{sec:problem}

\subsection{Problem Formulation}\label{subsec:prob:formulation}
Let $\{x_t\}_{t\ge0}$ denote the stochastic process with a random initial condition, $x\sim \mu_0\in \polish_2(\Ree^n)$. Unless otherwise stated, we always work with absolutely continuous initial probability measures with respect to the Lebesgue measure. Due to the presence of a stopping boundary, we work with the alive component of the full law (sub-probability measure) of $x_t$ on the continuation region $\mathcal{D}^g$ with respect to the Lebesgue measure, i.e., the alive component is absolutely continuous with respect to the restriction of the Lebesgue measure to $\mathcal{D}^g$. We refer the interested reader to \cite[\S0.6, \S10.4]{book_dynkin} for a classical treatment of subprocesses of Markov processes and the associated sub-probability measures.

We state our primary problem now.

\begin{mdframed}[backgroundcolor=black!10,skipabove=4pt, linewidth=1pt, innertopmargin=4pt]
\highlight{Problem (P0):} The objective is to minimize the expected value of a Bolza-type cost functional consisting of a standard Lagrangian and a Mayer term
\begin{align} \label{eqn2}
\min_{\substack{g \in \mathcal{G}, u\in \mathcal{U}}} \; \mathbb{E}^{\mathcal{F}_0}_{x\sim\mu_0} \left[ \int_0^{\tau_g} L(t, x_t, u_t(x_t)) \, \d t + \Phi(\tau_g, x_{\tau_g}) \right],
\end{align}
where $L:\mathbb{R}_{\ge 0}\times \mathbb{R}^n \times \Ree^m \lra \mathbb{R}$ represents the running (Lagrangian) cost, $\tau_g$ is the stopping time corresponding to the boundary $g\in \mathcal{G}$, the state $\{x_t\}_{t \ge 0}$ evolves according to the stochastic dynamics \eqref{eqn1}, $x_{\tau_g}$ is the stopped state, and $ \Phi: \mathbb{R}_{\ge 0}\times \mathbb{R}^n \lra \mathbb{R} $ is the terminal (Mayer) cost at $ {\tau_g} $. 
The control $u\in \mathcal{U}$ has to be chosen to minimize the total cost \eqref{eqn2} up to the stopping time ${\tau_g}$. 
\end{mdframed}
We make the following standing regularity assumptions for well-posedness, which will always be enforced throughout the paper.

\begin{assumption}[Regularity of the cost] \label{regularity_cost}
The maps $L,\Phi$ are continuously differentiable in all their arguments with bounded first derivatives, and there exists a constant $\ell>0$ and a modulus of continuity $\Gamma:[0,+\infty) \ra [0,+\infty)$ such that, for all $t\in \Ree_{\ge0}$, $x, \hat x\in \mathbb{R}^n$, $u, \hat u\in \mathbb{R}^m$, the following hold:
\begin{align}
    &|L(t,x,u)-L(t,\hat x,\hat u) | \le \ell \|x-\hat x\|+\Gamma(\| u-\hat u\|), \tag{(A4)}\label{A4}\\
    &|L(t,0,u)|\le \ell \tag{(A5)}\label{A5}, \\
    &|\Phi(t,x)-\Phi(t,\hat x)|\le \ell\|x-\hat x\|. \tag{(A6)}\label{A6}
\end{align}
\end{assumption}
Under the assumptions \eqref{A1}-\eqref{A6}, the cost functional \eqref{eqn2} is well-defined  \cite[Ch.~3, Sec.~3]{book_sc}.

\begin{remark}[Regularity of the Optimal Stopping Boundary]\label{reg_opt_bound}
    In this paper, we do not address the regularity of the optimal stopping boundary; free-boundary regularity in optimal stopping literature is complex and heavily problem dependent, and remains an unsolved problem in the general setting. Throughout, we fix an admissible class $\mathcal{G}$ of stopping boundaries and assume that there exists an optimal pair $(g^*,u^*)\in \mathcal{G} \times \mathcal{U}$. All the results in the following sections provide necessary conditions for any such optimal boundary $g^* \in \mathcal{G}$.
\end{remark}


\subsection{Sub-probability Measures} \label{subs:subprob}
To distinguish the absorbed states after exit, we augment an absorbing state $\dagger\notin \Ree^n$ (sometimes called a \emph{cemetery state}), and define the extended state space $\mathbb{X} \Let  \Ree^n \cup\{\dagger\} $, with $\dagger$ being a singleton that is disjoint from $\mathbb{R}^n$. 
Define the \emph{killed process} $\{\tilde x_t\}_{t\ge 0}$ 
on $\mathbb{X}$
by
\begin{align} \label{killed_process}
    \tilde x_t \Let x_t \indic{\{t<\tau_g\}} + \dagger\,\indic{\{t\ge \tau_g\}}.
\end{align}
We will use $\mu_t$ to denote the full probability measure of $\tilde x_t$ on $\mathbb{X}$. Consequently, it is possible to decompose the full probability measure of the process $\tilde x_t$ into an \emph{alive part} supported on $\mathcal{D}_t^g$ and a \emph{cemetery part} concentrated at the cemetery state $\dagger$.
\blue{
To describe the distribution of states that have not yet been stopped, we introduce the following measure.
\begin{definition}[Alive sub-probability measure]\label{def:alive-subprobability}
Let \(\{x_t\}_{t\geq 0}\) be the state process and let \(\tau_g\) be the stopping time associated with \(g\in\mathcal{G}\). For each \(t\geq 0\), the \emph{alive sub-probability measure} is the nonnegative Borel measure \(\mu_t^a\) on \(\mathbb{R}^n\) defined by
\begin{align}\label{eq:alive-subprobability} \mu_t^a(B) \Let \mathbb{P}\bigl(x_t\in B,\ t<\tau_g\bigr) \quad \text{for all }B\in\mathcal{B}(\mathbb{R}^n). \end{align}
Thus, \(\mu_t^a(B)\) is the probability that the process has not stopped by time \(t\) and its state lies in \(B\). In other words, \(\mu_t^a\) is the \emph{restriction to} \(\mathbb{R}^n\) of the law \(\mu_t\) of the killed process \(\tilde{x}_t\). Its mass is concentrated in the continuation set \(\mathcal{D}_t^g\), and \(\mu_t^a(\mathbb{R}^n)=\mu_t^a(\mathcal{D}_t^g)=\mathbb{P}(\tau_g>t)\leq 1.\)
In particular, its total mass equals the survival probability at time \(t\), which explains the term `sub-probability'.
\end{definition}
}


Unless otherwise stated, we impose the standing regularity hypothesis that, $\mu_t^{a}\ll \mathrm{Leb}^n \upharpoonright_{\mathcal{D}_t^g}$, holds for all $t\ge 0$. Therefore, by the Radon--Nikodym theorem (e.g., see \cite[Theorem 1.30]{evans_measure}), there exists a (unique up to a.e. equality) sub-probability density, denoted by $\rho^a(t,\cdot)$, such that,
\begin{align*}
\mu_t^a(A)=\int_A \rho^a(t,x)\, \d x \quad \text{for all }A\in \Borelsigalg(\mathcal{D}_t^g).
\end{align*} 
With the cemetery-state convention, the process is sent to the state $\dagger$ after it exits the region $\mathcal{D}^g$, and as a result, the absorbed mass at time $t$ equals $\mathbb{P}(\tau_g\le t) \, \delta_\dagger$ (with $\delta_\dagger$ denoting the Dirac measure at $\dagger$).

\begin{remark}[Induced $\mathcal{W}_p$ topology for sub-probability measures]
\label{rem_wp_subp}
For any $\mu \in \mathcal P_{p,\rm{sub}}(\Ree^n)$, define the associated probability
measure $\mu^\dagger$ on the extended state space $\mathbb{X}$ by
\begin{align*}
    \mu^\dagger(A)
    \Let
    \mu(A \cap \Ree^n)
    +
    \left(1-\mu(\Ree^n)\right)\delta_\dagger(A)
    \quad \text{ for all } A \in \mathcal B(\mathbb X).
\end{align*}
Whenever a $\mathcal{W}_p$-type topology on $\mathcal P_{p,\rm{sub}}(\Ree^n)$ is invoked, we mean the topology induced by this definition, namely,
\begin{align*}
    \mu_k \lra  \mu \text{ in }\mathcal{P}_{p,\rm sub}(\Ree^n) \iff \mathcal{W}_p(\mu_k^\dagger,\mu^\dagger) \lra 0.
\end{align*}
This topology is used only to topologize sub-probability measures; the value function itself remains defined on the alive sub-probability measures on $\Ree^n$.
\end{remark}

Finally, to characterize when and where the stopping occurs, we define the absorbed (stopped) distribution \cite{killing_tg1,ref:AE:TTG:SB-timeMarg}. Let ${\rm E}$ be any Borel set ${\rm E} \subseteq \partial\mathcal{D}^g$. 
Then, the cumulative total exit mass through ${\rm E}$, up until time $t\ge 0$, can be written as
\begin{align*}
\tilde\rho(t,{\rm E}) \Let \mathbb{P}\big(\tau_g\le t, (\tau_g,x_{\tau_g})\in {\rm E} \big),
\end{align*}
where $\tilde\rho(t,\partial\mathcal{D}^g)=\mathbb{P}(\tau_g\le t)$ is the total absorbed probability by the time $t\ge 0$. From the conservation of the total probability, we have the classical identity $\mu_t^a(\mathcal{D}_t^g)+\tilde\rho(t,\partial\mathcal{D}^g)=1$ for all $t\ge 0$. Moreover, under $g\in \mathcal{G}$, the rate of loss of the alive mass  equals the outward probability flux across $\partial\mathcal{D}^g$, i.e., $\frac{\mathrm d}{\mathrm dt}\mu_t^a(\mathcal{D}_t^g)=-\frac{\mathrm d}{\mathrm dt}\tilde\rho(t,\partial\mathcal{D}^g)$.

Assume that $\rho^a(\cdot)$ satisfies the Fokker--Planck equation in $\mathcal D^g$
\begin{align*}
\frac{\partial}{\partial t}\rho^a(t,x) = -\nabla_x\!\cdot J(t,x), \quad \text{for all } (t,x)\in\mathcal D^g,
\end{align*}
where $J:\Ree_{\ge 0}\times \Ree^n \lra \Ree^n$ is the associated probability flux density for the SDE \eqref{eqn1}. Let $S(\cdot)$ be the $(n-1)$-dimensional Hausdorff measure and $\alpha(t,x)$ be the outward probability flux density across $\partial\mathcal{D}_t^g$, defined by
\begin{align*}
\partial \mathcal{D}^g 
\ni (t,x) \mapsto \alpha(t,x) \Let \Big(J(t,x)-\rho^a(t,x)\,v_{\mathrm{out}}(t,x)\Big)\cdot \hat{n}(t,x),
\end{align*}
with $\hat{n}:\Ree_{\ge 0}\times \Ree^n \lra \Ree^n$ being the outward-pointing unit normal field on $\partial\mathcal D_t^g$ and $v_{\mathrm{out}}:\Ree_{\ge 0}\times\Ree^n \lra \Ree^n$ being the boundary velocity vector field (note that $v_{\mathrm{out}}\equiv 0$ whenever $g$ is time-invariant).
Then, applying the divergence theorem yields
\begin{align} \label{mass_lost_flux}
    \frac{\mathrm d}{\mathrm dt}\mu_t^a(\mathcal{D}_t^g)= -\int_{\partial\mathcal{D}_t^g}\alpha(t,x)\,\d S(x),
\end{align}
which immediately implies $\frac{\mathrm d}{\mathrm dt}\tilde\rho(t,\partial\mathcal{D}^g)
=
\int_{\partial\mathcal{D}_t^g}\alpha(t,x)\,\d S(x)$.

\begin{assumption}[Dirichlet Boundary Conditions for Stopping]
\label{remark_abs2}
We adopt the standard absorbing (Dirichlet) boundary condition \cite{book_stop} for the stopped process. Specifically, we define the alive probability density function such that $\d \mu^a_t=\rho^a(t,x) \d x$, for all $(t,x)\in \mathcal{D}^g$. Moreover, $\rho^a(\cdot)$ satisfies the Dirichlet boundary condition, in the trace sense, $\rho^a(t,x)=0$ for all $(t,x)\in \partial \mathcal
D^g$. 
\end{assumption}

The preceding discussion of sub-probability measures will be fundamental to the subsequent derivations. Specifically, the alive sub-probability measure $\mu_t^a$ will be used to compute the accumulated running cost by integration over the continuation region $\mathcal D^g$, and moreover, the outward flux across $\partial\mathcal D^g$ represents the rate of probability mass loss through the boundary, and will provide the basis for deriving the accumulated terminal cost.


\subsection{Stochastic Hamilton-Jacobi-Bellman (sHJB) Equation} \label{subs_3_2}

Let the value function $(t,x)\mapsto V(t,x)$ corresponding to the objective function \eqref{eqn2} with $\mu_0 \Let \delta_{x}$ (the delta distribution concentrated at $x$), i.e., conditioned on a deterministic initial condition, be defined as
\begin{align} \label{single_V}
(t,x) \mapsto V(t,x) \Let  \min_{g\in \admbdr,u\in \admcon} \;\mathbb{E}_x^{\mathcal{F}_t}\left[\int_t^{\tau_g} L(s,x_s,u_s(x_s)) \, \d s + \Phi({\tau_g},x_{\tau_g})\right],
\end{align}
and consider the following \textit{stochastic} Hamilton-Jacobi-Bellman (sHJB) 
partial differential 
equation~\cite{book_sc}, associated with Problem \highlight{(P0)} and \eqref{single_V}
\begin{align} \label{eqn:value_func}
    \min\limits_{u\in \Ree^m} \aset[\big]{\mathcal{L}_u V(t,x)+ L(t,x,u)} = 0,
\end{align}
over the domain $\mathcal{D}^g$, with terminal condition $V(\tau_g,x_s)=\Phi(\tau_g,x_s)$ for all $ (\tau_g,x_s)\in\partial \mathcal{D}^g$. Unless stated otherwise, \eqref{eqn:value_func} is understood in the distributional sense, and in \eqref{eqn:value_func}, for a smooth test function $\varphi\in C^{1,2}$, the controlled second-order generator $\mathcal{L}_u$ for the SDE \eqref{eqn1}, is defined as~\cite{book_sc}
\begin{align} \label{sHJB-generator}
\mathcal{L}_u \varphi(t,x) &\Let \frac{\partial }{\partial t} \varphi(t,x) + f(t,x,u)^{\t} \frac{\partial }{\partial x}\varphi(t,x) \nonumber \\
&\hspace{0 em}+ \half \trace\Big( \sigma_t \sigma_t\t \frac{\partial^2}{\partial x^2}\varphi(t,x) \Big), \quad \text{for } u\in \Ree^m,(t,x)\in \mathcal
D^g.
\end{align}
Let $(g^*,u^*)\in \mathcal{G} \times \mathcal{U}$ be an optimizer pair for the problem \eqref{single_V}. 
Then, for all $(t,x) \in \Ree_{\ge 0} \times \Ree^n$, assuming sufficient regularity conditions, i.e., \eqref{A1}-\eqref{A6}, the value function \eqref{single_V} satisfies the following Variational Inequalities \cite{book_stop},
\begin{subequations}
\begin{align}
\mathcal{L}_{u_t^*(x)} V(t,x)+L(t,x,u_t^*(x)) \ge 0, \quad & \text{if } g^*(t,x)\ge0, \tag{(11-a)} \label{eq:HJBforStopping0} \\ 
\mathcal{L}_{u_t^*(x)}  V(t,x)+L(t,x,u_t^*(x)) =0, \quad& \text{if } g^*(t,x) <0, \tag{(11-b)}\label{eq:HJBforContinuation}\\
V(t,x)=\Phi(t,x), \quad& \text{if } g^*(t,x) =0 \tag{(11-c)}\label{eq:HJBforStopping},
\end{align}
\end{subequations}
where \eqref{eq:HJBforStopping0}-\eqref{eq:HJBforStopping} are understood in the weak sense, unless additional regularity is imposed.

Moreover, with a given stopping boundary $g\in\mathcal{G}$, the fixed-boundary value function $(t,x) \mapsto V^g(t,x)$ is defined by
\begin{align} \label{single_Vg}
    V^g(t,x) \Let \min_{u\in \mathcal{U}} \;\mathbb{E}_x^{\mathcal{F}_t}\left[\int_t^{\tau_g} L(s,x_s,u_s(x_s)) \, \d s + \Phi({\tau_g},x_{\tau_g}) \right].
\end{align}
It is well known that, for all $(t,x)\in \mathcal{D}^g$ both $V(t,x)$ and $V^g(t,x)$ satisfy \eqref{eqn:value_func}; see, e.g., \cite{book_stop} for more details. We define a couple of more crucial objects for our subsequent analysis. 


\begin{definition} \label{Def:dist}
Recalling the continuation set $\mathcal{D}^g$ in \eqref{backgrd:alive:region}, we define the \emph{distributional continuation region} 
\begin{align} \label{cntr:dist_alive:region}
    \bar{\mathcal{D}}^g \Let \aset[\big]{(t,\mu) \in \Ree_{\ge 0} \times \mathcal{P}_{2,\rm sub}(\Ree^n), \,  x\in \operatorname{supp}(\mu) \suchthat (t,x) \in \mathcal{D}^g}.
\end{align}
Then, the \emph{distributional value functions} $\bar V(\cdot)$ and $\bar V^g(\cdot)$ for the objective function in \eqref{eqn2} are defined by
\begin{align} 
&\bar{\mathcal{D}}^g \ni (t,\mu)\mapsto \bar{V}^g(t,\mu) \Let \min_{u\in \mathcal{U}} \mathbb{E}^{\mathcal{F}_t}_{x\sim\mu} \left[ \int_t^{\tau_g} L(s, x_s, u_s(x_s)) \, \d s + \Phi({\tau_g}, x_{\tau_g})  \right], \label{eqn3}\\
&\bar V(t,\mu) \Let \min_{g\in\mathcal{G}} \bar V^g(t,\mu). \label{eqn55}
\end{align}
\end{definition}


\begin{definition}\label{def:test}
We define the test functional class $C^{1,2}_2$ as follows. A functional $\Xi:\Ree_{\ge 0} \times \mathcal{P}_{2,\rm sub}(\Ree^n)\lra\Ree $ belongs to the class $C^{1,2}_2$, if:
\begin{enumerate}[label={\textup{(\eqref{def:test}-\alph*)}}, leftmargin=*, widest=b, align=left]
\item $\frac{\partial}{\partial t} \Xi$ exists and is continuous in all variables;
\item the linear functional derivative $\frac{\delta \Xi}{\delta \mu}$, and its derivatives, $\frac{\partial}{\partial x}\bigl(\frac{\delta \Xi}{\delta \mu}\bigr)$,
$\frac{\partial^2}{\partial x^2}\bigl(\frac{\delta \Xi}{\delta \mu}\bigr)$ exist, and are continuous in all variables;
\item $\frac{\partial^2}{\partial x^2}\bigl(\frac{\delta \Xi}{\delta \mu}\bigr)$ is bounded in $x$, locally uniformly in $(t,\mu)\in \Ree_{\ge 0}\times \mathcal{P}_{2,\rm sub}(\Ree^n)$.
\end{enumerate}
\end{definition}

\begin{assumption}[Viscosity Regularity of the Value Function] \label{regularity_val} Let $g\in \mathcal{G}$ and consider the distributional value functions $\bar V^g,\bar V:\Ree_{\ge 0} \times \mathcal{P}_{2,\rm sub}(\Ree^n)\lra  \Ree$ defined in \eqref{eqn3} and \eqref{eqn55}, respectively. We assume the following:

\begin{enumerate}[label={\textup{(\eqref{regularity_val}-\alph*)}}, leftmargin=*, widest=b, align=left]
\item $\bar V^g$ and $\bar V$ are locally bounded on their domain, and are upper semicontinuous with respect to their arguments under the induced $\mathcal{W}_2$ topology on $\mathcal P_{2,\rm{sub}}(\Ree^n)$
introduced in Remark~\ref{rem_wp_subp}.

\item $\bar V^g$ and $\bar V$ satisfy the associated distributional sHJB/variational inequality in the viscosity sense (see, e.g.,~\cite{vis_comp}), with the test functional class $C^{1,2}_2$ specified in Definition \ref{def:test}. In other words, whenever derivatives of $\bar V^g$ and $\bar V$  appear, they are understood through smooth test functionals $\Xi \in C^{1,2}_2$.
\end{enumerate}
\end{assumption}

A few remarks are in order. 


\begin{remark} \label{remark:vis_sense}
All applications of It\^o's formula (see, e.g., \cite[Thm.~4.2.1]{book1}), integration by parts, and Stein-type identities in the rest of the article are carried out on the corresponding smooth test functional $\Xi\in C^{1,2}_2$.
\end{remark}

\begin{remark}\label{rem:van:tail}
Assumption~\ref{regularity_val} is standard in the stochastic optimal stopping setting~\cite[Theorem 3.10]{vis_comp}
and only assumes that $\bar V$ is a viscosity solution;  
in particular, it does not imply that $\bar V\in C^{1,2}_2$, even on the continuation region $\bar{\mathcal{D}}^{g^*}$, associated with an optimal stopping rule $g^*\in\arg\min_{g\in\mathcal G}\bar V^g(\cdot)$ (when a minimizer exists).
However, under stronger regularity conditions \cite[Lemma 3.7]{vis_comp}, the unstopped distributional value functional is as regular as the test class. 
Consequently, on $\bar{\mathcal D}^{g^*}$, $\bar V$ is of class $C^{1,2}_2$, and can be interpreted as a classical solution \cite[Theorem 3.4]{vis_comp} in that region.
\end{remark}

Throughout this work, we use the following boundary conditions on an unbounded domain and the stopping boundary $g\in \mathcal{G}$.

\begin{remark}[Vanishing Tail Conditions]\label{remark_abs}
    Let the density of the initial probability measure $\mu_0 \in \mathcal{P}(\Ree^n)$ be denoted by $\rho_0 \in C^1(\Ree^n)$, 
    that is, $\d \mu_0 (x) = \rho_0 (x) \, \d x$.
    Then, for uniformly elliptic systems, i.e., for the case $\sigma_t \sigma_t \t \succ 0$, and under no stopping conditions, the state distribution $\mu_t$ is fully supported for all $t>0$, i.e., ${\rm S}_{\mu_t} = \Ree^n$, even if $\mu_0$ is compactly supported. Furthermore, $\lim_{\lVert x\rVert\to +\infty} \rho(t,x) =\lim_{\lVert x\rVert\to +\infty} \nabla_{x_i} \rho(t,x) =  0$, for all $t > 0$ and $i\in{1,\dots,n}$ \cite[Theorem 4.1]{Pavliotis}. 
    This is a direct result of two-sided Gaussian estimates proven by Aronson \cite{aronson}, \cite[Theorem 7]{aronson2} for the transition density of uniformly parabolic operators and their derivatives. 
    For recent results on general coefficients, we refer the reader to \cite{modern_bound}.
\end{remark}
Accordingly, whenever needed in the following derivations, we impose the vanishing-tail conditions from Remark \ref{remark_abs} 
so that the boundary terms at infinity vanish. These conditions are invoked to justify the cancellation of boundary terms when the divergence theorem is applied on unbounded domains. 
\section{Main results}\label{sec:main_result}
We begin with one of our main results concerning the distributional stochastic HJB equations~\eqref{thm_4_1_decomp} and \eqref{eqn4}.

\subsection{Distributional Stochastic HJB Equation}\label{subsec:main:results:Dist:sHJB}
Using Definition~\ref{Def:dist} for the \textit{distributional} value functions $\bar V^g,\bar V$, and Remarks \ref{remark_abs} and Assumption \ref{remark_abs2}, we derive the \emph{distributional} version of the stochastic HJB and the decomposition result in the next Theorem.
\begin{theorem} \label{thm:disint_sHJB}
    Consider Problem \textup{\highlight{(P0)}}, and let Assumptions \ref{regularity_dyn}, \ref{regularity_cost}, \ref{remark_abs2} and \ref{regularity_val} hold. Then, the associated \emph{distributional} value function defined in \eqref{eqn3} can be written as
    \begin{align} \label{thm_4_1_decomp}
        &\bar V^g(t,\mu) = \mathbb{E}_{x\sim\mu} \left[V^g(t,x)  \right],
    \end{align}
     and satisfies (in the viscosity/distributional sense) the following \emph{distributional} stochastic HJB equation:
     \begin{align}
                 \frac{\partial }{\partial t} \bar V^g  (t,\mu) +  \mathbb{E}_{x\sim\mu}  & \left\{\min\limits_{u \in \Ree^m}\left[f(t,x,u)\t\frac{\partial}{\partial x} V^g(t,x) \right. \right. \nonumber \\
        &+ \left. \left.
        \trace\Big(\frac{\sigma_t\sigma_t\t }{2}\frac{\partial^2}{\partial x^2}  V^g(t,x)\Big) +L(t,x,u)\right] \right\} = 0 \label{eqn4},
     \end{align}
     for all $(t,\mu)\in\bar{\mathcal{D}}^g$.
\end{theorem}

\begin{proof}
Let $\hat{\mu}\in \mathcal{P}_{2,\rm sub}(\Ree^n)$ denote the \emph{alive push-forward} of the distribution $\mu\in \mathcal{P}_{2,\rm sub}(\Ree^n)$ under the flow of the controlled dynamics \eqref{eqn1} with feedback control $u\in \mathcal{U}$ and stopping rule $\tau_g$ for some $g\in \mathcal{G}$ over a time increment $h$, that is, $\hat{\mu} = \Tflows{u,g}{h} \mu$,
where $\Tflow{u,g}{h}:\mathcal{P}_{2,\rm sub}(\Ree^n) \lra \mathcal{P}_{2,\rm sub}(\Ree^n)$ is the flow map induced by the dynamics \eqref{eqn1}, and $\#$ denotes the push-forward operator. In other words,
\begin{align*}
\Tflows{u,g}{h}\mu(A) = \mathbb P\bigl(x_h\in A,\tau_g>h \bigr) \quad\text{for all Borel sets } A\subset\Ree^n.  
\end{align*}
Thus, $\hat\mu$ denotes the alive sub-probability measure (cf.\ Subsection~\ref{subs:subprob}). Using the SDE \eqref{eqn1} and the Fokker--Planck equation for the stopped process yields
\begin{align} \label{FP_proof}
\frac{\partial }{\partial t}\rho^a(t,x) = - \nabla_{x} \cdot \left[f(t,x,u_t(x))\rho^a(t,x) - \nabla_x \cdot \left(\frac{\sigma_t\sigma_t\t}{2} \rho^a(t,x)\right) \right],
\end{align} 
for all $(t,x) \in \mathcal D^g$ with the Dirichlet boundary condition $\rho^a(t,x)=0$ for all $(t,x)\in \partial\mathcal D^g$, where $\rho^a$ denotes the alive probability density of the probability measure $\mu^a_t$; see Subsection \ref{subs:subprob}.

From the definition of $\bar{V}^g$ in \eqref{eqn3}, for \blue{any sufficiently small time increment} $h > 0$, it follows that 
\begin{align} 
\bar V^g(t,\mu)  & = \min  \limits_{u \in \mathcal U} \left \{ \mathbb{E}^{\mathcal{F}_t}_{x\sim\mu} \int_t^{t+h} L(s,x_s,u_s(x_s)) \, \d s \, + \Phi({\tau_g},x_{\tau_g}) \indic{\tau_g \le t+h} \right. \nonumber \\ 
    & \left. + \  \mathbb{E}^{\mathcal{F}_{t+h}}_{x\sim \hat{\mu}} \int_{t+h}^{\tau_g} L(s,x_s,u_s(x_s)) \, \d s \, + \Phi({\tau_g},x_{\tau_g}) \indic{\tau_g > t+h}  \right\}, \nonumber \\
    &\hspace{-1em}=\min\limits_{u \in \mathcal U} \mathbb{E}^{\mathcal{F}_t}_{x\sim\mu} \int_t^{t+h} L(s,x_s,u_s(x_s)) \, \d s +  \bar \Phi^g(t+h,\hat \mu)  -\bar \Phi^g(t,\mu)  +\bar V^g(t+h,\hat \mu), \label{bellman_eq1}
\end{align}
where $\bar \Phi^g(t,\mu) \Let \mathbb{E}_{x\sim \mu}^{\mathcal{F}_t} [\Phi(t,x) \indic{\tau_g \le t}]$ is the accumulated terminal cost at time $t$. 

Under the standing assumptions \eqref{A1}--\eqref{A6}, Remark \ref{reg_opt_bound}, the probability mass flux analysis from \eqref{mass_lost_flux} yields, a.e.,
\begin{align} 
    &\frac{\d }{\d t}\bar \Phi^g(t,\mu) = \int_{\partial \, {\rm S}^a_{\mu}} \left\{\Phi (t,x)   [{\rm J}_{\rm out}(t,x,u_t(x))-\rho^a(t,x){\rm v}_{\rm out}(t,x) ]\cdot \hat n(t,x) \right \} \d S(x), \\
&\bar \Phi^g(t+h,\hat \mu)  -\bar \Phi^g(t,\mu) = \int_t^{t+h} \frac{\d }{\d t}\bar \Phi^g(s,\mu_s) \ \d s \label{bellman_phi_eqns},
\end{align}
where $\hat n:\Ree_{\ge 0}\times \Ree^n \lra \Ree^n$ is the outward-pointing unit normal field on $\partial{\rm S}_{\mu_s^a}$, ${\rm v}_{\rm out}:\Ree_{\ge 0}\times\Ree^n \lra \Ree^n$ is the surface velocity vector, and
\begin{align} \label{J_out}
{\rm J}_{\rm out}(t,x,u) \Let f(t,x,u)\rho^a(t,x) - \nabla_x \cdot \left[\frac{\sigma_t\sigma_t\t}{2} \rho^a(t,x) \right].
\end{align}
Taking the functional time derivative of $\bar V^g$ based on the chain rule in the space of measures \cite[Prop. 5.102]{prob_mean_field} and applying the divergence theorem for the time-varying boundary $\partial \, {\rm S}_{\mu}$, leads to the following Reynolds transport-like expression
\begin{align} \label{proof_grad_flow}
\frac{\d }{\d t}\bar V^g(t,\mu) &= \frac{\partial }{\partial t}\bar V^g(t,\mu) + \int_{{\rm S}_{\mu}} \left[\left(\frac{\delta \bar V^g}{\delta \mu}(t,\mu) \right) (t,x)  \frac{\partial }{\partial t}\rho^a(t,x) \right] \d x \nonumber \\
    & + \int_{\partial \, {\rm S}_{\mu}} \left[\left(\frac{\delta \bar V^g}{\delta \mu}(t,\mu) \right) (t,x) \rho^a(t,x)  \big( {\rm v}_{\rm out}(t,x) \cdot \hat{n} (t,x) \big) \right] \d S(x).
\end{align}
Next, we aim to find an expression for $\bar V^g(t+h,\hat \mu)$ to simplify \eqref{bellman_eq1}. We claim that
\begin{align} \label{proof_int1}
\bar V^g(t+h,\hat \mu) = \bar V^g(t,\mu)+ \int_t^{t+h} \frac{\d }{\d t}\bar V^g(s,\mu^a_s) \, \d s.
\end{align}
We remind the reader that the derivations leading to this point are understood in the viscosity sense (see Remark \ref{remark:vis_sense}), i.e., they are carried out on test functionals in the class $C_2^{1,2}$. In particular, with a sufficiently regular stopping boundary $\partial\mathcal D^g$ (see Remark~\ref{reg_opt_bound}), Assumptions \ref{regularity_dyn} and \ref{regularity_cost}, and the regularity of the test functional class $C^{1,2}_2$ (Definition \ref{def:test}), the resulting integrand in \eqref{proof_grad_flow} is integrable on the finite interval $[t,t+h]$ (e.g., bounded on the domain $[t,t+h]$). Consequently, \eqref{proof_int1} follows from the fundamental theorem of calculus (FTC), and it can be interpreted as a generalized Dynkin's Formula \cite{book1} with a given initial distribution. 
Substituting \eqref{proof_grad_flow}, and then \eqref{bellman_eq1} and \eqref{bellman_phi_eqns} in \eqref{proof_int1}, yields
\begin{align} \label{FP_proof1}
&0=\min\limits_{u \in \mathcal U} \mathbb{E}^{\mathcal{F}_t}_{x\sim\mu} \int_t^{t+h} L(s,x_s,u_s(x_s)) \, \d s + \int_t^{t+h}  \frac{\partial }{\partial t}\bar V^g(s,\mu^a_s) \, \d s   \\
&  +  \int_t^{t+h} \int_{{\rm S}_{\mu^a_s}} \left[\left(\frac{\delta \bar V^g}{\delta \mu}(s,\mu^a_s) \right) (s,x)  \frac{\partial }{\partial t}\rho^a(s,x) \right] \d x \, \d s \nonumber \\
&  +  \underbrace{\int_t^{t+h} \int_{\partial \, {\rm S}_{\mu^a_s}} \Bigg[\left(\frac{\delta \bar V^g}{\delta \mu}(s,\mu^a_s) \right) (s,x)  \rho^a(s,x) ({\rm v}_{\rm out}(s,x) \cdot \hat{n}(s,x)) \Bigg] \, \d S(x) \, \d s}_{\varheart} \nonumber \\
&  +  \underbrace{\int_t^{t+h} \int_{\partial \, {\rm S}_{\mu^a_s}} \big\{\Phi (s,x)   [{\rm J}_{\rm out}(s,x,u_s(x)) -\rho^a(s,x){\rm v}_{\rm out}(s,x) ] \cdot \hat{n} (s,x) \big \}  \d S(x) \, \d s.}_{\vardiamond}\nonumber
\end{align}
Plugging $\frac{\partial \rho^a}{\partial t}(\cdot)$ from \eqref{FP_proof} into \eqref{FP_proof1}, and focusing only on the inner integration of the second integrand and applying integration by parts yields
\begin{align} \label{FP_proof2}
\hspace{-1 em}&\int_{{\rm S}_{\mu^a_s}} \left[\left(\frac{\delta \bar V^g}{\delta \mu}(s,\mu^a_s) \right) (s,x)  \frac{\partial }{\partial t}\rho^a(s,x) \right] \d x = \\
\hspace{-1 em}&\int_{{\rm S}_{\mu^a_s}} {\rm J}_{\rm out}(s,x,u_s(x))\t \frac{\partial}{\partial x}\left(\frac{\delta \bar V^g}{\delta \mu}(s,\mu^a_s) \right) (s,x) \,   \, \d x \nonumber \\
\hspace{-1 em} & \underbrace{-\int_{\partial {\rm S}_{\mu^a_s}} ({\rm J}_{\rm out}(s,x,u_s(x)) \cdot \hat n(s,x))\t \left(\frac{\delta \bar V^g}{\delta \mu}(s,\mu^a_s) \right) (s,x)  \, \d S(x)}_{\varclub} \nonumber.
\end{align}
Substituting ${\rm J}_{\rm out}$ from \eqref{J_out}
into \eqref{FP_proof2} and applying integration by parts to the diffusion term, the first term on the right-hand side of \eqref{FP_proof2} becomes
\begin{align} \label{FP_proof3}
    &\int_{{\rm S}_{\mu^a_s}} \left[f(s,x,u_s(x)) \t \frac{\partial}{\partial x} \left(\frac{\delta \bar V^g}{\delta \mu}(s,\mu^a_s) \right) (s,x)  \right. \\ \nonumber
    & \left. \hspace{15mm}+\half \trace\Big( \sigma_s \sigma_s\t \frac{\partial^2}{\partial x^2} \left(\frac{\delta \bar V^g}{\delta \mu}(s,\mu^a_s) \right) (s,x) \Big) \right] \rho^a(s,x) \,\d x,
\end{align}
where we employed the vanishing-tail conditions for the boundary at infinity, and Dirichlet boundary condition for the stopping boundary.
Using \eqref{FP_proof3} inside \eqref{FP_proof1} and the fact that the stopped trajectories only incur a terminal cost, i.e, $\frac{\delta \bar V^g}{\delta \mu}(t,\mu_t)(t,x) = \Phi(t,x)$ when $x\in {\rm S}_{\mu^a_t}$,
the fourth $\varheart$, fifth $\vardiamond$ terms from \eqref{FP_proof1} and the second term $\varclub$ from \eqref{FP_proof2}, cancel each other out, and finally, \eqref{FP_proof1} yields
\begin{align} \label{FP_proof4}
0 & = \min_{u \in \mathcal U}  \Bigg \{ \mathbb{E}^{\mathcal{F}_t}_{x\sim\mu}\int_t^{t+h} \Bigg[\frac{\partial }{\partial t}\bar V^g(s,\mu^a_s) + f(s,x_s,u_s(x_s)) \t \frac{\partial}{\partial x} \left(\frac{\delta \bar V^g}{\delta \mu}(s,\mu^a_s) \right) (s,x_s)  \Bigg. \Bigg. \nonumber \\ 
    & +\Bigg. \Bigg. \half \trace\Big( \sigma_s \sigma_s\t \frac{\partial^2}{\partial x^2} \left(\frac{\delta \bar V^g}{\delta \mu}(s,\mu^a_s) \right) (s,x_s) \Big) + L(s,x_s,u_s(x_s))  \Bigg] \, \d s  \Bigg \}.
\end{align}
The sHJB equation for the \textit{distributional} value function $\bar V^g$ is obtained from \eqref{FP_proof4} by dividing by $h$ and taking the limit as $h \ra 0$.
Since the integrand in \eqref{FP_proof4}
is measurable in its arguments and is locally integrable as a function of time, it converges to its pointwise value for almost every $t$ by Lebesgue–Besicovitch Differentiation Theorem \cite[Thm. 1.32]{evans_measure}, yielding the following compact form
\begin{align} \label{compact_int_proof}
    \min\limits_{u\in \Pi} \ \aset[\Big]{\bar{\mathcal{L}}_u \bar V^g(t,\mu)+ \mathbb{E}_{t,x\sim\mu}L(t,x,u(x))} = 0, \quad \text{for a.e.} \ t,
\end{align}
where in \eqref{compact_int_proof}:
\begin{enumerate}[leftmargin=*,label=\alph*)]
 \item $\bar{\mathcal{L}}_u$ is the \textit{distributional} second-order generator \cite{vis_comp} defined for a test function $\Xi \in C^{1,2}_2$ and feedback controller $u\in \Pi$ by, 
\begin{align} \label{dist_sHJB-generator}
    \bar{\mathcal{L}}_u \Xi(t,\mu) & \Let \frac{\partial }{\partial t} \Xi(t,\mu) + \mathbb{E}_{t,x\sim \mu}\Bigg [f(t,x,u(x))^{\t} \frac{\partial}{\partial x} \left(\frac{\delta \Xi}{\delta \mu}(t,\mu) \right)(t,x) \Bigg. \nonumber \\
    & \Bigg. + \half \trace\Big( \sigma_t \sigma_t\t \frac{\partial^2}{\partial x^2} \left(\frac{\delta \Xi}{\delta \mu}(t,\mu) \right) (t,x) \Big) \Bigg];
\end{align}
\item the conditioning on the filtration ${\mathcal{F}_t}$ is dropped since the limiting equation depends only on the deterministic variables $(t,\mu)$, and we also dropped $t$ from $\mu_t$ since $t$ here is arbitrary.
\end{enumerate}
By defining
\begin{align} \label{dist_V_der}
(t,x) \mapsto V^{g,\dagger}_{\mu}(t,x) \Let \left(\frac{\delta \bar V^g}{\delta \mu}(t,\mu)\right)(t,x),
\end{align}
from \eqref{compact_int_proof}, we have
\begin{align} \label{b4_RW}
\frac{\partial}{\partial t} \bar V^g(t,\mu) 
&= -\min\limits_{u \in \Pi} \left \{ \mathbb{E}_{t,x \sim \mu}\left[f(t,x,u(x))\t\frac{\partial}{\partial x} V^{g,\dagger}_{\mu}(t,x)  \right. \right.\nonumber \\
&  \hspace*{15mm} \left. \left. +\trace\Big(\frac{\sigma_t\sigma_t\t }{2}\frac{\partial^2}{\partial x^2}  V^{g,\dagger}_{\mu}(t,x)\Big) + L(t,x,u) \right] \right \}.
\end{align}
Since the minimization is taken over Borel measurable feedback maps $u_t\in \Pi$ at an arbitrary $t\ge 0$, and by the standing Assumptions \eqref{A1}-\eqref{A6} and 
Assumption~\ref{regularity_val}, the integrand \eqref{b4_RW} inside the expectation operator is Carathéodory, i.e., measurable in $t$ for each fixed $(x,u)$ and jointly continuous in both $(x,u)$ for each fixed $t$, hence a normal integrand \cite[Ch. 14. D.]{rockafellar}, we invoke the Rockafellar-Wets interchange theorem \cite[Theorem 14.60]{rockafellar} to interchange the minimization and expectation.
It follows that
\begin{align} \label{sHJB_dist}
\frac{\partial }{\partial t}\bar V^g(t,\mu) 
    & =-\mathbb{E}_{t,x \sim \mu}\left \{ \min\limits_{u \in \Ree^m}\left[f(t,x,u)\t\frac{\partial}{\partial x} V^{g,\dagger}_{\mu}(t,x)  \right. \right.\nonumber \\
    & \hspace*{15mm}  \left. \left. +\trace\Big(\frac{\sigma_t\sigma_t\t }{2}\frac{\partial^2}{\partial x^2}  V^{g,\dagger}_{\mu}(t,x)\Big) + L(t,x,u) \right] \right \},
\end{align}
which is the stochastic Hamilton–Jacobi–Bellman (sHJB) equation satisfied by the function  $ \mathbb{E}_{t,x\sim\mu} \left[V^g(t,x)  \right]$ with $V^{g,\dagger}_{\mu}=V^g$. 
This can be seen by taking $\mathbb{E}_{t,x\sim\mu} \left[V^g(t,x)  \right]$ in the place of $\bar V^g$ from \eqref{sHJB_dist}, whose functional derivative from \eqref{dist_V_der} is $V^g$, and then comparing with the sHJB, that is, by substituting $V^g$ into \eqref{eqn:value_func} and taking the expectation with respect to $x \sim \mu$. 

Since both $\bar V^g(t,\mu)$ and $\mathbb{E}_{t,x\sim\mu} \left[V^g(t,x)  \right]$ satisfy the same sHJB, obtained after performing the pointwise minimization over $u\in \Ree^m$, then under the conditions ensuring uniqueness of solutions to this reduced sHJB, i.e., \eqref{A1}-\eqref{A6} together with 
Assumption~\ref{regularity_val}, from the Comparison Theorem \cite[Theorem 3.13]{vis_comp} (alternatively, when $\bar V^g$ is the classical solution to \eqref{sHJB_dist}, uniqueness follows from the Verification Theorem \cite[Theorem 4.5]{class_ver}), we conclude that,
\begin{align}
    \bar V^g(t,\mu) = \mathbb{E}_{t,x\sim\mu} \left[V^g(t,x)  \right],
\end{align}
which completes the proof.
\end{proof}

In this section, we discussed the uniqueness of the value function associated with the stochastic optimal control and stopping problem \eqref{eqn1}, \eqref{eqn2} together with the required regularity assumptions, and established the decomposition principle together with the associated distributional sHJB through Theorem \ref{thm:disint_sHJB}. 
In the following section, our main goal is to rewrite this problem in an equivalent deterministic, nonlocal continuity form, which will enable a calculus of variations analysis over the space of probability measures.
\subsection{Fokker--Planck Transformation}\label{subsec:main:results:FP:transform}

Consider the controlled SDE \eqref{eqn1} and its associated Fokker--Planck equation
\begin{align} \label{FP_general}
\frac{\partial }{\partial t}\rho(t,x) = - \nabla_{x} \cdot \left[f(t,x,u_t(x))\rho(t,x) - \nabla_x \cdot \left(\frac{\sigma_t\sigma_t\t}{2} \rho(t,x)\right) \right],
\end{align}
subject to $\rho(0,x) \Let \rho_0(x)$,
where $\rho(t,\cdot)\in C^{2+\alpha}(\Ree^n)$ for some $\alpha\in(0,1)$ and $\mu_0(x)$ such that $\d \mu_0(x) = \rho_0(x) \, \d x$, and $u\in \mathcal{U}$ is an admissible feedback controller. 

Since $\sigma_t$ in \eqref{eqn1} is state-independent, we can define the transformed drift $\tilde{f}$, also called the velocity field, by
\begin{align} \label{vel_field}
(t,x,u) \mapsto \tilde{f}(t,x,u) \Let f(t,x,u) - \frac{\sigma_t\sigma_t\t }{2}\nabla_{x} \log{\rho(t,x)} \in \Ree^n,
\end{align}
where $(t,x) \mapsto \nabla_{x} \log{\rho(t,x)}$ is the (Stein) score function. Using \eqref{vel_field}, we rewrite \eqref{FP_general} as the Liouville equation
\begin{align} \label{trans_eq}
\frac{\partial }{\partial t}\rho(t,x) = - \nabla_{x} \cdot \left[\tilde{f}(t,x,u_t(x))\rho(t,x)\right] \quad\text{for all }(t,x) \in \Ree_{\ge 0}\times \Ree^n,
\end{align} 
which is also known as the transport equation. 
This interpretation has been an important tool in many recent advancements in gradient flows, optimal transport, and geodesics in Wasserstein spaces; we refer the interested reader to \cite{book_ge3}. 
For some recent results on existence and uniqueness, see Filippov's and Relaxation Theorems \cite{wass_filippov}. This transformation will be called the \emph{Fokker-Planck Transformation}.

We highlight the fact that \eqref{trans_eq} is \emph{not} a linear transport equation. 
However, since $\tilde f$ also depends on $\rho$ itself, \eqref{trans_eq} is equivalent to the uniformly parabolic Fokker--Planck equation~\eqref{FP_general}.

\begin{remark}
We employ the Fokker--Planck transformation at the level of marginal density evolution. Specifically, if \(\rho\) denotes the (alive) density associated with the SDE \eqref{eqn1} and $\tilde f$ is defined by \eqref{vel_field}, then the continuity (Liouville) equation \eqref{trans_eq} is an equivalent rewriting of the Fokker--Planck equation, and hence yields the same one-time marginals $\{\mu_t\}_{t\ge 0}$. 
Consequently, the characteristic ODE $\d x_t = \tilde f (t,x_t,u_t(x_t)) \, \d t$ of \eqref{trans_eq} should not be interpreted as generating sample paths of \eqref{eqn1}; rather, it provides a deterministic characteristic representation of the marginal law (density) evolution.
\end{remark}

Conversely, one may start from a deterministic continuity equation whose velocity field depends explicitly on the score function and recover an SDE that induces the same density evolution. 
We refer to this inverse construction as the \emph{inverse Fokker--Planck Transformation}. 
In the present setting, applying the inverse Fokker--Planck Transformation to the deterministic dynamics generated by \eqref{vel_field} recovers the controlled SDE \eqref{eqn1}.

Next, we show that It\^o's formula is recovered at the level of the Fokker--Planck equation after the Fokker-Planck transformation.


\begin{lemma}[inverse Fokker--Planck Recovery of It\^o's formula]
Consider the controlled SDE \eqref{eqn1} and its associated Fokker-Planck equation \eqref{FP_general}. 
Let the corresponding Fokker--Planck transformed deterministic system be  $\d x_t=\tilde f(t,x_t,u_t(x_t)) \,\d t$, where the transformed velocity field $\tilde f$ is given by \eqref{vel_field}, for some $u\in\mathcal U$ and every $t\in\Ree_{\ge 0}$. \blue{Furthermore, let $C^{1,2}\ni h:\Ree_{\ge 0}\times\Ree^n \lra \Ree^n$, and assume that, for each $t\in \Ree_{\ge 0}$, the map $x \mapsto h(t,x)$ is a $C^2$ diffeomorphism onto its image.} For each \(t\), let $\varpi(t,\cdot)$ be the density of \(y_t\).

Then $\varpi$ satisfies the Fokker--Planck equation associated with the SDE obtained by applying It\^o's formula to $y_t=h(t,x_t)$ along the original controlled SDE \eqref{eqn1}. Equivalently, applying the inverse Fokker--Planck transformation to $y$ recovers precisely the It\^o-transformed SDE.
\end{lemma}

\begin{proof}
Since for each fixed $t\in \Ree_{\ge 0}$, $y\mapsto \varpi(t,y)$ is the push-forward density of $x \mapsto \rho(t,x)$ under the map $x \mapsto h(t,x)$, \blue{the change of variables} formula yields,
\begin{align}
    \rho(t,x_t)=\varpi(t,h(t,x_t)) \left|\det \frac{\partial}{\partial x} h(t,x_t)\right| \label{eq:rho_varpi}.
\end{align}

Differentiating \eqref{eq:rho_varpi} with respect to $x$, leads to
\begin{align} \label{eq:gradlog}
\nabla_x \log \rho(t,x_t) =
\frac{\partial}{\partial x} h(t,x_t)\t  \nabla_y \log \varpi(t,y_t)\big|_{y_t=h(t,x_t)} + \nabla_x \log \left|\det \frac{\partial}{\partial x} h(t,x_t)\right|.
\end{align}

Now, let $y_t=h(t,x_t)$, with $x_t$ evolving according to the Fokker--Planck transformed velocity field \eqref{vel_field}. By the chain rule,
\begin{align} \label{eq:ito_rev_lem_1}
\dot y_t &= \frac{\partial}{\partial t}h(t,x_t)+\frac{\partial}{\partial x} h(t,x_t)\dot x_t \nonumber\\
&= \frac{\partial}{\partial t}h(t,x_t) + \frac{\partial}{\partial x} h(t,x_t)f(t,x_t,u_t(x_t))
-\frac12 \frac{\partial}{\partial x} h(t,x_t)\sigma_t\sigma_t \t \nabla_x \log \rho(t,x_t).
\end{align}

Substituting \eqref{eq:gradlog} into \eqref{eq:ito_rev_lem_1} yields
\begin{align} \label{eq:ito_rev_lem_2}
\dot y_t =& \frac{\partial}{\partial t}h(t,x_t) + \frac{\partial}{\partial x} h(t,x_t)f(t,x_t,u_t(x_t))  \nn\\
&- \frac12 \frac{\partial}{\partial x} h(t,x_t)\sigma_t\sigma_t \t \frac{\partial}{\partial x} h(t,x_t)\t \nabla_y \log \varpi(t,y_t) \nn \\
&- \frac12 \frac{\partial}{\partial x} h(t,x_t)\sigma_t\sigma_t\t \nabla_x \log \left|\det \frac{\partial}{\partial x} h(t,x_t)\right|.
\end{align}

Define $M(t,y_t)\Let \frac{\partial}{\partial x} h(t,x_t)\sigma_t\sigma_t \t \frac{\partial}{\partial x} h(t,x_t) \t$ with $y_t=h(t,x_t)$. Using the chain rule, the inverse Jacobian relation $\nabla_y={\left(\frac{\partial}{\partial x}h(t,x_t)\right)\t}^{-1}\nabla_x$, and Jacobi's formula, one obtains,
\begin{align*}
    \nabla_y \cdot M(t,y_t) =
\tr \left(\sigma_t\sigma_t \t \frac{\partial^2}{\partial x^2}h(t,x_t)\right)
+
\frac{\partial}{\partial x} h(t,x_t)\sigma_t\sigma_t \t \nabla_x \log \left|\det \frac{\partial}{\partial x} h(t,x_t)\right|,
\end{align*}
where the trace is understood componentwise for each entry of $h$, i.e.,
\begin{align*}
    \tr \left(\sigma_t\sigma_t\t \frac{\partial^2}{\partial x^2}h(t,x_t)\right) \Let 
    \begin{bmatrix}
    \tr \left(\sigma_t\sigma_t\t \frac{\partial^2}{\partial x^2}h_1(t,x_t)\right)\\
    \vdots \\
    \tr \left(\sigma_t\sigma_t\t \frac{\partial^2}{\partial x^2}h_n(t,x_t)\right)
    \end{bmatrix}.
\end{align*}

Hence, \eqref{eq:ito_rev_lem_2} becomes,
\begin{align} \label{eq:ito_rev_lem_3}
\dot y_t &= \frac{\partial}{\partial t}h(t,x_t) + \frac{\partial}{\partial x} h(t,x_t)f(t,x_t,u_t(x_t)) + \frac12 \tr \big(\sigma_t\sigma_t \t \frac{\partial^2}{\partial x^2}h (t,x_t)\big)\nn\\
&\quad-\frac12 \nabla_y \cdot M(t,y_t)-\frac12 M(t,y_t)\nabla_y \log \varpi(t,y_t).
\end{align}

Since the multiplier of the term $\nabla_{y} \log{\varpi(t,y_t)}$ in \eqref{eq:ito_rev_lem_3} depends explicitly on $y_t = h(t,x_t)$, we invoke the \emph{inverse Fokker-Planck Transformation}. To do so, first observe that,
\begin{align*}
    \left(-\frac12 \nabla_y \cdot M(t,y_t)-\frac12 M(t,y_t)\nabla_y  \log \varpi(t,y_t) \right)\varpi(t,y_t) \\
     = - \frac{1}{2} \nabla_{y} \cdot \left(\frac{\partial }{\partial x}h(t,x_t) \sigma_t\sigma_t\t \frac{\partial }{\partial x}h(t,x_t) \t \varpi(t,y_t) \right).
\end{align*}

As a result, the Fokker-Planck equation for $\varpi$ is
\begin{alignat}{2} \label{eq:ito_rev_lem_4}
\frac{\partial }{\partial t}\varpi(t,y_t) 
                    &= - \nabla_{y} \cdot 
                                       &&\Bigg[\left(\frac{\partial }{\partial t}h(t,x_t)+\frac{\partial }{\partial x}h(t,x_t) f(t,x,u_t(x_t))+\frac12 \tr \big(\sigma_t\sigma_t \t \frac{\partial^2}{\partial x^2}h (t,x_t)\big) \right) \varpi(t,y_t)  \nonumber \\
                    &                  && - \frac{1}{2} \nabla_{y} \cdot                                        \left(\frac{\partial }{\partial x}h(t,x_t)                                  \sigma_t\sigma_t\t \frac{\partial }{\partial x}h(t,x_t) \t                      \varpi(t,y_t) \right)\Bigg],
\end{alignat}
which is also the Fokker-Planck equation associated with the following SDE,
\begin{align} \label{eq:ito_rev_lem_fin}
\d z_t &= \left( \frac{\partial}{\partial t}h(t,x_t) + \frac{\partial}{\partial x}h(t,x_t)f(t,x_t,u_t(x_t)) + \frac12 \tr \big(\sigma_t\sigma_t \t \frac{\partial^2 }{\partial x^2}h(t,x_t)\big) \right) \d t \nn \\ 
&\quad + \frac{\partial}{\partial x}h(t,x_t)\sigma_t\,\d w_t.
\end{align}

Notice that \eqref{eq:ito_rev_lem_fin} is exactly the SDE obtained by applying It\^o's formula to $y_t=h(t,x_t)$. Hence, the inverse Fokker--Planck transformation recovers precisely the It\^o-transformed dynamics.
\end{proof}

Next, in analogy with the stopping boundary family $\mathcal{G}$, we define the set of \textit{final-time assignment functions} by
\begin{align} \label{fpt_set_tau}
\mathcal{T} \Let \aset[\big]{ \taum:\Ree^n \ra \Ree_{\ge 0} \suchthat \taum \in C^1(\Ree^n)},
\end{align}
that is, $\mathcal{T}$ is the set of nonnegative, continuously differentiable functions assigning a final time to each initial condition $x_0 \in \mathbb R^n$. 
Specifically, the value $\tau_m(x_0)$ denotes the prescribed final time associated with the trajectory with the initial condition $x_0$. Therefore, $\mathcal T$ is the class of initial-condition-dependent time-assignment maps.

In what follows, for $\tau_m\in\mathcal T$, we do not attempt to characterize whether the induced alive set is a sufficiently regular time-dependent space–time domain. Instead, we restrict attention to \emph{admissible} time-assignment functions, $\mathcal{T}_{\rm ad} \subset \mathcal{T}$ for which the alive domain and Dirichlet problem are well posed for the variational analysis carried out in this paper. For further analysis, we define the following sets.

\begin{definition} \label{fpt_admis_tau}
Let $\taum \in \mathcal{T}_{\rm ad}$, and denote by $\mathcal D^\taum$ the admissible alive space--time domain associated with $\taum$ and by $\mathcal E^\taum$ the corresponding stopping surface. For each $t\ge 0$, denote by $\mathcal D_t^\taum$ the time-slice of $\mathcal D^\taum$ and by $\partial \mathcal D_t^\taum$ the corresponding time-sliced boundary.
We define these sets as follows:
\begin{align}
\mathcal D^\taum &\Let \aset[\big]{(t,x)\in[0,+\infty)\times\Ree^n \suchthat \exists\, x_0 \in \Ree^n  \text{ s.t.}, \ x=x_t, \, 0\le t < \taum(x_0)}, \nn\\
\mathcal D_t^\taum &\Let \aset[\big]{x \in\Ree^n \suchthat \exists\, x_0 \in \Ree^n,  \text{ s.t.}, \ x=x_t, \, 0\le t < \taum(x_0)}, \nn\\
\partial \mathcal D_t^\taum &\Let \aset[\big]{x\in \Ree^n \suchthat \exists \ x_0 \in \Ree^n,\text{ s.t.}, \ x=x_t, \ t = \taum(x_0)}, \nn \\
\mathcal E^\taum &\Let \aset[\big]{(t,x)\in[0,+\infty)\times\Ree^n \suchthat \exists \ x_0 \in \Ree^n, \text{ s.t.}, \ x=x_t, \  t = \taum(x_0)}.
\end{align}
\end{definition}
Given an admissible stopping rule $\taum \in \mathcal{T}_{\rm ad}$, $\rho$ is understood as the solution to the parabolic Dirichlet Fokker--Planck problem \eqref{FP_general} on the alive domain $\mathcal D^{\taum}$, and the induced field $\tilde{f}$ is then defined a posteriori from this $\rho$. For the well-posedness of the problem formulation, in the spirit of our Remark \ref{reg_opt_bound} on the regularity of the stopping boundary class $\mathcal{G}$, we impose the following assumption.

\begin{assumption}[Regularity on $\mathcal D^\taum$]
\label{ass_reg_tau}
Let $\taum \in \mathcal{T}_{\rm ad}$. Throughout the analysis in \S\ref{subsec:DetStoc:rel}, \S\ref{var_analysis_tr_pr}, we assume that the associated stopping surface $\mathcal E^\taum$ is a $C^1$
hypersurface in $\Ree_{\ge 0}\times\Ree^n$. In particular, $\mathcal E^\taum$ is local Lipschitz regularity (cf.\ the stopping boundary class $\mathcal G$ in Section~\ref{sect:back_theory}). Moreover, for any truncation
\begin{align*}
    \mathcal D^{\taum}_{a,b}\Let\mathcal D^{\taum}\cap\big([0,a]\times\aset[]{x\in\Ree^n \suchthat \|x\|<b}\big),
\end{align*}
with $a,b>0$, we assume that $\partial\mathcal D^{\taum}_{a,b}$ is a Lipschitz boundary. Hence, the divergence theorem is applicable on each truncated domain $\mathcal D^{\taum}_{a,b}$, and the corresponding boundary integrals are well-defined.
\end{assumption}
\begin{remark}
\label{remark_tau_ibp}
Assumption \ref{ass_reg_tau} is a regularity hypothesis that will be imposed only to justify the later use of integration by parts and the divergence theorem on the domain for the alive trajectories, such that, under the admissible perturbations of $\taum$, we can differentiate relevant domain integrals. We do not address when such regularity holds,
as it is, as pointed out earlier in Remark \ref{reg_opt_bound}, is heavily problem dependent.
\end{remark}

\begin{assumption} \label{Assumption_5}
    Let $\alpha \in(0,1)$, and for each admissible $\taum \in \mathcal{T}_{\rm ad}$, assume that the associated stopping boundary $\mathcal{E}^{\taum}$ belongs to the class $C^{1+\alpha/2, \,2+\alpha}$ (parabolic H\"older class of order $2+\alpha$). 
    We assume that there exists a unique nonnegative classical solution, $\rho \in C^{1+\alpha/2,\,2+\alpha}_{\rm loc}(\mathcal{D}^{\taum})$, for the uniformly parabolic Dirichlet problem defined by the PDE \eqref{FP_general} (equivalently, by its transport-form \eqref{trans_eq}) on the domain $\mathcal{D}^{\taum}$, Dirichlet boundary condition $\rho(t,x)=0$ for all $(t,x) \in \mathcal{E}^{\taum}$ and the initial condition $\rho(0,\cdot)=\rho_0(\cdot)$. 
    Furthermore, for any compact set $\mathcal{X}$ that is compactly contained in  $\mathcal{D}^{\taum}$, the functions $\rho(t,x), \nabla_x \rho(t,x), \nabla_{xx} \rho(t,x)$ are assumed to be integrable and bounded, and $\rho(t,x)>0$ for all $(t,x)\in \mathcal{X}$.
\end{assumption}

Assumption~\ref{Assumption_5} is imposed as the standing regularity hypothesis for the subsequent analysis. As pointed out in Remark~\ref{remark_tau_ibp}, verifying such regularity for a given admissible time-assignment function $\tau_m \in \mathcal{T}_{\rm ad}$ depends strongly on the particular problem, hence we do not attempt a detailed characterization here. Rather, the following remark is included only to provide a standard sufficient set of stronger assumptions under which Assumption~\ref{Assumption_5} follows from a classical Schauder-type parabolic regularity theory. In this sense, Remark~\ref{Remark_4_Asmp5} should be viewed as a justification for the plausibility of Assumption~\ref{Assumption_5}, and not as an additional standing hypothesis.

\begin{remark}[Classical regularity under stronger assumptions] \label{Remark_4_Asmp5}
Assume, in addition to Assumption \ref{Assumption_5}, that $\rho_0 \in C^{2+\alpha}(\mathcal{D}_0^\taum)$ satisfies the compatibility condition, i.e., $\rho_0(x_0)=0$ for any $x_0 \in \partial \mathcal{D}_0^\taum$ and $\Pi \ni u_t: (t,x)\mapsto \Ree^m$ is also H\"older continuous in $t,x$. Since $\sigma_t \sigma_t\t$ is a uniformly elliptic matrix for all $t\ge 0$, and $f$ is H\"older continuous in $t,x,u_t$ under the standing Assumption \ref{regularity_dyn}, then Lieberman’s Schauder (regularity) theory \cite[Theorem 5.14]{lieberman} for uniformly parabolic Cauchy-Dirichlet problems on $C^{1+\alpha/2,2+\alpha}$
parabolic boundaries, guarantees the existence and uniqueness of the solution $\rho \in C^{1+\alpha/2,\,2+\alpha}_{\rm loc}(\mathcal{D}^\taum)$. Boundary gradient estimates near the (parabolic) boundary are also available under additional geometric assumptions on the Dirichlet problem. For more information, see Lieberman \cite[Theorem 10.4]{lieberman} for up-to-boundary boundedness of gradients. Interior gradient bounds on compact subsets of the domain follow from Lieberman \cite[Theorem 11.3(c)]{lieberman}.
\end{remark}

Based on Assumption \ref{Assumption_5}, the score function is continuously differentiable, and together with the Assumption \ref{regularity_dyn}, a unique Lagrangian flow associated with the time-dependent velocity field \eqref{vel_field}, i.e., for $\mu_0$-a.e. (almost everywhere) initial condition $x_0$, a unique solution to the ODE $\d x_t = \tilde f (t,x_t,u_t(x_t)) \d t$ exists, and is well-defined globally; see, e.g., \cite[Lemma 8.1.4]{book_gr} and \cite{book_gr2,book_ge3}.

\begin{remark}[Score field near the stopping boundary]
Since the alive density, i.e., $\rho^a(\cdot)$ satisfies an absorbing (Dirichlet) boundary condition on $\partial \mathcal{D}^\taum$ (Assumption \ref{remark_abs2}) for some admissible $\taum \in \mathcal{T}_{\rm ad}$, the score term $\nabla_x \log \rho^a(\cdot)$ typically becomes unbounded near the boundary $\partial \mathcal{D}^\taum$. We interpret all the score-related expressions on compact subsets $\mathcal{D}^{\taum,\delta}$ that are strictly inside $\mathcal{D}^\taum$ (excluding a $\delta$-boundary layer near the stopping boundary $\partial \mathcal{D}^\taum$),
\begin{align*}
    \mathcal{D}^{\taum,\delta} 
\Let \aset[\big]{(t,x)\in \mathcal{D}^{\taum} \suchthat  \inf_{y\in \partial \mathcal{D}_t^{\taum}} \|x-y\|\ge \delta},
\end{align*}
for some $\delta >0$, and are thus well-defined and bounded in that region.
\end{remark}

Using the Fokker--Planck transformation, as in \eqref{vel_field}, we are interested in a deterministic representation of the original stochastic optimal control problem for the given SDE \eqref{eqn1} and the objective function \eqref{eqn2}. To this end, we leverage the new deterministic dynamics, consisting of \eqref{FP_general} and \eqref{vel_field}, and define the problem \highlight{(P1)} by:
\begin{mdframed}[backgroundcolor=black!10,skipabove=4pt, linewidth=1pt, innertopmargin=4pt]
\begin{align}
&\min_{\taum\in\mathcal{T}_{\rm ad},u\in \mathcal{U}} \mathbb{E}_{0,x \sim \mu_0}\left[\int_0^{\taum(x_0)} L(s,x_s,u_s(x_s)) \, \d s 
+ \Phi({\taum(x_0)},x_{\taum(x_0)})\right], \tag{$\text{(P1a)}$} \label{eqn:P1} \nonumber\\
\text{subject to} & \nonumber\\
& \quad \d x_t = \tilde f(t, x_t, u_t(x_t)) \, \d t \quad \text{with }x_0\sim\mu_0 \tag{$\text{(P1b)}$} \label{eqn:P2} \nonumber.
\end{align}
\end{mdframed}
In \highlight{(P1)}, we have the following data: 
\begin{enumerate}[label={\textup{(P1-\alph*)}}, leftmargin=*, widest=b, align=left]
\item $\mu_0\in \mathcal{P}_2(\Ree^n)$ is the initial probability measure; 

\item we denote the alive probability measure at time $t$ as $\mu_t^a\in \mathcal{P}_{2, \rm sub}(\Ree^n)$, which satisfies $\d \mu^a_t(x) = \rho^a(t,x)\d x$, and propagates according to 
\[
\frac{\partial }{\partial t}\rho^a(t, x) + \nabla_x \cdot \left[\tilde{f}(t,x,u_t(x))\rho^a(t,x)\right] = 0,
\]
with $x\in \mathcal D^\taum_t, t \in \mathbb{R}_{\ge0}$ and $u_t=u(t,\cdot)$, $u\in\mathcal{U}$; together with the absorbing (Dirichlet) boundary condition
\begin{align*}
    \rho^a(t,x)=0, \quad \forall (t,x)\in  \mathcal{E}^{\taum}, 
\end{align*}

\item we emphasize the fact that the drift term $\Ree_{\ge0} \times \R[n] \times \R[m] \ni (t,x,u) \mapsto \tilde f(t,x,u) \in \R[n]$ is nonlocal, i.e., it depends on the current probability measure $\mu_t$ at $t\ge 0$;

\item and therefore, the constraint \eqref{eqn:P2} is understood in McKean-Vlasov/nonlocal sense, and is associated with the density evolution that is induced by \eqref{eqn:P2} and admissible stopping rule (the Dirichlet boundary condition for the Fokker--Planck equation \eqref{trans_eq}) $\taum\in\mathcal{T}_{\rm ad}$.
\end{enumerate}

We now provide our second main result.

\subsection{Deterministic \& Stochastic Relationships}\label{subsec:DetStoc:rel}
A key result used throughout this paper is Stein's identity, which we state in the following lemma. For a detailed proof, we refer the reader to \cite{stein1,stein2}.

\begin{lemma}[Stein's Identity]\label{lem_stein}
    Let $\zeta: \Ree^n \to \Ree^k$ be a differentiable multivariate test function that takes state $x$ as an input, and let $\d \mu(x)=\rho(x)\,\d x$ with $\mu\in \mathcal{P}_{\rm sub}(\Ree^n)$, and $\rho(x)$ is the associated probability density function. 
    Assuming that $\rho(x)\zeta(x)|_{\partial{\rm S}_{\mu}} = 0$, or $\lim_{\lVert x \rVert\to +\infty} \rho(x)\zeta(x) = 0$ when ${\rm S}_{\mu} = \Ree^n$, then 
    \begin{align}
        \int_{{\rm S}_{\mu}} \big(\zeta(x) \nabla_x \log \rho(x)\t + \nabla_x \zeta(x) \big) \, \d \mu = 0.
    \end{align}
\end{lemma}


\begin{definition}
We define the \emph{pointwise} and \emph{distributional} value functions for the deterministic optimal control problem (P1) by
\begin{align}
     V_d^{\tf,\mu}(t,x) &\Let \min_{u\in \mathcal{U}} \quad\int_t^{\tf} L(s, x_s, u_s(x_s)) \, \d s  
     + \,\Phi(\tf, x_{\tf})\label{relation_d_ll}, \\
    \bar V_d^{\taum}(t,\mu) &\Let\mathbb{E}_{t,x\sim\mu}  \left[V_d^{\taum(x),\mu}(t,x)  \right], \label{relation_ll} \\
    \bar V_d(t,\mu) &\Let \min_{\taum\in\mathcal{T}_{\rm ad}} \bar V_d^{\taum}(t,\mu) \label{relation_lll},
\end{align}
with $V_d^{\tf,\mu}(t,x)$ denoting the value function of the auxiliary deterministic control problem obtained from (\textbf{P1}) by prescribing the time-indexed family of \blue{sub-probability measures $\{\mu_s\}_{s\ge t}$ and fixing the terminal time at $\tf$, with the state initialized at $x$ at time $t$.} The prescribed measure family is held fixed while the control is optimized.

\end{definition}

Building on the discussion in Subsection~\ref{subsec:main:results:FP:transform}, we next show that the transformed problem (\textbf{P1}) satisfies the same optimality conditions as the stochastic optimal control and stopping problem (\textbf{P0}). 
In particular, the following theorem shows that the \emph{distributional} value function for the transformed problem (\textbf{P1}) satisfies the same HJB equation, namely the \textit{distributional} sHJB \eqref{eqn4}.

\begin{theorem}[Relation between value functions] \label{thm_fp_dshjb}
Consider Problem (\textbf{P1}) and suppose that Assumptions \ref{regularity_dyn}, \ref{regularity_cost}, \ref{ass_reg_tau} and \ref{Assumption_5} hold. Then, for any $\taum\in\mathcal{T}_{\rm ad}$, the function $\bar V_d^\taum(\cdot,\cdot)$ \eqref{relation_ll} satisfies the \textit{distributional} sHJB equation \eqref{eqn4}.
\end{theorem}

\begin{proof}
The deterministic optimal control problem (P1), with the deterministic initial condition $x$, some fixed probability field $(t,x)\mapsto \rho(t,x)$ for the exogenously given sub-probability measure $\mu \in \mathcal P_{2,\rm sub}(\Ree^n)$ and fixed final time $\tf>0$, satisfies the following HJB equation for $V_d^{\tf,\mu}(t,x)$ at time $t>0$, given in \eqref{relation_d_ll},
\begin{align}
    &\frac{\partial}{\partial t} V_d^{\tf,\mu}(t,x) + \min\limits_{u\in \Pi} \left \{f(t,x,u(x))\t \frac{\partial}{\partial x}V_d^{\tf,\mu}(t,x) \right.  \\
    &\hspace*{10mm} \left. -\left[\frac{\sigma_t\sigma_t\t }{2}\nabla_{x} \log \rho(t,x) \right]\t\frac{\partial}{\partial x}V_d^{\tf,\mu}(t,x)+ L(t,x,u(x))\right \} = 0. \nonumber
\end{align}
Recall the definition \eqref{fpt_set_tau} and the discussion thereafter. Let the admissible time assignment function to be denoted as $\taum \in \mathcal{T}_{\rm ad}$, with $\taum(x_0)=\tf$. Taking the expectation of both sides over a given measure $\mu$, with each trajectory having possibly different final times, and using Stein's identity from Lemma \ref{lem_stein}, leads to
\begin{align} \label{eqn5}
    &\frac{\partial}{\partial t} \left[\mathbb{E}_{t,x \sim \mu}  V_d^{\taum(x),\mu}(t,x) \right] +\mathbb{E}_{t,x\sim\mu} \left\{\min\limits_{u\in \Ree^m}\left[f(t,x,u)\t\frac{\partial}{\partial x} V_d^{\taum(x),\mu}(t,x)\right. \right. \nonumber \\
    &\hspace*{10mm} + \left. \left. \trace\left(\frac{\sigma_t\sigma_t\t }{2}\frac{\partial^2}{\partial x^2}  V_d^{\taum(x),\mu}(t,x)\right)+ L(t,x,u)\right] \right\}= 0,
\end{align}
where the interchange of $\frac{\partial}{\partial t}$ and integration with respect to the fixed probability measure $\mu$ (cf.\ the proof of Theorem~\ref{thm:disint_sHJB}) follows by applying the FTC in $s$ to $V_d^{\taum(x),\mu}(s,x)$ and then invoking Fubini's theorem, under the joint integrability of $(s,x)\mapsto \partial_s V_d^{\taum(x),\mu}(s,x)$ on $\mathcal{D}^{\taum}\cap([t,t+h]\times\Ree^n)$ for arbitrarily small $h>0$. 
Furthermore, the boundary term from integration by parts vanished due to the vanishing tail condition (see Remark~\ref{remark_abs}) and/or the absorbing boundary condition (see Assumption~\ref{remark_abs2}).

From \eqref{relation_ll} and the standard normalization 
$\int \frac{\delta \bar V_d^{\taum}}{\delta\mu}(t,\mu)(x)\, \d \mu(x)=0$ (discussed in Section~\ref{notation_sec}), it follows that,
\begin{align} \label{eqn6}
    V^{\taum(x),\mu}_{d}(t,x) &= \left(\frac{\delta \bar V_d^{\taum}}{\delta \mu}(t,\mu)\right)(t,x) + \underbrace{\int_{{\rm S}_{\mu_t}}   V^{\taum(x),\mu}_{d}(t,x) \, \d \mu(x)}_{\text{normalization constant}},
\end{align}
and, \blue{since the partial derivative of the normalization constant (from \eqref{eqn6}) with respect to $x$ is zero, plugging \eqref{eqn6} into \eqref{eqn5} yields}
\begin{align}
        &\frac{\partial}{\partial t}\bar V_d^{\taum}(t,\mu) +\mathbb{E}_{t,x\sim\mu} \left\{\min\limits_{u\in \Ree^m}\left[f(t,x,u)\t\frac{\partial}{\partial x} V_d^{\taum(x)}(t,x)\right. \right. \nonumber \\
        &\hspace*{10mm} + \left. \left. \trace\left(\frac{\sigma_t\sigma_t\t }{2}\frac{\partial^2}{\partial x^2}  V_d^{\taum(x)}(t,x)\right)+ L(t,x,u)\right] \right\}= 0,
\end{align}
which is the same HJB equation as the \textit{distributional} sHJB \eqref{eqn4}. The proof is complete.
\end{proof}

In summary, \highlight{(P0)} is the original stochastic optimal control and stopping problem, while \highlight{(P1)} is its transformed distributional formulation. Theorem \ref{thm_fp_dshjb} shows that \highlight{(P1)} preserves the same optimality structure as \highlight{(P0)} through the same distributional sHJB equation. This allows us to carry out the variational analysis on the transformed problem \highlight{(P1)}. To this end, we next introduce \highlight{(P2)}, a general distributionally-constrained stochastic optimal control and stopping problem that will serve as the problem formulation for the variational analysis. The necessary conditions derived in the sequel are thus established for \highlight{(P2)} and are interpreted in view of the correspondence between \highlight{(P0)} and \highlight{(P1)} established above.
\subsection{Variational Analysis of the Distributionally-Constrained Problem} \label{var_analysis_tr_pr}

\subsubsection{Distributionally-Constrained Problem Definition}

According to the problem formulation of \highlight{(P1)}, we are interested in solving the general stochastic optimal control problem with distributional constraints, denoted by \highlight{(P2)}:
\begin{subequations}
\begin{equation}
\min_{\taum\in\mathcal{T}_{\rm ad},u\in \mathcal{U}}  J =  \mathbb{E}_{0,x \sim \mu_0}\left[
\int_0^{\taum(x_0)} L(s,x_s,u_s(x_s)) \, \d s +\Phi\bigl({\taum(x_0)},x_{\taum(x_0)}\bigr)\right] , \label{eq:st_obj_fnc} \tag{($\text{P2a}$)}
\end{equation}
subject to,
\begin{align}
\quad & \d x_t = \tilde f(t, x_t, u_t(x_t)) \, \d t, \quad x_0\sim\mu_0, \label{eq:dyn_const} \tag{($\text{P2b}$)} \\
& \mathbb{E}_{0,x \sim \mu_0} \Psi({\taum}(x_0), x_{\taum(x_0)}) = 0\tag{($\text{P2c}$)}, \label{eq:dist_const}
\end{align}
\end{subequations}
where $\mu_0\in \mathcal{P}_2(\Ree^n)$ denotes the fixed initial distribution, and $\Psi(\cdot,\cdot)$ is the terminal distribution constraint. As in \highlight{(P1)}, \eqref{eq:dyn_const} is understood in the McKean-Vlasov (nonlocal) sense and is coupled with the induced density evolution \eqref{trans_eq} and admissible stopping rule associated with $\taum \in \mathcal{T}_{\rm ad}$.

\subsubsection{Unconstrained Problem Formulation} \label{opt_unc_form}
In order to derive the optimality conditions with independent decision variables, we transform Problem \highlight{(P2)} into an unconstrained optimization problem. Let $\{\mu_t\}_{t\ge0}\subset\mathcal{P}_2(\mathbb{R}^n)$ be a time-indexed family of absolutely continuous probability measures, and $v:[0,+\infty)\times\mathbb{R}^n\lra \mathbb{R}^n$ be the associated vector field.
We define the measure flow operator, $\Tflow{v}{t}:\mathcal{P}_2(\mathbb{R}^n) \lra \mathcal{P}_2(\mathbb{R}^n)$, that maps a given initial measure $\mu_0$ to its image at time $t$, i.e., $\mu_t=\Tflows{v}{t}\mu_0$. 

\begin{remark} \label{rem_flow_op}
The measure flow operator $\Tflow{v}{t}$ is induced by the characteristic flow of the vector field $v$. In the present setting, this flow field is the closed-loop dynamics of \highlight{(P2)}, where the characteristic solutions satisfy $\dot x(t,x)=v(t,x)$, and $\Tflow{v}{t}$ follows from the Representation Formula for the Continuity Equation \cite[Prop. 8.1.8, Theorem 8.3.1]{book_gr} (alternatively, the Superposition Principle from \cite[Theorem 3.4]{superposition}). 
\end{remark} 

Assuming a sufficiently regular velocity field $v$ along the density path $\{\mu_t\}_{t\ge0}$ with a fixed initial probability measure $\mu_0$ for the problem \highlight{(P2)}, such that, for every compact subset $A \subset \Ree^n$,
\begin{align} \label{mu_nec_cond}
    \int_0^{+\infty} \left( \sup_{x\in A} |v(t,x)|+\text{Lip}(v_t,A) \right) \, \d t < +\infty,
\end{align}
where $v_t(\cdot) \Let v(t,\cdot)$ and $\text{Lip}(v_t,A)$ denotes the Lipschitz constant of $x\mapsto v_t(x)$ in the set $A$. Moreover,
\begin{align} \label{mu_suff_cond}
    \int_0^{+\infty} \int_{{\rm S}_{\mu_t}} |v(t,x)|^p \, \d \mu_t(x) \, \d t < +\infty, \, \text{for some} \, p > 1.
\end{align}
We denote by $\mathcal V^{\mu}_{p}$ the class of velocity fields $v$ satisfying \eqref{mu_nec_cond}--\eqref{mu_suff_cond} with respect to $\{\mu_t\}_{t\ge0}$. Under $v\in\mathcal V_p^{\mu}$, the hypotheses required for both Representation Formula for the Continuity Equation \cite[Prop. 8.1.8, Theorem 8.3.1]{book_gr}, and for the Superposition Principle \cite[Theorem 3.4]{superposition} are satisfied. Then, $\{ \Tflow{v}{t}\}_{t \ge 0}$ defines a continuous flow on $\mathcal{P}_2(\Ree^n)$, the map $t\mapsto\Tflows{v}{t}\mu_0$ is continuous with respect to the $\mathcal{W}_2$ metric; moreover, the induced regular Lagrangian flow is unique $\mu_0$-a.e. \cite{superposition}. Consequently, there is a well-defined inverse (pullback) operator $\Tflows{v}{-t}=\Tflows{-v}{t}$, such that $\mu_0=\Tflows{v}{-t}\mu_t$. 

The indicator function $\indic{t<\taum(x_0)}$ will be employed to denote the alive trajectories associated with a given admissible time assignment function $\taum\in \mathcal{T}_{\rm ad}$ and the initial condition $x_0\in\mathbb{R}^n$. 
We will use the indicator function to impose $\tilde f(\cdot) = 0$ and $L(\cdot)=0$ for the stopped trajectories. Throughout the rest of this paper, we restrict our attention to pairs $(v,\{\mu_t\}_{t\ge0})$ that are consistent in the sense of the Remark \ref{rem_flow_op}, i.e., $\mu_t=\Tflows{v}{t}\mu_0$ and hence $\{\mu_t\}_{t\ge 0}$ solves the continuity equation with velocity field $v\in \mathcal V^{\mu}_{p}$.

\subsubsection{Augmented Lagrangian Formulation} \label{opt_auglag_form}

For Problem \highlight{(P2)}, define the augmented objective functional $\mathcal{J}$ as
\begin{align} \label{orig_mod_obj}
&\mathcal{J}\bigl(\{\lambda_t\}_{t\in[0,+\infty\bigr)},\{u_t\}_{t\in[0,+\infty)},\{\mu_t\}_{t\in[0,+\infty)},\taum,\eta)  \nonumber \\
& =\int_0^{+\infty} \int_{{\rm S}_{\mu_t}} \left[ \onep L(t, x, u_t(x))   \right. \nonumber \\
& \qquad  \left. -\lambda(t,x)\t  (\dot x(t,x) - \onep \tilde f(t,x,u_t(x))) \right] \, \d \mu_t(x) \, \d t \nonumber \\
& \quad+ \mathbb{E}_{0,x \sim \mu_0} [\Phi({\taum(x_0)}, x_{\taum(x_0)}) - \eta \Psi({\taum(x_0)}, x_{\taum(x_0)})],
\end{align}
where $\lambda:[0,+\infty)\times \mathbb{R}^n \to \mathbb{R}^n$, $\lambda_t\Let \lambda(t,\cdot)$ and $\eta$ are Lagrange multipliers, $u\in \mathcal{U}$, $\taum\in \mathcal{T}_{\rm ad}$ and $\mu_t\in \mathcal{P}_2(\Ree^n)$, for all $t\in[0,+\infty)$. Here $\{\mu_t\}_{t\ge 0}$ is always understood as the density path induced by the Lagrangian flow $\Tflows{v}{t}$ generated by $\dot x(t,x) = v(t,x)$ (see Remark~\ref{rem_flow_op}).

Note that, when the optimization problem \eqref{eq:st_obj_fnc}-\eqref{eq:dist_const} attains a feasible solution, vanishing boundary conditions will hold. 
\begin{assumption}[Regularity of the Lagrange Multiplier]\label{assump_reg_lm}
    We assume that the optimal Lagrange multiplier field $\lambda:[0,+\infty)\times\mathbb{R}^n\lra\mathbb{R}^n$ associated with the minimizer of the augmented objective functional \eqref{orig_mod_obj} (and hence with problem \highlight{(P2)}, admits at least weak derivatives of first order in both space and time, with all such derivatives locally integrable on $\Ree_{\ge 0}\times\mathbb{R}^n$. Accordingly, throughout the following variational analysis \blue{below}, 
    $\lambda$ and its derivatives are interpreted in the weak sense.
\end{assumption}

Assumption \ref{assump_reg_lm} is consistent with the regularity required for the viscosity characterization of the value function in Assumption \ref{regularity_val} (see also Definition \ref{def:test} for the class of regular test functionals), in the sense that the Lagrange multiplier field $\lambda$ is formally interpreted as the costate, that is, as the spatial gradient of the value function.

We state our final result for the first-order optimality conditions for the distributionally-constrained stochastic optimal control and stopping problem \highlight{(P2)}.

\begin{theorem}[First-order optimality conditions with free stopping time]\label{thm:common-stopping}
Consider the transformed problem \textbf{(P2)}, i.e., \eqref{eq:st_obj_fnc}-\eqref{eq:dist_const}, with initial probability measure $\mu_0\in\mathcal{P}_2(\Ree^n)$ with density $\rho_0$ and terminal distribution constraint \eqref{eq:dist_const}. Suppose that the regularity assumptions, Assumption \ref{regularity_dyn}, Assumption \ref{regularity_cost}, Assumption \ref{ass_reg_tau}, Assumption \ref{Assumption_5}, and Assumption \ref{assump_reg_lm} hold. 
Let $\taum^*\in \mathcal{T}_{\rm ad}$ be the optimal admissible time-assignment function, $u^* \in \mathcal U$ be the optimal feedback law, and let $x^*(\cdot)$ denote the corresponding optimal state trajectory (a characteristic trajectory associated with $x_0 \in \supp(\mu_0)$). 
Define the induced alive probability measure $\mu^{a,*}_t$ for all $t \ge 0$, with density $\rho^{a,*}(t,\cdot)$, as the solution to the Liouville equation,
\begin{align}\label{eq:liouville_star}
&\frac{\partial}{\partial t}\rho^{a,*}(t,x)=-\nabla_x\!\cdot\!\Big [\tilde f\big(t,x,u^*(t,x)\big)\,\rho^{a,*}(t,x)\Big], \nonumber \\
&\rho^{a,*}(0,x)=\rho_0(x), \quad \text{for all } (t,x) \in \mathcal{D}^{\taum^*},
\end{align}
with Dirichlet boundary condition $\rho^{a,*}(t,x)=0$ for all $(t,x) \in\mathcal{E}^{\taum^*}$, where the transformed drift $\tilde f\in \mathcal{V}_p^{\mu^*}$, i.e., the velocity field, is
\begin{align}\label{eq:tildef_star}
\tilde f(t,x^*,u^*(t,x^*)) \Let f(t,x^*,u^*(t,x^*))-\frac{\sigma_t\sigma_t\t}{2}\nabla_x\log\rho^{a,*}(t,x^*).
\end{align}
For each initial condition $x_0\in {\rm S}(\mu_0)$, let $x^*(\cdot;x_0)$ denote the associated characteristic solving the ODE
\begin{align}\label{eq:char_star}
\dot x^*(t;x_0)&=\tilde f\big(t,x^*(t;x_0),u^*(t,x^*(t;x_0))\big), \nonumber  \\
x^*(0;x_0)&=x_0, \quad \text{for all } (x_0,t) \in {\rm S}(\mu_0) \times \lcrc{0}{\taum^*(x_0)}.
\end{align}
Then, there exists a scalar multiplier $\eta\in\Ree$, and a co-state field $\lambda:\Ree_{\ge0}\times\Ree^n\to\Ree^n$, such that, for $\mu_0$-a.e. $x_0$ and for a.e. $t\in[0,\taum^*(x_0)]$, the following first-order optimality conditions hold along the characteristic, i.e., along $t \mapsto x^*(t;x_0)$, with the shorthand notation, $x^*(t)\Let x^*(t;x_0), u^*(t) \Let u^*(t,x^*(t;x_0)), \lambda(t) \Let\lambda(t,x^*(t;x_0))$, $\taum^* \Let \taum^*(x_0)$:
\begin{enumerate}[label={\textup{(\eqref{thm:common-stopping}-\alph*)}}, leftmargin=*, widest=b, align=left]
\item \highlight{Optimal control:} 
\begin{align}
\frac{\partial}{\partial u} L(t,x^*(t),u^*(t))
+ \lambda(t)\t \frac{\partial}{\partial u} f(t,x^*(t),u^*(t)) = 0.
\label{OC7} \tag{(OS1)}
\end{align}

\item \highlight{Co-state dynamics:} Characteristics to the co-state field, i.e., $t \mapsto \lambda(t)$, (with $\dot \lambda(t)\Let \frac{\d}{\d t}\lambda(t,x^*(t;x_0))$ being the material derivative along the ODE \eqref{eq:char_star}), \blue{i.e., $\frac{\d}{\d t}\lambda(t,x^*(t;x_0))
=\frac{\partial}{\partial t}  \lambda(t,x^*(t;x_0))
+\frac{\partial}{\partial x} \lambda(t,x^*(t;x_0)) \t \tilde f \big(t,x,u^*(t,x)\big)$,} satisfies
\begin{align}
\hspace{-1 em}\dot \lambda(t) &=
- \frac{\partial}{\partial x} L(t,x^*(t),u^*(t))
- \lambda(t)\t \frac{\partial}{\partial x} f(t,x^*(t),u^*(t))-\frac{\partial}{\partial x}\trace\left(\frac{\partial }{\partial x}\lambda(t) \frac{\sigma_t\sigma_t\t }{2}\right) \nonumber \\
&\quad
- \frac{\partial}{\partial x}\lambda(t)\t
\frac{\sigma_t\sigma_t\t}{2}\nabla_x \log \rho^{a,*}\big(t,x^*(t)\big), \, \text{a.e.,}
\label{OC6} \tag{(OS2)}
\end{align}
where $\frac{\partial}{\partial x}\lambda(t)$ denotes the Jacobian of the field $x\mapsto\lambda(t,x)$ with $\lambda$ being subject to the following transversality condition.

\item \highlight{Terminal transversality with distribution constraint:} At the optimal stopping time
\begin{align}
&\lambda(\taum\as,x^*(\taum\as))\t
=\frac{\partial}{\partial x}\Phi(\taum\as,x^*(\taum\as))
-\eta\,\frac{\partial}{\partial x}\Psi(\taum\as,x^*(\taum\as)), \label{OC5} \tag{(OS3)} \\
&\mathbb{E}_{0,x\sim \mu_{0}}\!\left[\Psi(\taum\as(x_0),x\as_{\taum\as(x_0)}) \right]=0. \label{OC9} \tag{(OS4)}
\end{align}

\item \highlight{Free-terminal-time (stopping) condition:} The transversality condition reads
\begin{align}
&\tilde H(\taum\as,x^*(\taum\as),u^*(\taum\as),\lambda(\taum\as,x^*(\taum\as)))
+ \frac{\partial}{\partial t}\Phi(\taum\as,x^*(\taum\as)) 
-\eta\,\frac{\partial}{\partial t}\Psi(\taum\as,x^*(\taum\as))
=0.
\label{OC8} \tag{(OS5)}
\end{align}
\end{enumerate}

The modified Hamiltonian is defined for $t\in[0,\taum^*]$ by
\begin{align}\label{mod-ham}
(t,x\as,u\as,\lambda) \mapsto \tilde H(t,x^*,u^*,\lambda)
\Let \lambda\t &f(t,x^*,u^*) + L(t,x^*,u^*)
+ \trace \left(\frac{\partial}{\partial x}\lambda\,\frac{\sigma_t\sigma_t\t}{2}\right).
\end{align}
Moreover, if we restrict $\mathcal{T}_{\rm ad}$ for \emph{common} stopping, (i.e., $\taum\in \mathcal{T}_{\rm ad}$ does not depend on the initial condition $x_0\in {\rm S}(\mu_0)$), then \eqref{OC8} is replaced with the following expectation form
\begin{align}
&\mathbb{E}_{0,x\sim \mu_{0}}\!\left[
\tilde H(\taum\as,x\as_{\taum\as},u\as_{\taum\as},\lambda(\taum\as,x\as_{\taum\as}))
+ \frac{\partial}{\partial t}\Phi(\taum\as,x\as_{\taum\as})
-\eta\,\frac{\partial}{\partial t}\Psi(\taum\as,x\as_{\taum\as})
\right] = 0.
\tag{(CS)}\label{OC4}
\end{align}
\end{theorem}

\begin{remark}
It can be readily verified that the optimality conditions in Section~\ref{var_analysis_tr_pr} hold against the analytically derived optimal control policies from financial engineering literature, including the American put option with infinite horizon \cite[Chapter VII]{book_stop} and the optimal stopping of a Brownian bridge \cite{opt_stop_ex1}.
\end{remark}

\section{Conclusion}

\blue{We have developed a density-based framework for nonlinear stochastic optimal control problems with optimal stopping and
terminal distributional constraints.
The central step rewrites the controlled Fokker–Planck
equation as a continuity equation whose velocity is the original drift minus one-half of the diffusion covariance multiplied by the score of the alive density. 
Integration by parts then converts
the score-dependent contributions into diffusion terms in the distributional equations, while the absorbing-boundary flux records when and where stopping occurs. 
This representation
provides the basis for the variational analysis of control and stopping under the stated regularity and characteristic-representation assumptions. Future work will focus on developing scalable,
sampling-based numerical methods for high-dimensional problems, including indirect physics-informed approaches that combine the optimality conditions with score-matching algorithms.}

\newpage
\appendix
\section{List of Notations}\label{appen:notation}
For ease of reference, Table \ref{table_notation} summarizes the main notation employed throughout the article.


\begin{longtable}{>{\raggedright\arraybackslash}p{0.24\textwidth} >{\raggedright\arraybackslash}p{0.70\textwidth}}
\caption{Notations}
\label{table_notation}\\
\toprule
Symbol & Meaning \\
\midrule
\endfirsthead

\toprule
Symbol & Meaning \\
\midrule
\endhead

\bottomrule
\endfoot

\textbf{Basic spaces and measures} & \\[0.2em]
$\mathbb{N}^*$, $\mathbb{Z}$ & Positive integers and integers, respectively. \\
$\Ree^n$ & $n$-dimensional Euclidean state space. \\
$\|\cdot\|$ & Euclidean norm. \\
$\mathcal{B}(\Ree^n)$ & Borel $\sigma$-algebra on $\Ree^n$. \\
$\mathcal{P}(\Ree^n)$ & Set of probability measures on $\Ree^n$. \\
$\mathcal{P}_p(\Ree^n)$ & Set of probability measures on $\Ree^n$ with finite $p$-th moment. \\
$\mathcal{P}_{\rm sub}(\Ree^n)$ & Set of sub-probability measures on $\Ree^n$. \\
$\mathcal{P}_{p,\rm sub}(\Ree^n)$ & Set of sub-probability measures on $\Ree^n$ with finite $p$-th moment. \\
$\mu \ll \nu$ & Absolute continuity of $\mu$ with respect to $\nu$. \\
$\mathrm{Leb}^n$ & $n$-dimensional Lebesgue measure. \\
$\mathrm{Leb}^n\!\restriction_D$ & Restriction of Lebesgue measure to a measurable set $D \subseteq \Ree^n$. \\
$\mathcal{W}_2$ & $2$-Wasserstein distance on $\mathcal{P}_2(\Ree^n)$. \\
${\rm S}_\mu$ & Support of the probability measure $\mu$. \\[0.4em]

\textbf{Probability and stochastic processes} & \\[0.2em]
$(\Omega,\mathcal{F},\mathbb{P})$ & Underlying complete probability space. \\
$\omega \in \Omega$ & Sample point. \\
$\{\mathcal{F}_t\}_{t\ge 0}$ & Filtration generated by the driving Brownian motion. \\
$x_t \in \Ree^n$ & State process at time $t\ge 0$. \\
$u_t \in \Ree^m$ & Control process at time $t\ge 0$. \\
$w_t$ & $n$-dimensional Brownian motion at time $t\ge 0$. \\
$\sigma_t \in \Ree^{n\times n}$ & Time-varying diffusion matrix at time $t\ge 0$. \\
$f(\cdot)$ & Drift vector field in the stochastic dynamics. \\
$\mu_t$ & Probability law of $x_t$ at time $t\ge 0$. \\
$\mathbb{P}^{t,x}$ & Regular conditional probability given $x_t = x$. \\
$\mathbb{E}^{\mathbb{P}}[\cdot]$ & Expectation with respect to $\mathbb{P}$. \\
$\mathbb{E}^{\mathcal{F}_t}_{x}[\cdot]$ & Conditional expectation given deterministic initial condition $x_t=x$. \\
$\mathbb{E}^{\mathcal{F}_t}_{x\sim \mu}[\cdot]$ & Conditional expectation when $x_t \sim \mu$ (see \eqref{backgrd:pw:expect}). \\
$\mathbb{E}_{t,x\sim\mu}[\cdot]$ & Integration with respect to a given measure $\mu$ at time $t$ in the deterministic/PDE formulation (see Remark \ref{backgrd:det_PDE:expect}). \\[0.4em]

\textbf{Function spaces and derivatives} & \\[0.2em]
$C^k(\Ree^n)$ & Space of $k$ times continuously differentiable functions on $\Ree^n$. \\
$C^k_{\mathrm{loc}}(\Ree^n)$ & Space of functions whose restriction to each compact subset is $k$ times continuously differentiable. \\
$\mathcal{L}^1(\Omega,\filtration_t,\mathbb P;\mathbb R^k)$ & Space of absolutely integrable and $\filtration$-measurable random vectors $\mathcal{L}^1(\Omega,\filtration_t,\mathbb P;\mathbb R^k) \ni Z:\Omega\lra\mathbb R^k$. \\
$\mathcal{L}^\infty(\mathcal{D};\mathcal{R})$ & Space of measurable essentially bounded maps from $\mathcal{D}$ to $\mathcal{R}$. \\
$\dfrac{\delta \mathcal{J}}{\delta \mu}(t,\mu)(x)$ & Linear functional derivative of a functional $\mathcal{J}$ with respect to $\mu$. \\
$\partial_\mu \mathcal{J}(t,\mu)(x)$ & Lions derivative, formally $\partial_\mu \mathcal{J}(t,\mu)(x)=\nabla_x \dfrac{\delta \mathcal{J}}{\delta \mu}(t,\mu)(x)$. \\[0.4em]

\textbf{Control and stopping sets} & \\[0.2em]
$\Pi$ & Class of measurable feedback policies $\pi$ (see \eqref{backgrd:mapping:sets}). \\
$\mathcal{U}$ & Admissible time-varying control set induced by feedback laws in $\Pi$ (see \eqref{backgrd:mapping:sets}). \\
$\mathcal{G}$ & Class of admissible stopping boundary functions $\mathcal{G} \ni g:\Ree_{\ge 0}\times\Ree^n \lra \Ree$ (see \eqref{backgrd:mapping:sets}). \\
$\tau_g$ & Stopping time induced by $g\in \mathcal{G}$, i.e. the first hitting time of the set $\{g\ge 0\}$. \\
$\mathcal{D}^g$ & Space-time continuation (alive) region associated with $g\in \mathcal{G}$ (see \eqref{backgrd:alive:region}). \\
$\mathcal{D}_t^g$ & Spatial continuation set at time $t$ (see \eqref{backgrd:boundary:sets}). \\
$\partial \mathcal{D}_t^g$ & Time-slice boundary of $\mathcal{D}^g$ at time $t$ (see \eqref{backgrd:boundary:sets}). \\
$\partial \mathcal{D}^g$ & Space-time stopping boundary associated with $g\in \mathcal{G}$ (see \eqref{backgrd:boundary:sets}). \\
$\bar{\mathcal{D}}^g$ & Distributional continuation region (see \eqref{cntr:dist_alive:region}).  \\[0.4em]
\textbf{Costs and value functions} & \\[0.2em]
$L(t,x,u)$ & Running (Lagrangian) cost. \\
$\Phi(t,x)$ & Terminal (Mayer) cost. \\
$\Psi(t,x)$ & Terminal distributional constraint function in the constrained formulation (see \eqref{eq:dist_const}). \\
$V(t,x)$ & Pointwise value function for the original stochastic optimal stopping problem (see \eqref{single_V}). \\
$V^g(t,x)$ & Pointwise value function for fixed stopping boundary $g$ (see \eqref{single_Vg}). \\
$\bar V^g(t,\mu)$ & Distributional value function for fixed stopping boundary $g$ (see \eqref{eqn3}). \\
$\bar V(t,\mu)$ & Distributional value function for the original problem, optimized over $g\in\mathcal{G}$ (see \eqref{eqn5}). \\[0.4em]

\textbf{Generators and test functionals} & \\[0.2em]
$C^{1,2}_2$ & Test-functional class used for the distributional sHJB equation (see Definition \ref{def:test}). \\
$\Xi$ & Generic test functional with $C^{1,2}_2 \ni \Xi: \Ree_{\ge 0}\times \mathcal{P}_2(\Ree^n) \lra \Ree$. \\
$\mathcal{L}_u$ & Classical second-order generator associated with the controlled SDE where $u\in\Ree^m$ (see \eqref{sHJB-generator}). \\
$\bar{\mathcal{L}}_u$ & Distributional second-order generator acting on functionals of measures with measurable feedback policies $u\in \Pi$ (see \eqref{dist_sHJB-generator}). \\[0.4em]

\textbf{Alive / absorbed measure notation (see \S\ref{subs:subprob})} & \\[0.2em]
$\dagger$ & Cemetery (absorbing) state added to the state space. \\
$\tilde x_t$ & Absorbed state process, equal to $x_t$ while alive and $\dagger$ after stopping. \\
$\mu_t^a$ & Alive component of the law of $x_t$, a sub-probability measure on the continuation region. \\
$\rho^a(t,x)$ & Alive probability density satisfying $\d \mu_t^a = \rho^a(t,x)\, \d x$. \\
$\tilde \rho(t,\rm E)$ & Cumulated total exit probability mass through some Borel set $\rm E \subseteq \partial \mathcal{D}^g$ up until the time $t\ge 0$. \\
$J(t,x)$ & Probability flux associated with the alive density. \\
$\alpha(t,x)$ & Outward boundary flux density on the stopping boundary. \\
$J^{\mathrm{out}}(t,x,u)$ & Outgoing flux used in the application of the divergence theorem (see \eqref{J_out}). \\[0.4em]
$\tilde f(t,x,u)$ & Transformed nonlocal drift / velocity field (see \eqref{vel_field}). \\[0.4em]

\textbf{Final-time assignment formulation} & \\[0.2em]
$\mathcal{T}$ & Set of continuously differentiable final-time assignment functions $\mathcal{T}\ni \taum:\Ree^n\lra\Ree_{\ge 0}$ (see \eqref{fpt_set_tau}). \\
$\mathcal{T}_{\mathrm{ad}}$ & Admissible subclass of $\mathcal{T}$ for which the alive domain and Dirichlet problem are well posed (see Definition \ref{fpt_admis_tau}). \\
$\mathcal{D}^{\taum}$ & Alive space-time domain induced by $\taum \in \mathcal{T}_{\mathrm{ad}}$ (see Definition \ref{fpt_admis_tau}). \\
$\mathcal{E}^{\taum}$ & Stopping surface induced by $\taum \in \mathcal{T}_{\mathrm{ad}}$ (see Definition \ref{fpt_admis_tau}). \\
$\partial \mathcal{D}_t^{\taum}$ & Time-sliced boundary of the alive set at time $t$ (see Definition \ref{fpt_admis_tau}). \\
$\one_{t<\taum(x_0)}$ & Indicator of whether the trajectory initialized at $x_0$ is still alive at a given time $t$ for $\taum \in \mathcal{T}_{\mathrm{ad}}$. \\[0.4em]

\textbf{Deterministic transformed value functions} & \\[0.2em]
$V_d^{t_f,\mu}(t,x)$ & Pointwise value function of the deterministic transformed problem with fixed final time $\tf$ and prescribed density path $\{\mu_s\}_{s\ge t}$ (see \eqref{relation_d_ll}). \\
$\bar V_d^{\taum}(t,\mu)$ & Distributional value function of the deterministic transformed problem for a fixed time assignment function $\taum \in \mathcal{T}_{\mathrm{ad}}$ (see \eqref{relation_ll}). \\
$\bar V_d(t,\mu)$ & Distributional value function of the deterministic transformed problem optimized over $\taum\in\mathcal{T}_{\mathrm{ad}}$ (see \eqref{relation_lll}). \\[0.4em]

\textbf{Measure flows (see \S\ref{opt_unc_form})} & \\[0.2em]
$v(t,x)$ & Generic velocity field in the continuity-equation formulation. \\
$v_t(\cdot)$ & Spatial map $x \mapsto v(t,x)$ at fixed time $t\ge 0$. \\
$\rm{Lip}(v_t,A)$ & Lipschitz constant of $v_t$ on the set $A$. \\
$\mathcal{V}_p^\mu$ & Class of admissible velocity fields $v\in \mathcal{V}_p^\mu$ with some $p>1$ satisfying the required integrability and local Lipschitz conditions along $\{\mu_t\}_{t\ge 0}$. \\
$\Tflow{u,g}{h}$ &Flow map generated by the feedback control $u\in \mathcal{U}$ and stopping rule $\tau_g$ for some $g\in \mathcal{G}$ over a time increment $h$. \\
$\Tflow{v}{h}$ & Flow map generated by the velocity field $v$ over a time increment $h$ (see Definition \eqref{rem_flow_op}). \\
$\Tflows{v}{h} \mu_0$ & Push-forward of the initial measure under the flow $T_t^v$. \\
$\Tflows{-v}{h}$ & Inverse pullback operator associated with the flow $\Tflow{-v}{h}$. \\[0.4 em]

\textbf{Variational formulation (see \S\ref{opt_auglag_form})} & \\[0.2em]
$\lambda(t,x)$ & Lagrange multiplier / costate field in the augmented Lagrangian formulation. \\
$\eta$ & Scalar Lagrange multiplier associated with the terminal distributional constraint. \\
$\mathcal{J}$ & Augmented objective functional in the unconstrained variational reformulation (see \eqref{orig_mod_obj}). \\
$\tilde{H}(t,x,u,\lambda)$ & Modified Hamiltonian associated with the transformed dynamics (see \eqref{mod-ham}).
\end{longtable}

\section{Proof of Theorem \ref{thm:common-stopping}} \label{Appendix:A}
\setcounter{equation}{0}
\renewcommand{\theequation}{B.\arabic{equation}}
\makeatletter
\renewcommand{\tagform@}[1]{(#1)}
\makeatother
We organize the proof of Theorem~\ref{thm:common-stopping} into the following steps.

\medskip
\noindent\textbf{Step 1: Euler and Lagrangian Formulations.}
Using the superposition principle \cite[Theorem 3.4]{superposition} (see \S\ref{opt_unc_form}), we pass from the Eulerian formulation of the augmented objective functional \ref{orig_mod_obj} to an equivalent Lagrangian representation along the characteristic trajectories, written as an expectation with respect to the initial probability measure $\mu_0$.

\medskip
\noindent\textbf{Step 2: Variational perturbations.}
We introduce admissible first-order perturbations of all the arguments of the augmented functional \ref{orig_mod_obj}, namely the control, Lagrange multiplier for dynamics and terminal distributional constraint, time-index family of probability measures and the final-time assignment function. In particular, the probability measure perturbation is described through Wasserstein-type tangent variations, while the stopping-time perturbation is handled through the Heaviside representation of stopped integrals.

\medskip
\noindent\textbf{Step 3: G\^ateaux Derivative of $\mathcal{J}$ and First-variation Decomposition.}
We define the perturbed functional and decompose its increment into three contributions: the variation with fixed probability measure, the variation with respect to the probability measure flow, and the endpoint variation.

\medskip
\noindent\textbf{Step 4: Evaluation of the probability measure-flow variation.}
Using the relation between density perturbations and spatial variations, together with the divergence theorem and the chain rule, we compute the contribution of the measure variation to the first variation of the functional. The perturbation of the final-time assignment function contributes a boundary term at the final-time.

\medskip
\noindent\textbf{Step 5: Evaluation of the fixed-measure and endpoint variations.}
Next, we compute the fixed-measure contribution by first-order Taylor expansion and evaluate the endpoint contribution through the total variation of the terminal map. This yields the terms associated with control, Lagrange multiplier for both dynamics and terminal distributional constraint and final-time assignment functions.

\medskip
\noindent\textbf{Step 6: Collection and simplification of first-order terms.}
We collect all first-order contributions into a single expression for the G\^ateaux derivative for both Eulerian formulation \eqref{zero_var_rhs} and the Lagrangian formulation \eqref{first_var_rhs}. Then, using the state equation \eqref{J2}, the identity obtained from the diffusion term \eqref{mod_obj_L_trick}, and an integration-by-parts argument along characteristics \eqref{appdx_light2}, we rewrite the first variation in the compact form \eqref{first_var_rhs_2}.

\medskip
\noindent\textbf{Step 7: Derivation of the necessary conditions.}
Since the admissible perturbations are independent, the vanishing of the first variation implies that each coefficient in \eqref{first_var_rhs_2} must vanish separately. This yields the state dynamics, the optimality condition with respect to the control, the adjoint equation along the characteristics, optimal final-time (stopping) condition, the terminal transversality condition, and the distribution constraint.

\medskip
\noindent\textbf{Step 8: Common stopping as a special case.}
Finally, when the stopping rule is common to all realizations, so that $\taum$ is scalar-valued, the stopping variation is no longer pointwise in the initial condition. In that case, the corresponding optimality condition is obtained by integrating the coefficient of $\delta \taum$ in \eqref{first_var_rhs_2} with respect to $\mu_0$, which gives \ref{OC4}.

\medskip
The remainder of this appendix follows the above roadmap. We first establish the first-order necessary conditions for a initial-condition-dependent final-time assignment function $\taum \in \mathcal{T}_{\rm ad}$, and then specialize the argument to the common-stopping case by restricting $\taum$ to be scalar.

\subsection{Euler and Lagrangian Formulations}

From the superposition principle \cite[Theorem 3.4]{superposition} (see \S\ref{opt_unc_form}) the family $\{\mu_t\}_{t \ge 0}$ and velocity field $v \in \mathcal V^{\mu}_{p}$ induce a probability measure on the path space concentrated on a.e. ODE solutions.  Based on this fact, we can rewrite our objective functional, i.e. \ref{orig_mod_obj}, from an alternative point of view. To do this, we exploit Fubini's theorem to rewrite the integrals over $(t,x)$ against $\mu_t$ as expectations over initial conditions $x_0\sim\mu_0$ of line integrals along trajectories, yielding the following equivalent objective functional
\begin{align} \label{appdx_equiv}
&\mathcal{J}(\cdot)  = \mathbb{E}_{0,x \sim \mu_0} \int_0^{\taum(x_0)} \left[L(t, x, u_t) - \lambda(t,x)\t  \times \, (\dot x(t,x) - \tilde f(t,x,u_t(x))) \right] \, \d t \nonumber \\
& \quad+ \mathbb{E}_{0,x \sim \mu_0} [\Phi({\taum(x_0)}, x_{\taum(x_0)}) - \eta \Psi({\taum(x_0)}, x_{\taum(x_0)})].
\end{align}
The expression \eqref{appdx_equiv} should be understood as a Lagrangian representation of the Eulerian cost \ref{orig_mod_obj}.

We will employ the notations established in \S\ref{var_analysis_tr_pr}.
Let $\d \mu^{\taum}_t(x)\Let \onep \d \mu_t(x)$, which will be represented with the probability density function $\rho^{\taum}(\cdot)$, i.e., $\d \mu^{\taum}_t(x)\Let\rho^{\taum}(t,x) \, \d x$ and $\d \mu^{\taum}_t(x)=\rho^{\taum}(t,x) \,\d x$ and $\rho^{\taum}(t,x)=\rho(t,x), \forall x\in {\rm S}_{\mu_t}^{\taum}$, zero otherwise. We denote the support of the alive probability measure $\mu^{\taum}_t(x)$ at space-time $(t,x)$ as ${\rm S}^{\taum}_{\mu_t}$. 

Observe that, under the Assumption \ref{Assumption_5},
\begin{align} \label{mod_obj_L_trick}
&\int_{{\rm S}_{\mu_t}} \lambda(t,x)\t  \tilde f(t,x,u_t(x)) \, \d \mu^{\taum}_t(x) \nonumber
 \\
&=\int_{{\rm S}_{\mu_t}} \lambda(t,x)\t  \left[f(t,x,u_t(x))-\frac{\sigma_t\sigma_t\t }{2}\nabla_x \log{\rho(t,x)}\right] \rho^{\taum}(t,x) \, \d x, \nonumber \\
&\overset{\text{div. thm.}}{=} \int_{{\rm S}^{\taum}_{\mu_t}} \left[\lambda(t,x)\t  f(t,x,u_t(x)) + \trace\left(\frac{\partial }{\partial x}\lambda(t,x)\frac{\sigma_t\sigma_t\t }{2}\right) \right] \, \d \mu^{\taum}_t(x),
\end{align}
where we used divergence theorem (assuming the regularity for the multiplier field $\lambda$ based on Assumption \ref{assump_reg_lm}) together with vanishing boundary conditions. 

The representation \eqref{appdx_equiv} together with the Eulerian objective functional \ref{orig_mod_obj} and the identity \eqref{mod_obj_L_trick} form the basis on which we conduct the variational analysis. To this end, we next introduce admissible perturbations of the arguments of the objective functionals \ref{orig_mod_obj} and \eqref{appdx_equiv}.

\subsection{Variational Perturbations}

In order to define the G\^ateaux derivative of $\mathcal{J}$, we take sufficiently regular variations for all $t \in[0,+\infty)$, $\delta u_t$, $\delta \lambda_t$, $\delta \taum$, $\delta \eta$, such that, for all sufficiently small $\eps>0$, the perturbed control satisfies $\{u_t+\eps \delta u_t \}_{t\ge 0}\in\mathcal{U}$. For variations, define $\lambda^\eps_t=\lambda_t+\eps \delta \lambda_t$, $u^\eps_t=u_t+\eps \delta u_t$, $\taum^\eps=\taum+\eps \delta \taum$, and $\eta^\eps=\eta + \eps \delta \eta$.

\subsubsection{Measure-Flow Variations} \label{sss_mfv} Based on Otto’s formal Riemannian calculus on Wasserstein space (see \cite[Sec. 1.3, and Sec. 4.4]{otto} and \cite[Ch. 8]{book_gr}), we take mass-preserving variations, or in other words, \emph{interior Otto perturbations} where $\mu_t$ is defined. Specifically, for each fixed $t\ge 0$, the first-order variation $\delta\mu_t$ is represented by a smooth variation $x\mapsto \delta x(t,x)$ on $\Ree^n$ and is understood as an interior Wasserstein-type perturbation, while the normal motion of the stopping surface is treated separately through an admissible variation of the final-time assignment function $\delta\tau_m$ in \S \ref{sss_ftv}.

We define the \emph{interior Otto perturbation} to be the first-order variation generated by a smooth transport field $x\mapsto \delta x(t,x)$
such that
\begin{align*}
    \d(\delta\mu_t)(x)=\delta\rho(t,x)\, \d x,
\quad
\delta\rho(t,x)=-\nabla_x\cdot\big(\rho(t,x)\delta x(t,x)\big)
\quad
x\in {\rm S}_{\mu_t}^{\tau_m},
\end{align*}
in the sense of distributions. In addition, for $\delta \mu_t$ to be a mass-preserving perturbation, we require that the transport
field has vanishing normal trace on the stopping boundary induced by $\taum$, i.e.,
\begin{align*}
    \delta x(t,x)\cdot n(t,x)=0 \quad\text{for  all} \ x\in \partial {\rm S}_{\mu_t}^{\tau_m},
\end{align*}
in the trace sense. Since $\rho^{\taum}(t,x)=0$ on $\partial {\rm S}_{\mu_t}^{\tau_m}$, it follows that
\begin{align*}
\delta\rho^{\taum}(t,x)\big|_{\partial {\rm S}_{\mu_t}^{\tau_m}}&=
-\nabla_x\rho^{\taum}(t,x)\cdot \delta x(t,x) \\
&=-\partial_n\rho^{\taum}(t,x)\,(\delta x(t,x)\cdot n(t,x))
=0 \quad\text{for  all} \ x\in \partial {\rm S}_{\mu_t}^{\tau_m},
\end{align*}
again in the trace sense with $\partial_n \rho^{\taum}(t,x) \Let \nabla_x \rho^{\taum}(t,x) \t n(t,x)$. Therefore, the interior Otto perturbations do not move the stopping surface to first order, and accordingly, preserve the stopping surface and the total mass to first order. Also note that $\delta \mu_0=0$, and the associated velocity field $v(t,x)\Let\dot x(t,x)$ satisfies the regularity condition \ref{mu_suff_cond}.

In summary, we define the class of \emph{admissible interior tangent perturbations} at $\mu_t$ by
\begin{align} \label{mu_var_rel}
\mathcal{T}_{\mu_t}^{\tau_m,\mathrm{int}}
=
\left\{
\delta\mu_t \;\middle|\;
\begin{array}{ll}
\d(\delta\mu_t)(x)=\delta\rho(t,x)\,\d x, & \\[0.3em]
\delta\rho(t,x)=-\nabla_x\cdot\big(\rho(t,x)\delta x(t,x)\big),
& \text{in } {\rm S}_{\mu_t}^{\tau_m},\\[0.3em]
\delta x(t,x)\cdot n(t,x)=0,
& \text{on } \partial {\rm S}_{\mu_t}^{\tau_m}
\end{array}
\right\}.
\end{align}

%
%

Let the varied probability measure be denoted as $\mu_t^\eps=\mu_t+\eps \delta \mu_t$, with $\delta \mu_t\in \mathcal{T}_{\mu_t}^{\tau_m,\mathrm{int}}$. 
Based on the variation over the flow field $\mu_t$, the trajectories are perturbed as $x_t \lra x_t+\eps \delta x(t,x_t)$ and, since $\frac{\d}{\d t}$ is a linear operator, $\dot x_t \ra \dot x_t+\eps \delta \dot  x(t,x_t)$, where, using chain rule, 
\begin{align*}
    \delta \dot x(t,x_t)=\frac{\partial}{\partial x}\dot x(t,x_t) \delta x(t,x_t) \quad \text{for a fixed }t.
\end{align*}
Consequently, for each fixed $t\in[0,+\infty)$, the integration over $\delta \mu_t$ with any sufficiently regular function $\xi:\mathbb{R}^n\to \mathbb{R}$, has the following relation
\begin{align} \label{appdx_idnt}
    \int_{{\rm S}_{\mu_t}} \xi(x) \,  \d (\delta\mu_t)(x)&=\int_{{\rm S}_{\mu_t}} \xi(x) \delta \rho(t,x) \,  \d x 
    =-\int_{{\rm S}_{\mu_t}} \xi(x) \nabla_x \cdot(\rho(t,x)\delta x(t,x)) \,  \d x \nonumber \\
    &\overset{\text{div. thm.}}{=}\int_{{\rm S}_{\mu_t}} \frac{\partial}{\partial x} \xi(x)\t \delta x(t,x) \, \d \mu_t(x),
\end{align}
which follows from a simple application of divergence theorem and vanishing and stopping boundary conditions. When $\xi$ is a function of $\dot x(t,x)$ instead, we simply have, for fixed $t$, from the chain rule
\begin{align} \label{app_chain_dot_x}
\frac{\partial}{\partial x} \xi(\dot x(t,x))\t \delta x(t,x)&=\frac{\partial}{\partial \dot x} \xi(\dot x(t,x))\t \frac{\partial}{\partial x} \dot x(t,x) \delta x(t,x)=\frac{\partial}{\partial \dot x} \xi(\dot x(t,x))\t \delta \dot x(t,x).
\end{align}
%

All first-order normal
motion of the stopping surface is represented separately by the variation $\delta\tau_m$
through the Heaviside formulation in \S \ref{sss_ftv}.

\subsubsection{Final-Time Variations} \label{sss_ftv} In view of \S \ref{sss_mfv}, the variation $\delta\tau_m$ is the only source of first-order normal displacement of the stopping surface. To take the variation against $\taum \in \mathcal{T}_{\rm ad}$, we rewrite the stopped integrals using the Heaviside indicator, $\one_{t<\taum(x_0)}$, for any $x_0\in{\rm S}(\mu_0)$. We interpret the derivative with respect to $\taum$ in the sense of distributions. Then, in the distributional sense, the derivative of a Heaviside step function is a Dirac delta, i.e.,
\begin{align} \label{heav_der}
   & \frac{\one_{t<\taum(\Tflow{v}{-t}(x))+\eps \delta \taum(\Tflow{v}{-t}(x))}-\one_{t<\taum(\Tflow{v}{-t}(x))}}{\eps} 
    \xrightarrow[\eps \to 0]{{\rm dist.}} \delta\bigl(t=\taum(\Tflow{v}{-t}(x)))\delta \taum(\Tflow{v}{-t}(x)\bigr),
\end{align}
where $\xrightarrow{{\rm dist.}}$ denotes convergence in distribution. A standard justification is obtained by replacing the Heaviside indicator with a smooth mollified approximation (e.g., see Appendix~C of \cite{fleming_soner}) and then passing to the limit in the sense of distributions.

\subsubsection{Endpoint Variations} Let the perturbed state trajectory be $x^\eps$ for the perturbed final time $\taum^\eps$ and perturbed probability measure $\mu^\eps$ with the corresponding velocity field $v^\eps\in \mathcal V^{\mu^\eps}_{p}$ starting from the same initial condition $x_0\in{\rm S}({\mu_0})$. Moreover, due to the perturbation of the terminal states, we denote the differential of the terminal map $\delta \Tflow{v}{\taum(x_0)}(x_0)$ with a given velocity field $v\in\mathcal V^{\mu}_{p}$ associated with $\{\mu_t \}_{t \ge 0}$, admissible time assignment function $\taum\in \mathcal{T}_{\rm ad}$ and initial condition $x_0\in {\rm S}(\mu_0)$. Here, $\delta\Tflow{v}{\taum(x_0)}(x_0)$ is used to denote the total endpoint variation, i.e.,
\begin{align}\label{term_map_perturb}
    \delta\Tflow{v}{\taum(x_0)} &\Let \lim_{\eps\lra 0} \frac{1}{\eps} \Bigl(\Tflow{v^\eps}{\taum^\eps(x_0)}-\Tflow{v}{\taum(x_0)}\Bigr) \nonumber \\
    &=\delta x(\taum(x_0),x_{\taum(x_0)}) + \dot x(\taum(x_0),x_{\taum(x_0)}) \delta \taum(x_0).
\end{align}

Having introduced the admissible variations of the arguments of the objective functional and the associated endpoint variations, we now proceed to compute the G\^ateaux derivative of our objective functional. For this purpose, we first start by defining the perturbed functional and the associated increment.

\subsection{G\^ateaux Derivative of $\mathcal{J}$ and First-variation Decomposition}

For notational convenience we define the tuple \(\Upsilon \Let \Bigl(\{\lambda_t\}_{t\in[0,+\infty)},\{u_t\}_{t\in[0,+\infty)},\{\mu_t\}_{t\in[0,+\infty)},\taum,\eta\Bigr)\) and recall the definition of the objective function \(\mathcal{J}(\cdot)\) from \ref{orig_mod_obj}. 

We define the objective functions
\begin{align}
&\mathcal{J}^\eps (\Upsilon)\Let \mathcal{J}\Bigl(\{\lambda_t^\eps\}_{t\in[0,+\infty)},\{u_t^\eps\}_{t\in[0,+\infty)},\{\mu_t^\eps\}_{t\in[0,+\infty)},\taum^\eps,\eta^\eps\Bigr), \\
    &\Delta^\eps \mathcal{J} (\Upsilon) \Let \mathcal{J}^\eps (\Upsilon) -\mathcal{J} (\Upsilon)\label{d_eps_J},
\end{align}
and define the G\^ateaux derivative of $\mathcal{J}$ \ref{orig_mod_obj} as 
\begin{align}\label{app_gat_def}
    &\delta \mathcal{J}\Bigl(\{\lambda_t\}_{t\in[0,+\infty)},\{u_t\}_{t\in[0,+\infty)},\{\mu_t\}_{t\in[0,+\infty)},\taum,\eta\Bigr) = \nonumber \\
    &\lim_{\eps\ra 0} \frac{1}{\eps} \Delta^\eps \mathcal{J}\Bigl(\{\lambda_t\}_{t\in[0,+\infty)},\{u_t\}_{t\in[0,+\infty)}, \{\mu_t\}_{t\in[0,+\infty)},\mu_{\taum},\taum,\eta\Bigr).
\end{align}
Based on our observation \eqref{mod_obj_L_trick}, also define,
\begin{align}
&\mathcal{J}_{\rm int}^\eps\bigl(t,x,\dot x,u_t^\eps,\lambda^\eps,\taum^\eps\bigr)= \onepe  L(t, x, u_t^\eps)   \nonumber \\
& \hspace{2 em}- \lambda^\eps(t,x)\t \, (\dot x(t,x) - \onepe f(t,x,u_t^\eps)) + \trace\left(\frac{\partial }{\partial x}\lambda^\eps(t,x)\frac{\sigma_t\sigma_t\t }{2}\right), \\
&\mathcal{J}_{\rm end}^\eps(t,\taum^{\eps},\eta^\eps) = \mathbb{E}_{0,x \sim \mu_0} \Big[\Phi\bigl({\taum^{\eps}(x_0)}, x_{{\taum^\eps}(x_0)}\bigr) - \eta^\eps \Psi\bigl({{\taum^\eps}(x_0)}, x_{{\taum^\eps}(x_0)}\bigr) \Big].
\label{J_end_eps}
\end{align}
Then, it follows from definition \eqref{d_eps_J},
\begin{align}
\Delta^\eps \mathcal{J}(\cdot) &= \int_0^{+\infty}\int_{{\rm S}_{\mu_t}} (\mathcal{J}_{\rm int}^\eps(\cdot) \, \d \mu^\eps_t(x)-\mathcal{J}_{\rm int}^0(\cdot) \, \d \mu_t(x)) \, \d t + \mathcal{J}_{\rm end}^\eps(\cdot)-\mathcal{J}_{\rm end}^0(\cdot)\\
&=\underbrace{\int_0^{+\infty} \int_{{\rm S}_{\mu_t}} (\mathcal{J}_{\rm int}^\eps(\cdot) -\mathcal{J}_{\rm int}^0(\cdot)) \, \d \mu_t(x)) \, \d t}_{\text{variation with fixed measure}} \nonumber \\
&+\underbrace{\int_0^{+\infty} \int_{{\rm S}_{\mu_t}} \mathcal{J}_{\rm int}^\eps(\cdot) \, (\d \mu_t^\eps(x)-\d \mu_t(x)) \, \d t}_{\text{variation wrt. measure}} \,+ \,\mathcal{J}_{\rm end}^\eps(\cdot)-\mathcal{J}_{\rm end}^0(\cdot) \label{d_e_J_2}.
\end{align}
From \eqref{app_gat_def}, the G\^ateaux derivative of $\mathcal{J}$ can be evaluated by applying a Taylor series expansion in terms of $\eps$ for the variations. Note that $\eps \to \mu_t^\eps$ is differentiable in weak sense for any sufficiently smooth test function $\xi:\mathbb{R}^n\to \mathbb{R}$, e.g. satisfying \ref{A4}. Then, Taylor series expansion can be applied in terms of $\eps$ to obtain,
\begin{align*}
    \int_{{\rm S}_{\mu_t}} \xi(x) \, \d \mu_t^\eps(x) = &\int_{{\rm S}_{\mu_t}} \xi(x) \, \d \mu_t(x) +\eps \int_{{\rm S}_{\mu_t}} \xi(x) \, \d (\delta \mu_t)(x)+o(\eps),
\end{align*}
where $o(\eps)$ denotes the higher-order terms in the Taylor series expansion. 
In a similar fashion, with the regularity assumptions \ref{A1}-\ref{A4}, dominated convergence theorem implies that we can pass the limit $\eps\to 0$ inside all time-space integrals. 

Define the fixed-measure variation term from the first expression in \eqref{d_e_J_2} as,
\begin{align} 
    &\delta \mathcal{J}_{\text{int}}(\cdot) =\lim_{\eps\to 0} \frac{1}{\eps}\bigl(\mathcal{J}_{\text{int}}^\eps(\cdot)-\mathcal{J}_{\text{int}}^0(\cdot)\bigr), \label{def_J_int_0_var} \\
    &\delta \mathcal{J}_{\text{end}}(\cdot) =\lim_{\eps\to 0}\frac{1}{\eps}\bigl(\mathcal{J}_{\rm end}^\eps(\cdot)-\mathcal{J}_{\rm end}^0(\cdot)\bigr) \label{def_J_end_0_var},
\end{align}
where we suppressed arguments for better readability. 
Therefore, applying a Taylor series expansion for the rest of the terms from \eqref{d_e_J_2}, \eqref{app_gat_def} becomes,
\begin{align} \label{del_J_fin_step}
&\delta \mathcal{J}\bigl(\{\lambda_t\}_{t\in[0,+\infty)},\{u_t\}_{t\in[0,+\infty)},\{\mu_t\}_{t\in[0,+\infty)},\taum,\eta\bigr) = \nonumber \\
&\quad =\int_0^{+\infty} \int_{{\rm S}_{\mu_t}} \mathcal{J}_{\rm int}^0(\cdot) \, \d (\delta \mu_t)(x) \, \d t +\int_0^{+\infty} \int_{{\rm S}_{\mu_t}} \delta \mathcal{J}_{\rm int}(\cdot) \, \d \mu_t(x) \, \d t  + \delta \mathcal{J}_{\text{end}}(\cdot). 
\end{align}

\subsection{Evaluation of the probability measure-flow variation}
For the first term from \eqref{del_J_fin_step}, observe that, the differentiation with respect to the Heaviside for the stopped running integral contributes the integrand evaluated at the stopping time (i.e. \eqref{heav_der}). Then, we use \eqref{mu_var_rel} together with the identities \eqref{appdx_idnt}, \eqref{app_chain_dot_x}, to obtain,
\begin{align} \label{del_J_fin_step2}
    &\int_0^{+\infty} \int_{{\rm S}_{\mu_t}} \mathcal{J}_{\rm int}^0(\cdot) \, \d (\delta \mu_t)(x) \, \d t =\int_0^{+\infty} \int_{{\rm S}_{\mu_t}} \Big \{\onep \left (\frac{\partial}{\partial x} L(t,x,u_t(x)) \nonumber \right. \Big. \\
    &\qquad \Big. \left. +\lambda(t,x)\t\frac{\partial}{\partial x} f(t,x,u_t(x)) + \frac{\partial}{\partial x}\trace\left(\frac{\partial }{\partial x}\lambda(t,x)\frac{\sigma_t\sigma_t\t }{2}\right) \right. \nonumber \\
    &\left. \qquad -\frac{\partial}{\partial x}\lambda(t,x)\t (\dot x(t,x) - \onep f(t,x,u_t(x))) \right. \Big. \nonumber \\
    &\left. \qquad \left(\frac{\partial}{\partial u} L(t,x,u_t(x))+\lambda(t,x)\t\frac{\partial}{\partial u} f(t,x,u_t(x)) \right) \frac{\partial}{\partial x}u_t(x) \right)\delta x(t,x) \Big. \nonumber \\
    &\qquad\Big. -\lambda(t,x)\t \delta \dot x(t,x)  +\Bigg[ L(t, x, u_t(x))   \nonumber \Big. \Big. \\
    &\qquad \Big. \Big. - \lambda(t,x)\t (\dot x(t,x) - f(t,x,u_t(x)))+\trace\left(\frac{\partial }{\partial x}\lambda(t,x)\frac{\sigma_t\sigma_t\t }{2}\right) \Bigg] \Big. \nonumber \\
& \qquad\Big.\quad \underbrace{\ \times \, \delta\bigl(t=\taum(\Tflow{v}{-t}(x))) \nabla_x \taum(\Tflow{v}{-t}(x)\bigr)  \delta x(t,x)}_{\text{surface/jump term}} \Big\} \d \mu_t(x) \, \d t.
\end{align}
\subsection{Evaluation of the fixed-measure and endpoint variations}

\subsubsection{Evaluation of the fixed-measure term}
To analyze the second term from \eqref{del_J_fin_step}, we start with the following identity at any time $t\ge 0$,
\begin{align} \label{lambd_perturb}
    &\int_{{\rm S}_{\mu_t}} \trace \Bigg(\frac{\partial }{\partial x}\delta\lambda(t,x)\t\frac{\sigma_t\sigma_t\t }{2}\Bigg) \, \d \mu_t(x)  =\int_{{\rm S}_{\mu_t}} \trace \Bigg(\frac{\partial }{\partial x}\delta\lambda(t,x)\t\frac{\sigma_t\sigma_t\t }{2}\Bigg) \rho(t,x)\, \d x, \nonumber \\
    &\overset{\text{div. thm.}}{=}-\int_{{\rm S}_{\mu_t}} \delta\lambda(t,x)\t\frac{\sigma_t\sigma_t\t }{2} \nabla_x \log\rho(t,x)\, \d x.
\end{align}
Focusing on the integrand $\delta \mathcal{J}_{\text{int}}$ \eqref{def_J_int_0_var}, and using the identity \eqref{lambd_perturb}, the first-order Taylor expansion around the nominal point yields,
\begin{align}
\delta \mathcal{J}_{\text{int}}(\cdot) 
&=
\onep \left (\frac{\partial}{\partial u} L(t,x_t,u_t(x_t))+\lambda(t,x_t)\t\frac{\partial}{\partial u}  f(t,x_t,u_t(x_t)) \right)\,\delta u_t(x_t) \nonumber
\\
&\quad
- \delta\lambda(t,x_t)\t\big(\dot x(t,x_t)-\onep \tilde f(t,x_t,u_t(x_t))\big) \nonumber \\
& \quad +\Big[ L(t, x_t, u_t(x_t)) - \lambda(t,x_t)\t \big(
\dot x(t,x_t) - f(t,x_t,u_t(x_t)) \big) \nonumber \\
& \quad  +\trace\big(\frac{\partial }{\partial x}\lambda(t,x_t)\frac{\sigma_t\sigma_t\t }{2}\big)\Big]  
\delta \bigl(t=\taum(\Tflow{v}{-t}(x_t)))\delta \taum(\Tflow{v}{-t}(x_t)\bigr),
\end{align}
where the last term comes from the variation with respect to the $\taum$. 

\subsubsection{Evaluation of the endpoint terms} Lastly, for the third term $\delta \mathcal{J}_{\text{end}}$ from \eqref{del_J_fin_step}, using the definitions \eqref{J_end_eps}, \eqref{def_J_end_0_var}, and the final state variation \eqref{term_map_perturb}, under the already established regularity conditions, we obtain
\begin{align*}
    \delta \mathcal{J}_{\text{end}}(\cdot) &=\mathbb{E}_{0,x \sim \mu_0} \Big[\frac{\partial}{\partial t}(\Phi({\taum(x_0)}, x_{{\taum}(x_0)}) -  \eta \Psi({{\taum}(x_0)}, x_{{\taum}(x_0)}))\delta \taum(x_0) \nonumber \Big. \\
    &\Big. +\frac{\partial}{\partial x}(\Phi({\taum(x_0)}, x_{{\taum}(x_0)}) - \eta \Psi({{\taum}(x_0)}, x_{{\taum}(x_0)}))\delta\Tflow{v}{\taum(x_0)} \Big. \nonumber \\
    & \Big.  +\Psi({{\taum}(x_0)}, x_{{\taum}(x_0)}) \delta \eta \Big].
\end{align*}

%
\subsection{Collection and simplification of first-order terms}
\subsubsection{Collecting first-order terms} For compactness of notation, we introduce the following functions,
\begin{subequations}
\begin{align}
&J_1(t,x_t,\lambda_t,u_t)\Let \frac{\partial }{\partial u}  L(t,x_t,u_t) + \lambda_t\t  \frac{\partial}{\partial u}  f(t,x_t,u_t), \label{J1}\\
&J_2(t,x_t,\lambda_t,u_t) \Let \frac{\partial}{\partial x} L(t,x_t,u_t) + \lambda_t\t  \frac{\partial}{\partial x}  f(t,x_t,u_t) \nonumber \\
& \qquad+\frac{\partial}{\partial x}\trace\left(\frac{\partial }{\partial x}\lambda_t\frac{\sigma_t\sigma_t\t }{2}\right) \nonumber \\
& \qquad -\frac{\partial}{\partial x}\lambda_t\t (\dot x(t,x) - \onep f(t,x,u_t)) \nonumber \\
& \qquad + \left(\frac{\partial}{\partial u} L(t,x_t,u_t)+\lambda_t\t\frac{\partial}{\partial u} f(t,x_t,u_t) \right) \frac{\partial}{\partial x}u_t \label{J2_old}, \\
&J_3(t,x_t,\lambda_t,u_t,\eta)\Let \lambda_t\t  (\tilde f(t,x_t,u_t)-\dot x(t,x_t))  \nonumber \\
& \qquad \qquad +  L(t,x_t,u_t)+ \frac{\partial }{\partial t} \Phi(t,x_t)- \eta \frac{\partial }{\partial t} \Psi(t, x_t), \label{J3}\\
& J_4(t,x_t,\eta) \Let \frac{\partial }{\partial x}\Phi(t, x_t) - \eta  \frac{\partial }{\partial x}\Psi(t, x_t) \label{J4}.
\end{align}
\end{subequations}

Collecting terms, together with \eqref{J1}-\eqref{J4}, we end up with 
\begin{align} \label{zero_var_rhs}
&\delta \mathcal{J}(\cdot)  =\int_0^{+\infty} \int_{{\rm S}_{\mu_t}} \Big [\onep J_1(t,x,\lambda(t,x),u_t(x))\delta u_t(x)  \nn \\
&  + \onep J_2(t,x,\lambda(t,x),u_t(x)) \delta x(t,x) -  \bigl(\dot x(t,x) - \onep \tilde f(t,x,u_t(x))\bigr)\t \delta \lambda_t(x) \nonumber \\
& -\lambda(t,x)\t \delta \dot x(t,x) \Big] \, \d \mu_t(x) \, \d t +\int_{{\rm S}_{\mu_0}}  J_3\Bigl( \taum(x),\Tflow{v}{\taum(x)}(x),\lambda\bigl(\taum(x),\Tflow{v}{\taum(x)}(x)\bigr),u_{\taum(x)},\eta\Bigr) \nonumber \\& \qquad\qquad\delta{\taum}(x)\, \d \mu_0(x) + \int_{{\rm S}_{\mu_0}} \Psi\bigl({{\taum}(x)}, \Tflow{v}{\taum(x)}(x)\bigr) \delta \eta \, \d \mu_0(x) +  \nonumber \\
&\qquad \qquad \int_{{\rm S}_{\mu_0}}  J_4\bigl(\taum(x),\Tflow{v}{\taum(x)}(x),\eta\bigr) \, \delta \Tflow{v}{\taum(x)}(x) \, \d \mu_{0}(x).
\end{align} 
Similarly, \
\begin{align} \label{first_var_rhs}
\delta \mathcal{J}(\cdot)  &= \mathbb{E}_{0,x \sim \mu_0} \Bigg\{\int_0^{\taum(x_0)} \Big [J_1(t,x_t,\lambda(t,x_t),u_t(x_t))\delta u_t(x_t)  -  (\dot x(t,x_t) - \tilde f(t,x_t,u_t(x_t)))\t \delta \lambda_t(x_t) \Big. \Bigg. \nonumber \\
& \Bigg. \Big. \qquad +  J_2(t,x_t,\lambda(t,x_t),u_t(x_t)) \delta x(t,x_t) -\lambda(t,x_t)\t \delta \dot x(t,x_t) \Big] \, \d t \Bigg \}\, \nonumber \\
&+\int_{{\rm S}_{\mu_0}}  J_3\Bigl(\taum(x),\Tflow{v}{\taum(x)}(x), \lambda\bigl(\taum(x),\Tflow{v}{\taum(x)}(x)\bigr),u_{\taum(x)}(\Tflow{v}{\taum(x)}(x)),\eta\Bigr) \, \delta{\taum}(x) \, \d \mu_0(x), \nonumber \\
& + \int_{{\rm S}_{\mu_0}}  J_4\bigl(\taum(x),\Tflow{v}{\taum(x)}(x),\eta\bigr) \, \delta \Tflow{v}{\taum(x)}(x) \, \d \mu_{0}(x) \nonumber \\
& + \int_{{\rm S}_{\mu_0}} \Psi\bigl({{\taum}(x)}, \Tflow{v}{\taum(x)}(x)\bigr) \delta \eta \, \d \mu_0(x).
\end{align} 

In the derivation of \eqref{J1}, \eqref{J2_old} and \eqref{J4}, we used the identity \eqref{appdx_idnt}, and for the variation of $\taum$, we used the equivalent formulation for the objective functional, \eqref{appdx_equiv}, by applying Lebesgue–Besicovitch Differentiation Theorem \cite[Thm. 1.32]{evans_measure} over the initial probability measure $\mu_0$, where it was defined over. 
Finally, the third term, i.e., $\lambda(t,x)\t \delta \dot x(t,x)$ is from \eqref{zero_var_rhs} (and \eqref{first_var_rhs}), is a direct result of \eqref{app_chain_dot_x}.

Notice that we didn't include the jump term from \eqref{del_J_fin_step2} when defining $J_1$ \eqref{J1}, as the trajectories are stopped and not defined afterwards. 

\subsubsection{Simplification of the collected terms}
Observe that, for any variation $\delta \lambda_t$, the constraint $\dot x(t,x)=\onep \tilde f(t,x,u_t(x))$ holds for every admissible triple $(t,x,u_t(x))$. Furthermore, since $J_1$ in \eqref{J1} vanishes, the last term in $J_2$ in \eqref{J2_old} also vanishes. Therefore, $J_2$ from \eqref{J2_old} can be rewritten as,
\begin{align}
J_2(t,x_t,\lambda_t,u_t) \Let &\frac{\partial}{\partial x} L(t,x_t,u_t) + \lambda_t\t  \frac{\partial}{\partial x}  f(t,x_t,u_t) \nonumber \\
& \quad +\frac{\partial}{\partial x}\trace\left(\frac{\partial }{\partial x}\lambda_t\frac{\sigma_t\sigma_t\t }{2}\right)  + \frac{\partial}{\partial x}\lambda_t\t \frac{\sigma_t\sigma_t\t }{2} \nabla_x \log\rho(t,x_t)  \label{J2}.
\end{align}

The forth term from \eqref{zero_var_rhs}, (equivalently, from \eqref{first_var_rhs}) related to $\delta \dot x(t,x)$ simplifies as
\begin{align}
&\hspace{-4 em}- \mathbb{E}_{0,x \sim \mu_0}\Big[\int_0^{\taum(x_0)} \lambda(t,x_t)\t \delta \dot x(t,x_t) \, \d t\Big] \nonumber \\
&\hspace{-2 em}=-\mathbb{E}_{0,x \sim \mu_0}[\lambda(\taum(x_0),x_\taum)\t \delta x(\taum(x_0),x_{\taum(x_0)})] \nonumber \\
&\hspace{-2 em}\ \quad+ \mathbb{E}_{0,x\sim \mu_0} \Big[ \int_0^{\taum(x_0)} \dot \lambda(t,x_t)\t \delta x(t,x_t) \, \d t \Big], \nonumber \\
&\hspace{-2 em}=-\int_{{\rm S}_{\mu_0}}\lambda(\taum(x),\Tflow{v}{t}(\taum(x)))\t \delta \Tflow{v}{\taum(x)} \, \d \mu_{0}(x) \nonumber \\
&\hspace{-2 em}+\int_{{\rm S}_{\mu_0}} \lambda(\taum(x),\Tflow{v}{t}(\taum(x)))\t \dot x(t,x) \delta \taum(x) \, \d \mu_0(x) \nonumber \\
&\hspace{-2 em}\ \quad+ \mathbb{E}_{0,x\sim \mu_0} \Big[ \int_0^{\taum(x_0)} \dot \lambda(t,x_t)\t \delta x(t,x_t) \, \d t \Big],\label{appdx_light2}
\end{align}
where $\dot\lambda(t,x_t)\Let\frac{d}{dt}\lambda(t,x_t)$ is the derivative of the map \(t\mapsto\lambda(t,x_t)\). Integration by parts for each path from the initial distribution $\mu_0$ was applied in the first line, and in the second line, we made use of  the identity \eqref{term_map_perturb}.

\subsection{Derivation of the necessary conditions}
Let us define the simplified functions
\begin{subequations}
\begin{align} 
&\bar J_1(t,x_t,\lambda_t,u_t)\Let \frac{\partial }{\partial u} L(t,x_t,u_t) + \lambda_t\t  \frac{\partial}{\partial u} f(t,x_t,u_t) \label{bar_J1}, \\
&\bar J_2(t,x_t,\lambda_t,u_t,\rho_t) \Let \dot \lambda_t\t+\frac{\partial}{\partial x} L(t,x_t,u_t) \nonumber \\
& \qquad \qquad+ \lambda_t\t  \frac{\partial}{\partial x} f(t,x_t,u_t) +\frac{\partial}{\partial x}\trace\left(\frac{\partial }{\partial x}\lambda_t\frac{\sigma_t\sigma_t\t }{2}\right) \nonumber \\
& \qquad \qquad+ \frac{\partial}{\partial x} \lambda_t \t \frac{\sigma_t \sigma_t\t}{2} \nabla_x \log \rho_t  \label{bar_J2}, \\
&\bar J_3(t,x_t,\dot x_t,\lambda_t,u_t,\eta)\Let \lambda_t\t  f(t,x_t,u_t)+ \trace\left(\frac{\partial }{\partial x}\lambda_t\frac{\sigma_t\sigma_t\t }{2}\right)  \nonumber \\
& \qquad \qquad + L(t,x_t,u_t)+ \frac{\partial }{\partial t} \Phi(t,x_t)- \eta \frac{\partial }{\partial t} \Psi(t, x_t),\label{bar_J3} \\
&\bar J_4(t,x_t,\lambda_t,\eta) \Let -\lambda_t+\frac{\partial }{\partial x}\Phi(t, x_t) - \eta  \frac{\partial }{\partial x}\Psi(t, x_t) \label{bar_J4}.
\end{align}
\end{subequations}
Finally, in the light of the simplifications from \eqref{J2} and \eqref{appdx_light2}, and the quantitites in \eqref{bar_J1}--\eqref{bar_J4}, the functional derivative \eqref{first_var_rhs} becomes,
\begin{align} \label{first_var_rhs_2}
\delta \mathcal{J}(\cdot) 
&= \mathbb{E}_{0,x \sim \mu_0} \Bigg\{\int_0^{\taum(x_0)} \Big [\bar J_1(t,x_t,\lambda(t,x_t),u_t(x_t))\delta u_t(x_t)   -  (\dot x(t,x_t) - \tilde f(t,x,u_t(x_t)))\t \delta \lambda_t(x_t) \Big. \Bigg. \nonumber \\
& \Bigg. \Big. \qquad +  \bar J_2(t,x_t,\lambda(t,x_t),u_t(x_t),\rho(t,x_t)) \delta x(t,x_t)  \Big] \, \d t \Bigg \}\, \nonumber \\
&+\int_{{\rm S}_{\mu_0}} \bar J_3\Bigl(\taum(x),\Tflow{v}{\taum(x)}(x),\dot x\bigl(\taum(x),\Tflow{v}{\taum(x)}(x)\bigr), \lambda\bigl(\taum(x),\Tflow{v}{\taum(x)}(x)\bigr),u_{\taum(x)}(\Tflow{v}{\taum(x)}(x)),\eta\Bigr)\nonumber \\
& \delta{\taum}(x) \, \d \mu_0(x) + \int_{{\rm S}_{\mu_0}} \bar J_4(\taum(x),\Tflow{v}{\taum(x)}(x),\lambda(\taum(x),\Tflow{v}{\taum(x)}(x)),\eta) \, \delta \Tflow{v}{\taum(x)}(x) \, \d \mu_{0}(x) \nonumber \\
& + \int_{{\rm S}_{\mu_0}} \Psi({{\taum}(x)}, \Tflow{v}{\taum(x)}(x)) \delta \eta \, \d \mu_0(x).
\end{align} 

From \eqref{first_var_rhs_2}, the first variation of the augmented functional is expressed as a sum of terms multiplying the admissible perturbations $\delta u_t$, $\delta \lambda_t$, $\delta x$, $\delta \taum$, $\delta \eta$, and $\delta\Tflow{v}{\taum(x)}$. Since these perturbations can be chosen independently, the vanishing of the first variation for every admissible perturbation implies, by the Fundamental Lemma of the Calculus of Variations, that each corresponding coefficient in \eqref{first_var_rhs_2} must vanish separately. In particular, the coefficients of $\delta \lambda_t$ and $\delta \eta$ yield the state dynamics \ref{eq:char_star} and the distribution constraint \ref{OC9}, respectively. Moreover, the coefficients of $\delta u_t$, $\delta x$, $\delta \taum$, and $\delta\Tflow{v}{\taum(x)}$ yield, respectively, the optimal control condition $\bar J_1(\cdot)=0$ in \eqref{bar_J1}, the co-state dynamics along the characteristic $\bar J_2(\cdot)=0$ in \eqref{bar_J2}, the free-terminal-time (stopping) condition $\bar J_3(\cdot)=0$ in \eqref{bar_J3}, and the terminal transversality condition for the co-state $\bar J_4(\cdot)=0$ in \eqref{bar_J4}. Expressed along the optimal characteristic trajectory, these relations are precisely the first-order necessary conditions \ref{OC7}--\ref{OC8}.

\subsection{Common stopping as a special case}

A special case arises when the stopping strategy is common to all realizations, that is, when $\taum$ is a constant function. In this case, the associated optimality condition is obtained by integrating, with respect to $\mu_0$, the coefficient multiplying $\delta\taum$ in \eqref{first_var_rhs_2}. This amounts to taking the expectation of \eqref{bar_J3}, yielding \ref{OC4}.

This completes the proof.

\bibliographystyle{amsalpha}
\bibliography{refs}

@article{wass_filippov,
title = {Differential inclusions in {Wasserstein} spaces: The {Cauchy-Lipschitz} framework},
journal = {Journal of Differential Equations},
volume = {271},
pages = {594-637},
year = {2021},
issn = {0022-0396},
NOTE = {doi: \url{https://doi.org/10.1016/j.jde.2020.08.031}},
author = {B. Bonnet and H. Frankowska},
}

@article {aronson,
    AUTHOR = {D. G. Aronson},
     TITLE = {Bounds for the fundamental solution of a parabolic equation},
  JOURNAL = {Bulletin of the American Mathematical Society},
    VOLUME = {73},
      YEAR = {1967},
     PAGES = {890--896},
       NOTE = {doi: \url{https://doi.org/10.1090/S0002-9904-1967-11830-5}},
}

@article{aronson2,
  title={Non-negative solutions of linear parabolic equations},
  author={D. G. Aronson},
  journal={Annali Della Scuola Normale Superiore Di Pisa-classe Di Scienze},
  year={1968},
  volume={22},
  number= {4},
  pages={607-694},
  NOTE = {URL: \url{https://www.numdam.org/item/ASNSP_1968_3_22_4_607_0}},
}

@book {Pavliotis,
    AUTHOR = {G. A. Pavliotis},
     TITLE = {Stochastic {P}rocesses and {A}pplications: {D}iffusion {P}rocesses, the {F}okker--{P}lanck and {L}angevin {E}quations},
    SERIES = {Texts in Applied Mathematics},
    VOLUME = {60},
publisher = {Springer},
address   = {New York, NY},
      YEAR = {2014},
     PAGES = {xiv+339},
       NOTE = {doi: \url{https://doi.org/10.1007/978-1-4939-1323-7}},
}

@book {book_gr,
    AUTHOR = {L. Ambrosio and N. Gigli and G. Savar\'e},
     TITLE = {Gradient {F}lows: {I}n {M}etric {S}paces and in the {S}pace of {P}robability {M}easures},
    SERIES = {Lectures in Mathematics ETH Z\"urich},
   EDITION = {Second},
 PUBLISHER = {Birkh\"auser Verlag},
 address = {Basel},
      YEAR = {2008},
     PAGES = {x+334},
     NOTE = {doi: \url{https://doi.org/10.1007/978-3-7643-8722-8}},  
}

@book {book_dynkin,
    AUTHOR = {Dynkin, E. B.},
     TITLE = {Markov {P}rocesses, {V}ol. {I}},
    SERIES = {Die Grundlehren der mathematischen Wissenschaften},
    VOLUME = {121},
publisher = {Springer-Verlag},
address   = {Berlin, Heidelberg},
      YEAR = {1965},
     PAGES = {xii+366},
  NOTE = {doi: \url{https://doi.org/10.1007/978-3-662-00031-1}},
}

@article{stein1,
author = {C. M. Stein},
title = {Estimation of the Mean of a Multivariate Normal Distribution},
volume = {9},
journal = {The Annals of Statistics},
number = {6},
publisher = {Institute of Mathematical Statistics},
pages = {1135--1151},
year = {1981},
NOTE = {doi: \url{https://doi.org/10.1214/aos/1176345632}},
}

@incollection{SEP7,
    AUTHOR = {H. Rost},
     TITLE = {Skorokhod stopping times of minimal variance},
 BOOKTITLE = {S\'eminaire de {P}robabilit\'es, {X}},
    SERIES = {Lecture Notes in Math.},
    VOLUME = {511},
     PAGES = {194--208},
 PUBLISHER = {Springer},
 address = {Berlin-New York},
      YEAR = {1976},
      NOTE = {doi: \url{https://www.numdam.org/item/SPS_1976__10__194_0}}
}

@inproceedings{stein2,
author = {Q. Liu and J. D. Lee and M. Jordan},
title = {A Kernelized {S}tein Discrepancy for Goodness-of-fit Tests},
year = {2016},
booktitle = {Proceedings of the 33rd International Conference on International Conference on Machine Learning},
volume = {48},
pages = {276–284},
numpages = {9},
address = {New York, NY},
NOTE = {URL: \url{https://proceedings.mlr.press/v48/liub16.html}},
}

@book{villani_ot,
  author    = {C. Villani},
  title     = {Optimal {T}ransport: {O}ld and {N}ew},
  publisher = {Springer Berlin Heidelberg},
  address   = {Berlin, Heidelberg},
  series    = {Grundlehren der mathematischen Wissenschaften},
  edition   = {1},
  year      = {2009},
  isbn      = {978-3-540-71049-3},
  NOTE       = {doi: \url{https://doi.org/10.1007/978-3-540-71050-9}}
}

@book{book_ge3,
    AUTHOR = {C. Villani},
     TITLE = {Topics in {O}ptimal {T}ransportation},
    SERIES = {Graduate Studies in Mathematics},
    VOLUME = {58},
 publisher = {American Mathematical Society},
 address = {Providence, RI},
      YEAR = {2003},
     PAGES = {xvi+370},
     NOTE = {doi: \url{https://doi.org/10.1090/gsm/058}},
}

@book{fleming_soner,
  author    = {W. H. Fleming and H. M. Soner},
  title     = {Controlled {M}arkov {P}rocesses and {V}iscosity {S}olutions},
  series    = {Stochastic Modelling and Applied Probability},
  volume    = {25},
  edition   = {2},
  publisher = {Springer New York},
  address   = {New York},
  year      = {2006},
  NOTE       = {doi: \url{https://doi.org/10.1007/0-387-31071-1}},
}

@article{otto,
author = {Otto, F.},
title = {The Geometry Of Dissipative Evolution Equations: The Porous Medium Equation},
journal = {Communications in Partial Differential Equations},
volume = {26},
number = {1-2},
pages = {101--174},
year = {2001},
publisher = {Taylor \& Francis},
NOTE = {doi: \url{https://doi.org/10.1081/PDE-100002243}},
}

@book{SEP1,
  author = {A. V. Skorokhod},
  title = {Issledovaniya po Teorii Sluchainykh Protsessov},
  publisher = {Izd-vo Kievskogo un-ta},
  address = {Kiev},
  year = {1961}
}

@book {SEP-eng,
    AUTHOR = {A. V. Skorokhod},
     TITLE = {Studies in the {T}heory of {R}andom {P}rocesses},
 PUBLISHER = {Dover Publications},
      YEAR = {2017},
      address = {Mineola, New York},
     PAGES = {viii+208},
     NOTE = {URL: \url{https://tinyurl.com/4xjsemch}},
}

@article {SEP2,
    AUTHOR = {S. Jaimungal and A. Kreinin and A. Valov},
     TITLE = {The generalized {S}hiryaev problem and {S}korokhod embedding},
  JOURNAL = {Theory of Probability and its Applications},
    VOLUME = {58},
      YEAR = {2014},
    NUMBER = {3},
     PAGES = {493--502},
     NOTE = {doi: \url{https://doi.org/10.1137/S0040585X97986734}},
}

@article {SEP3,
    AUTHOR = {D. H. Root},
     TITLE = {The existence of certain stopping times on {B}rownian motion},
  JOURNAL = {Annals of Mathematical Statistics},
    VOLUME = {40},
     number = {2},
      YEAR = {1969},
     PAGES = {715--718},
       NOTE = {doi: \url{https://doi.org/10.1214/aoms/1177697749}},
}

@article {SEP4,
    AUTHOR = {M. G. C. Alexander and J. Wang},
     TITLE = {Root's barrier: construction, optimality and applications to variance options},
  JOURNAL = {The Annals of Applied Probability},
    VOLUME = {23},
      YEAR = {2013},
    NUMBER = {3},
     PAGES = {859--894},
       NOTE = {doi: \url{https://doi.org/10.1214/12-AAP857}},
}

@article {SEP5,
    AUTHOR = {P. Gassiat and H. Oberhauser and C. Z. Zou},
     TITLE = {A free boundary characterisation of the {R}oot barrier for
              {M}arkov processes},
  JOURNAL = {Probability Theory and Related Fields},
    VOLUME = {180},
      YEAR = {2021},
    NUMBER = {1-2},
     PAGES = {33--69},
       NOTE = {doi: \url{https://doi.org/10.1007/s00440-021-01052-6}},
}

@article {SEP6,
    AUTHOR = {H. Rost},
     TITLE = {The stopping distributions of a {M}arkov {P}rocess},
  JOURNAL = {Inventiones Mathematicae},
    VOLUME = {14},
      YEAR = {1971},
     PAGES = {1--16},
       NOTE = {doi: \url{https://doi.org/10.1007/BF01418740}},
}

@article{SEP9,
     author = {J. Az\'ema and M. Yor},
     title = {A simple solution to the {S}korokhod problem},
     journal = {S{\'e}minaire de probabilit{\'e}s de Strasbourg},
      volume = {17},
  year = {1983},
  pages = {221-224},
     publisher = {Springer - Lecture Notes in Mathematics},
     language = {fr},
     NOTE = {doi: \url{https://www.numdam.org/item/SPS_1979__13__90_0}}
}

@incollection{SEP11,
  author = {E. Perkins},
  title = {The {C}ereteli--{D}avis Solution to the \({H}^{1}\)-Embedding Problem and an Optimal Embedding in {B}rownian Motion},
  booktitle = {Seminar on Stochastic Processes, 1985. Progress in Probability and Statistics},
  volume = {12},
  year = {1985},
  pages = {172-223},
  NOTE = {doi: \url{https://doi.org/10.1007/978-1-4684-6748-2_12}},
}

@article{vis_comp,
author = {M. Talbi and N. Touzi and J. Zhang},
title = {Viscosity Solutions for Obstacle Problems on {W}asserstein Space},
journal = {SIAM Journal on Control and Optimization},
volume = {61},
number = {3},
pages = {1712-1736},
year = {2023},
NOTE = {doi: \url{https://doi.org/10.1137/22M1488119}},
}

@article{class_ver,
author = {Talbi, M. and Touzi, N. and Zhang, J.},
title = {Dynamic Programming Equation for the Mean Field Optimal Stopping Problem},
journal = {SIAM Journal on Control and Optimization},
volume = {61},
number = {4},
pages = {2140-2164},
year = {2023},
NOTE = {doi: \url{https://doi.org/10.1137/21M1404259}},
}

@misc{SEP12,
      title={Perkins Embedding for General Starting Laws}, 
      author={A. Grass},
      journal={arXiv preprint, arXiv:2307.03618},
      year={2023},
      NOTE = {URL: \url{https://arxiv.org/abs/2307.03618}}, 
}

@article{SEP13,
title = {Embedding in Brownian motion with drift and the {A}zéma–{Y}or construction},
journal = {Stochastic Processes and their Applications},
volume = {85},
number = {2},
pages = {249-254},
year = {2000},
issn = {0304-4149},
NOTE = {doi: \url{https://doi.org/10.1016/S0304-4149(99)00077-0}},
author = {P. Grandits and N. Falkner},
}

@article{modern_bound,
title = {Density and gradient estimates for non degenerate Brownian {SDE}s with unbounded measurable drift},
journal = {Journal of Differential Equations},
volume = {272},
pages = {330-369},
year = {2021},
issn = {0022-0396},
NOTE  = {doi: \url{https://doi.org/10.1016/j.jde.2020.09.004}},
author = {S. Menozzi and A. Pesce and X. Zhang},
}

@article {SEP14,
    AUTHOR = {J. L. Pedersen and G. Peskir},
     TITLE = {The {A}z\'ema-{Y}or embedding in non-singular diffusions},
  JOURNAL = {Stochastic Processes and their Applications},
    VOLUME = {96},
      YEAR = {2001},
    NUMBER = {2},
     PAGES = {305--312},
       NOTE = {doi: \url{https://doi.org/10.1016/S0304-4149(01)00120-X}},
}

@article{SEP17,
  author = {G. Peskir},
  title = {Designing Options Given the Risk: the Optimal {S}korokhod-Embedding Problem},
  journal = {Stochastic Processes and their Applications},
  volume = {81},
  year = {1999},
  pages = {25-38},
  NOTE = {doi: \url{https://doi.org/10.1016/S0304-4149(98)00097-0}},
}

@article{SEP18,
  author = {N. Ghoussoub and Y. Kim and A. Palmer},
  title = {{PDE} Methods for Optimal {S}korokhod Embeddings},
  year = {2019},
  number = {3},
  pages = {113-143},
  volume = {58},
  journal = {Calculus of Variations and Partial Differential Equations},
 NOTE = {doi: \url{https://doi.org/10.1007/s00526-019-1563-7}},
}

@phdthesis{SEP19,
  author = {S. Kinsley},
  title = {Duality Methods for Barrier-type Solutions to the {S}korokhod Embedding Problem},
  school = {University of Bath},
  year = {2018},
  NOTE = {URL: \url{https://purehost.bath.ac.uk/ws/portalfiles/portal/187964586/Thesis.pdf}},
}

@incollection {SEP21,
    AUTHOR = {N. Falkner},
     TITLE = {On {S}korohod embedding in {$n$}-dimensional {B}rownian motion
              by means of natural stopping times},
 BOOKTITLE = {Seminar on {P}robability, {XIV} ({P}aris, 1978/1979)
              ({F}rench)},
    SERIES = {Lecture Notes in Math.},
    VOLUME = {784},
     PAGES = {357--391},
 PUBLISHER = {Springer},
 address   = {Berlin, Heidelberg},
      YEAR = {1980},
      NOTE = {doi: \url{https://www.numdam.org/item/SPS_1980__14__357_0/}}
}

@article{opt_stop_ex1,
 ISSN = {00219002},
 author = {E. Ekstr\"{o}m and H. Wanntorp},
 journal = {Journal of Applied Probability},
 number = {1},
 pages = {170-180},
 publisher = {Applied Probability Trust},
 title = {Optimal Stopping Of A {B}rownian Bridge},
 volume = {46},
 year = {2009},
 NOTE = {doi: \url{https://doi.org/10.1239/jap/1238592123}}
}

@article{SEP25,
  author = {A. M. G. Cox and S. M. Kinsley},
  title = {Discretisation and Duality of Optimal {S}korokhod Embedding Problems},
  journal = {Stochastic Processes and their Applications},
  volume = {129},
  number = {7},
  year = {2019},
  pages = {2376-2405},
  publisher = {Elsevier},
  NOTE = {doi: \url{https://doi.org/10.1016/j.spa.2018.07.008}},
}

@article {SEP29,
    AUTHOR = {J. Ob{\l}{\'o}j},
     TITLE = {The {S}korokhod embedding problem and its offspring},
  JOURNAL = {Probability Surveys},
    VOLUME = {1},
      YEAR = {2004},
     PAGES = {321--390},
       NOTE = {doi: \url{https://doi.org/10.1214/154957804100000060}},
}

@book {book1,
    AUTHOR = {{\O}ksendal, B.},
     TITLE = {Stochastic {D}ifferential {E}quations: {A}n {I}ntroduction with {A}pplications},
    SERIES = {Universitext},
   EDITION = {Sixth},
 PUBLISHER = {Springer-Verlag},
address   = {Berlin, Heidelberg},
      YEAR = {2003},
     PAGES = {xxiv+360},
       NOTE = {doi: \url{https://doi.org/10.1007/978-3-642-14394-6}},
}

@book {evans_measure,
    AUTHOR = {L. C. Evans and R. F. Gariepy},
     TITLE = {Measure {T}heory and {F}ine {P}roperties of {F}unctions},
    SERIES = {Textbooks in Mathematics},
   EDITION = {Revised},
 publisher = {CRC Press},
  address   = {Boca Raton, FL},
      YEAR = {2015},
     PAGES = {xiv+299},
     NOTE = {doi: \url{https://doi.org/10.1201/b18333}},
}

@book {lieberman,
    AUTHOR = {G. M. Lieberman},
     TITLE = {Second {O}rder {P}arabolic {D}ifferential {E}quations},
 PUBLISHER = {World Scientific Publishing Co., Inc.},
 address   = {River Edge, NJ},
      YEAR = {1996},
     PAGES = {xii+439},
       NOTE = {doi: \url{https://doi.org/10.1142/3302}},
}

@book {rockafellar,
    AUTHOR = {R. T. Rockafellar and R. J.-B. Wets},
     TITLE = {Variational {A}nalysis},
    SERIES = {Fundamental
              Principles of Mathematical Sciences},
    VOLUME = {317},
 PUBLISHER = {Springer-Verlag},
 address = {Berlin},
      YEAR = {1998},
     PAGES = {xiv+733},
       NOTE = {doi: \url{https://doi.org/10.1007/978-3-642-02431-3}},
}

@book {prob_mean_field,
    AUTHOR = {R. Carmona and F. Delarue},
     TITLE = {Probabilistic {T}heory of {M}ean {F}ield {G}ames with {A}pplications.
              {I}},
    SERIES = {Probability Theory and Stochastic Modelling},
    VOLUME = {83},
 PUBLISHER = {Springer},
 address = {Cham},
      YEAR = {2018},
     PAGES = {xxv+713},
     NOTE  = {doi: \url{https://doi.org/10.1007/978-3-319-58920-6}},
}

@ARTICLE{killing_tg1,
  author={A. Eldesoukey and O. M. Miangolarra and T. T. Georgiou},
  journal={IEEE Control Systems Letters}, 
  title={An Excursion Onto {S}chrödinger’s Bridges: Stochastic Flows With Spatio-Temporal Marginals}, 
  year={2024},
  volume={8},
  number={},
  pages={1138-1143},
  NOTE = {doi: \url{https://doi.org/10.1109/LCSYS.2024.3409107}}
}

@article{killing_tg2,
author = {Y. Chen and T. T. Georgiou and M. Pavon},
title = {Optimal Survival Strategies for Diffusive Flows: A {S}chrödinger Bridge Approach to Unbalanced Transport},
journal = {SIAM Review},
volume = {67},
number = {3},
pages = {579-604},
year = {2025},
NOTE = {doi: \url{https://doi.org/10.1137/25M176581X}},
}

@article{ref:teter2025probabilistic,
  title={Probabilistic {L}ambert problem: connections with optimal mass transport, {S}chr{\"o}dinger bridge, and reaction-diffusion {PDE}s},
  author={A. M. H. Teter and I. Nodozi and A. Halder},
  journal={SIAM Journal on Applied Dynamical Systems},
  volume={24},
  number={1},
  pages={16--43},
  year={2025},
  publisher={SIAM},
  NOTE = {doi: \url{https://doi.org/10.1137/24M1646145}}
}

@article{superposition, 
title={Continuity equations and {ODE} flows with non-smooth velocity},
volume={144}, 
number={6},
journal={Proceedings of the Royal Society of Edinburgh: Section A Mathematics}, author={L. Ambrosio and G. Crippa},
year={2014},
pages={1191–1244},
NOTE = {doi: \url{https://doi.org/10.1017/S0308210513000085},}}

@book {book_gr2,
    AUTHOR = {Ambrosio, L. and Dacorogna, B. and Mascolo, E. and Marcellini, P. and Caffarelli, L.A. and Crandall, M.G. and Evans, L.C. and Fusco, N.},
     TITLE = {Calculus of {V}ariations and {N}onlinear {P}artial {D}ifferential {E}quations},
    SERIES = {Lecture Notes in Mathematics},
    VOLUME = {1927},
    EDITOR = {Dacorogna, Bernard and Marcellini, Paolo},
 PUBLISHER = {Springer-Verlag},
 address = {Berlin Heidelberg},
      YEAR = {2008},
     PAGES = {xi+206},
       NOTE = {doi: \url{https://doi.org/10.1007/978-3-540-75914-0}},
}

@article{ref:YC:TG:MP:SB-and-OT,
  title={On the relation between optimal transport and {S}chr{\"o}dinger bridges: A stochastic control viewpoint},
  author={Y. Chen and T. T. Georgiou and M. Pavon},
  journal={Journal of Optimization Theory and Applications},
  volume={169},
  number={2},
  pages={671--691},
  year={2016},
  publisher={Springer},
  NOTE = {doi: \url{https://doi.org/10.1007/s10957-015-0803-z}},
}

@book {book_sc,
    AUTHOR = {J. Yong and X. Y. Zhou},
     TITLE = {Stochastic {C}ontrols: {H}amiltonian {S}ystems and {HJB} {E}quations},
    SERIES = {Stochastic Modelling and Applied Probability},
    VOLUME = {43},
 PUBLISHER = {Springer},
 address = {New York},
      YEAR = {1999},
     PAGES = {xxii+439},
       NOTE = {doi: \url{https://doi.org/10.1007/978-1-4612-1466-3}},
}

@book {book_stop,
    AUTHOR = {G. Peskir and A. Shiryaev},
     TITLE = {Optimal {S}topping and {F}ree-{B}oundary {P}roblems},
    SERIES = {Lectures in Mathematics ETH Z\"urich},
 PUBLISHER = {Birkh\"auser Verlag},
 address = {Basel},
      YEAR = {2006},
     PAGES = {xxii+500},
      NOTE = {doi: \url{https://doi.org/10.1007/978-3-7643-7390-0}},
}

@inproceedings{intro1,
author = {A. Saravanos and A. Tsolovikos and E.  Bakolas and E. Theodorou},
year = {2021},
pages = {},
title = {Distributed Covariance Steering with Consensus {ADMM} for Stochastic Multi-Agent Systems},
NOTE = {doi: \url{https://doi.org/10.15607/RSS.2021.XVII.075}},
booktitle = {Robotics: Science and Systems 2021},
}

@article{intro2,
author = {C. Greco and S. Campagnola and M. Vasile},
title = {Robust Space Trajectory Design Using Belief Optimal Control},
journal = {Journal of Guidance, Control, and Dynamics},
volume = {45},
number = {6},
pages = {1060-1077},
year = {2022},
NOTE = {doi: \url{https://doi.org/10.2514/1.G005704}},
}

@article{intro3,
author = {Z. Zhao and H. Shang and Z. Yu and J. Ren},
year = {2024},
pages = {1521--1541},
number = {8},
volume = {47},
title = {Stochastic Trajectory Planning for Autonomous Aerobraking Using Convex Optimization and Covariance Control},
journal = {Journal of Guidance, Control, and Dynamics},
NOTE = {doi: \url{https://doi.org/10.2514/1.G008030}},
}

@article{intro4,
author = {E. Bayraktar and C. W. Miller},
title = {Distribution-constrained optimal stopping},
journal = {Mathematical Finance},
volume = {29},
number = {1},
pages = {368--406},
year = {2019},
NOTE = {doi: \url{https://doi.org/10.1111/mafi.12171}},
}

@article{intro6,
  author    = {E. Bayraktar and Z. Zhou},
  title     = {On Zero-Sum Optimal Stopping Games},
  journal   = {Applied Mathematics \& Optimization},
  volume    = {78},
  number    = {3},
  pages     = {457--468},
  year      = {2018},
  NOTE       = {doi: \url{https://doi.org/10.1007/s00245-017-9412-6}},
}

@article{intro5,
title = {Stochastic control/stopping problem with expectation constraints},
journal = {Stochastic Processes and their Applications},
volume = {176},
pages = {104430},
year = {2024},
author = {E. Bayraktar and S. Yao},
NOTE = {doi: \url{https://doi.org/10.1016/j.spa.2024.104430}},
}

@article {Kiefer,
    AUTHOR = {J. Kiefer},
     TITLE = {Skorohod embedding of multivariate {RV}'s, and the sample
              {DF}},
  JOURNAL = {Zeitschrift f\"ur Wahrscheinlichkeitstheorie und Verwandte
              Gebiete},
    VOLUME = {24},
      YEAR = {1972},
    NUMBER = {1},
     PAGES = {1--35},    
       NOTE = {doi: \url{https://doi.org/10.1007/BF00532460}},
}

@article {numeric2,
    AUTHOR = {S. Becker and P. Cheridito and A. Jentzen and T. Welti},
     TITLE = {Solving high-dimensional optimal stopping problems using deep
              learning},
  JOURNAL = {European Journal of Applied Mathematics},
    VOLUME = {32},
      YEAR = {2021},
    NUMBER = {3},
     PAGES = {470--514},
       NOTE = {doi: \url{https://doi.org/10.1017/S0956792521000073}},
}

@article {numeric3,
    AUTHOR = {A. M. Reppen and H. M. Soner and V. Tissot-Daguette},
     TITLE = {Neural optimal stopping boundary},
  JOURNAL = {Mathematical Finance. An International Journal of Mathematics,
              Statistics and Financial Economics},
    VOLUME = {35},
      YEAR = {2025},
    NUMBER = {2},
     PAGES = {441--469},
       NOTE = {doi: \url{https://doi.org/10.1111/mafi.12450}},
}

@inproceedings{sch1,
author = {N. Gushchin and S. Kholkin and E. Burnaev and A. Korotin},
title = {Light and Optimal {S}chr\"odinger Bridge Matching},
year = {2024},
publisher = {JMLR.org},
booktitle = {Proceedings of the 41st International Conference on Machine Learning},
number = {680},
address = {Vienna, Austria},
pages = {17100--17122},
NOTE = {URL: \url
{https://openreview.net/forum?id=EWJn6hfZ4J}}
}

@inproceedings{sch2,
author = {J. Garg and X. Zhang and Q. Zhou},
title = {Soft-constrained {S}chr\"odinger Bridge: a Stochastic Control Approach},
year = {2024},
publisher = {PMLR},
booktitle = {Proceedings of the 27th International Conference on 
             Artificial Intelligence and Statistics (AISTATS)},
pages = {4429–4437},
volume = {238},
address = {Valencia, Spain},
NOTE = {URL: \url{https://proceedings.mlr.press/v238/garg24a.html}}
}

@article{pakniyat2022convex,
  title={A convex duality approach for assigning probability distributions to the state of nonlinear stochastic systems},
  author={A. Pakniyat},
  journal={IEEE Control Systems Letters},
  volume={6},
  pages={3080--3085},
  year={2022},
  NOTE = {doi: \url{https://doi.org/10.1109/LCSYS.2022.3181525}}
}

@article{pakniyat2025graphon,
  title={A Graphon Mean Field Convex Duality Approach to Shaping the Terminal Probability Distribution of a Network of Nonlinear Multi-Agent Systems},
  author={A. Pakniyat},
  journal={Journal of Systems Science and Complexity},
  volume={38},
  number={1},
  pages={349--368},
  year={2025},
  publisher={Springer},
  NOTE = {doi: \url{https://doi.org/10.1007/s11424-025-4507-7}}
}

@article{ref:PM:DC:JL:TAC:Motion:Planning,
  title={Motion planning for continuous-time stochastic processes: A dynamic programming approach},
  author={P. M. Esfahani and D. Chatterjee and J. Lygeros},
  journal={IEEE Transactions on Automatic Control},
  volume={61},
  number={8},
  pages={2155--2170},
  year={2015},
  NOTE = {doi: \url{https://doi.org/10.1109/TAC.2015.2500638}},
}

@article{ref:PM:DC:JL:Aut:ReachAvoid,
  title={The stochastic reach-avoid problem and set characterization for diffusions},
  author={P. M. Esfahani and D. Chatterjee and J. Lygeros},
  journal={Automatica},
  volume={70},
  pages={43--56},
  year={2016},
  publisher={Elsevier},
  NOTE = {doi: \url{https://doi.org/10.1016/j.automatica.2016.03.016}},
}

@inproceedings{ref:AP:PS:CDC:PartObs:Steering,
  title={Partially observed steering the state of linear stochastic systems},
  author={A. Pakniyat and P. Tsiotras},
  booktitle={2021 60th IEEE Conference on Decision and Control (CDC)},
  pages={3780--3785},
  year={2021},
  NOTE = {doi: \url{https://doi.org/10.1109/CDC45484.2021.9683692}}
}

@inproceedings{ref:AP:PP:ACC:MinPrnciple:Steering,
  title={Steering the state of linear stochastic systems: a constrained minimum principle formulation},
  author={A. Pakniyat and P. Tsiotras},
  booktitle={2021 American Control Conference (ACC)},
  pages={1300--1305},
  year={2021},
  NOTE = {doi: \url{https://doi.org/10.23919/ACC50511.2021.9483276}}
}

@article{ref:AE:TTG:SB-timeMarg,
  title={Schr{\"o}dinger's Control and Estimation Paradigm With Spatio-Temporal Distributions on Graphs},
  author={A. Eldesoukey and T. T. Georgiou},
  journal={IEEE Transactions on Automatic Control},
  volume={70},
  number={4},
  pages={2466--2478},
  year={2024},
  publisher={IEEE},
  NOTE = {doi: \url{https://doi.org/10.1109/TAC.2024.3485537}}
}

@article{ref:OMM:AE:TTG:ReacDiff,
  title={Maximum entropy inference of reaction--diffusion models},
  author={O. M. Miangolarra and A. Eldesoukey and A. M. Miangolarra and T. T. Georgiou},
  journal={The Journal of Chemical Physics},
  volume={162},
  number={19},
  year={2025},
  publisher={AIP Publishing},
  NOTE = {doi: \url{https://doi.org/10.1063/5.0256659}}
}

@misc{ref:AS:SG:FL:PT:FFTAut,
      title={Continuous-Time Covariance Steering with Common Free-Final Time: Finite-Horizon Solutions and Infinite-Horizon Limits}, 
      author={A. Selim and F. Liu and S. Ganguly and P. Tsiotras},
      year={2026},
      eprint={2608.10267},
      archivePrefix={arXiv},
      primaryClass={math.OC},
      NOTE ={doi: \url{https://doi.org/10.48550/arXiv.2608.10267}}, 
}

@article{ref:OMM:AE:TTG:Caliber,
  title={Inferring potential landscapes: A {S}chr{\"o}dinger bridge approach to maximum caliber},
  author={O. M. Miangolarra and A. Eldesoukey and T. T. Georgiou},
  journal={Physical Review Research},
  volume={6},
  number={3},
  pages={033070},
  year={2024},
  publisher={APS},
  NOTE = {doi: \url{https://doi.org/10.1103/PhysRevResearch.6.033070}}
}

\end{document}